\documentclass[twoside]{amsart}

\usepackage{amsmath, amsthm, amssymb, mathrsfs}
\usepackage{epsfig, graphicx}
\usepackage{verbatim}
\usepackage[colorlinks=true, citecolor=blue]{hyperref}

\newtheorem{thm}{Theorem}

\newtheorem{lem}[thm]{Lemma}
\newtheorem{prop}[thm]{Proposition}
\newtheorem{df}{Definition}

\newtheorem*{thmT1}{Theorem T1}
\newtheorem*{thmT21}{Theorem T2(1)}
\newtheorem*{thmT22}{Theorem T2(2)}
\newtheorem*{thmT3}{Theorem T3}
\newtheorem*{thmT4}{Theorem T4}

\numberwithin{equation}{section}

\newcommand{\D}		{\mathbb{D}}
\newcommand{\A}		{\mathbb{A}}
\newcommand{\R}		{\mathbb{R}}
\newcommand{\C}		{\mathbb{C}}
\newcommand{\N}		{\mathbb{N}}
\newcommand{\Z}		{\mathbb{Z}}
\newcommand{\rhy}		{\textnormal{RHP}-$\mathscr{Y}$}
\newcommand{\rht}		{\textnormal{RHP}-$\mathscr{T}$}
\newcommand{\rhn}		{{\textnormal{RHP}}-$\mathscr{N}$}
\newcommand{\rhp}		{\textnormal{RHP}-$\mathscr{P}$}
\newcommand{\rhpt}	{\textnormal{RHP}-$\mathscr{\widetilde P}$}
\newcommand{\rhpsi}	{\textnormal{RHP}-$\Psi$}
\newcommand{\rha}		{\textnormal{RHP}-$\mathscr{A}$}
\newcommand{\rhb}		{\textnormal{RHP}-$\mathscr{B}$}
\newcommand{\rhr}		{\textnormal{RHP}-$\mathscr{R}$}
\newcommand{\rhds}	{\textnormal{RH$\bar\partial$P}-$\mathscr{S}$}
\newcommand{\dbd}	{$\bar\partial$\textnormal{P}-$\mathscr{D}$}

\newcommand{\oi}	{\mathcal{I}}
\newcommand{\ob}	{\mathcal{B}}
\newcommand{\oc}	{\mathcal{C}}
\newcommand{\os}	{\mathcal{S}}

\newcommand{\ok}	{\mathcal{K}}
\newcommand{\gm}	{G}
\newcommand{\szf}	{S}

\newcommand{\ed}	{\omega}
\newcommand{\hf}	{\mathrm{H}}
\newcommand{\vp}	{\mathrm{S}}
\newcommand{\cf}	{\mathrm{C}}
\newcommand{\lf}	{\mathrm{L}}
\newcommand{\sof}	{\mathrm{W}}
\newcommand{\map}	{\varphi}
\newcommand{\jt}	{J}
\newcommand{\sr}		{\mathfrak{w}}
\newcommand{\Arg}		{\textnormal{Arg}}

\newcommand{\const}	{\textnormal{const.}}
\newcommand{\supp}	{\textnormal{supp}}

\newcommand{\im}		{\textnormal{Im}}
\newcommand{\re}		{\textnormal{Re}}

\begin{document}

\title[Convergent Interpolation and Jacobi-Type Weights]{Convergent Interpolation to Cauchy Integrals over Analytic Arcs with Jacobi-Type Weights}

\author[L. Baratchart]{Laurent Baratchart}

\address{INRIA, Project APICS \\
2004 route des Lucioles --- BP 93 \\
06902 Sophia-Antipolis, France}

\email{laurent.baratchart@sophia.inria.fr}

\author[M. Yattselev]{Maxim Yattselev}

\address{Corresponding author \\
Center for Constructive Approximation \\
Department of Mathematics, Vanderbilt University \\
Nashville, TN, 37240, USA}

\email{maxim.yattselev@vanderbilt.edu}

\begin{abstract}
We design convergent multipoint Pad\'e interpolation schemes to Cauchy transforms of non-vanishing complex densities with respect to Jacobi-type weights on analytic arcs, under mild smoothness assumptions on the density. We rely on the work \cite{uBY3} for the choice of the interpolation points, and dwell on the Riemann-Hilbert approach to asymptotics of orthogonal polynomials introduced in \cite{KMcLVAV04} in the case of a segment. We also elaborate on the $\bar\partial$-extension of the Riemann-Hilbert technique, initiated in \cite{McLM08} on the line to relax analyticity assumptions. This yields strong asymptotics for the denominator polynomials of the multipoint Pad\'e interpolants, from which convergence follows.
\end{abstract}



\maketitle

\section{Introduction}

Classical Pad\'e approximants (or interpolants) and their multipoint generalization are probably the oldest and simplest candidate rational-approximants to a holomorphic function of one complex variable. They are simply those rational functions of type\footnote{A rational function is said to be of type $(m,n)$ if it can be written as the ratio of a polynomial of degree at most $m$ and a polynomial of degree at most $n$.} $(m,n)$ that interpolate the function in $m+n+1$ points of the domain of analyticity, counting multiplicity. \emph {Classical} Pad\'e approximants refer to the case where interpolation takes place in a single point with multiplicity $m+n+1$ \cite{Pade92}.

Besides their everlasting number-theoretic success \cite{Siegel, KratRiv08, RivZud03}, they are common tools in modeling and numerical analysis of various fields, ranging from boundary value problems and convergence acceleration \cite{But64, Gr65, BrezinskiRedivoZaglia,GonNovHen91,DruMos01} to continuous mechanics \cite{AvFaRe07,TeToga01}, quantum mechanics \cite{Baker, Tj_PRA77}, condensed matter physics \cite{SarkarBhatt92}, fluid mechanics \cite{Pozzi}, system and circuits theory \cite{Brezinski,Horiguchi,CelOcTanAt95}, and even page ranking the {\it Web} \cite{BrezR-Z06}.

In spite of this, the convergence properties of Pad\'e or multipoint Pad\'e approximants are still far from being understood. For particular classes of functions like Markov functions, some elliptic functions, and certain entire functions such as Polya frequencies or functions with smooth and fast decaying Taylor coefficients, classical Pad\'e approximants at infinity are known to converge, locally uniformly on the domain of analyticity \cite{Mar95,Suet00,ArmEd70,Lub85}. But when applied to more general cases they seldom accomplish the same, due to the occurrence of ``spurious poles'' that may wander about the domain of analyticity. Further distinction should be made here between diagonal approximants ({\it i.e.} interpolants of type $(m,m)$) and row approximants ({\it i.e.} interpolants of type $(m,n)$ where $n$ is kept fixed), and we refer the reader to the comprehensive monograph \cite{BakerGravesMorris} for a detailed account of many works on the subject. Let us simply mention that, for the case of diagonal approximants which is the most interesting as it treats poles and zeros on equal footing, the disproof of the Pad\'e conjecture \cite{Lub03} and of the Stahl conjecture \cite{Bus02} have only added to the picture that classical Pad\'e approximants are not seen best through the spectacles of uniform convergence.

The case of multipoint Pad\'e approximants is somewhat different, since choosing the interpolation points offers new possibilities to help convergence. However, it is not immediately clear how to use these additional parameters. The theory was initially developed for Markov functions ({\it i.e.} Cauchy transforms of positive measures compactly supported on the real line) showing that multipoint Pad\'e approximants converge locally uniformly on the complement of the smallest segment containing the support of the defining measure, provided the interpolation points are conjugate symmetric \cite{GL78}. The crux of the proof is the remarkable connection between rational interpolants and orthogonal polynomials: the denominator of the $n$-th diagonal multipoint  Pad\'e approximant is the $n$-th orthogonal polynomial of the measure defining the Markov function, weighted by the inverse of the polynomial whose zeros are the interpolation points (this polynomial is identically 1 for classical Pad\'e approximants). The conjugate symmetric distribution of the interpolation points is to the effect that the weight is positive, so one can apply the asymptotic theory of orthogonal polynomials with varying weights \cite{StahlTotik}. 

When trying to generalize this approach to more general Cauchy integrals than Markov functions, one is led to consider non-Hermitian orthogonal polynomials with respect to complex-valued measures on more general arcs than segments, and for a while it was unclear what could be hoped for. In the pathbreaking papers \cite{St85, St86, St89, St97}, devoted to the convergence \emph{in capacity} of classical Pad\'e approximants to functions with branchpoints, it was shown that such orthogonal polynomials lend themselves to analysis when the measure is supported on a system of arcs of minimal logarithmic capacity linking the branchpoints, in the complement of which the function is single-valued. Shortly after, the same type of convergence was established for multipoint Pad\'e approximants to Cauchy integrals of continuous (quasi-everywhere) non-vanishing densities over arcs of minimal {\em weighted capacity}, provided that the interpolation points asymptotically distribute like a measure whose potential is the logarithm of the weight \cite{GRakh87}. Such an extremal system of arcs is called  a {\it symmetric contour}, or {\em $S$-contour}, and is characterized by a symmetry property of the (two-sided) normal derivatives of its equilibrium potential. The corresponding condition on the distribution of the interpolation points may be viewed as a far-reaching generalization of the conjugate-symmetry with respect to the real line that was required to interpolate Markov functions in a convergent way.

After these works it became apparent that the appropriate class of Cauchy integrals for Pad\'e approximation should consist of those taken over $S$-contours, and that the interpolation points should distribute according to the weight that defines the symmetry property. However, it is not so easy to decide which systems of arcs are $S$-contours, since finding a weight making the arcs of smallest weighted capacity is a nontrivial inverse problem, and in any case convergence in capacity is much weaker than locally uniform convergence.

For the class of Jordan arcs, new ground  was recently broken in \cite{uBY3} where it is shown that such an arc, if rectifiable and Ahlfors regular at the endpoints, is an $S$-contour if and only if it is analytic. The proof recasts the $S$-property for Jordan arcs as the existence of a sequence of ``pseudo-rational'' functions, holomorphic and tending to zero off the arc, whose boundary values from each side of the latter remain bounded, and whose zeros remain at positive distance from the arc. There are in fact many such sequences that can be computed explicitly from an analytic parameterization of the arc. Then, translating the non-Hermitian orthogonality equation for the denominator into an integral equation involving Hankel operators and using compactness properties of the latter, the reference just quoted establishes that multipoint Pad\'e approximants to Cauchy transforms of Dini-continuous (essentially) non-vanishing densities with respect to the equilibrium distribution of the arc converge locally uniformly in its complement when the interpolation points are the zeros of these pseudo-rational functions.

Still the above result remains unsatisfactory, for the hypotheses entail that the density with respect to arclength in the integral goes to infinity towards the endpoints of the arc, since so does the equilibrium distribution. In particular, ultra-smooth situations like the one of Cauchy integrals of smooth functions over analytic arcs are not covered. The present paper develops a new technique to handle \emph{any} non-vanishing integrable Jacobi-type density under mild smoothness assumptions, thereby settling more or less  the issue of convergence in multipoint Pad\'e interpolation to functions defined as Cauchy integrals over analytic Jordan arcs.

We dwell on the Riemann-Hilbert approach to asymptotics of orthogonal polynomials with analytic weights, pioneered on the line in \cite{BlIt99,DKMLVZ99a} and carried over to the segment in \cite{KMcLVAV04}. We also elaborate on the $\bar\partial$-extension thereof, initiated on the line in \cite{McLM08} to relax the analyticity requirement. This will provide us with strong ({\it i.e.} Plancherel-Rotach type) asymptotics for the denominator polynomials of the multipoint Pad\'e interpolants we construct and for their associated functions of the second kind, from which the local uniform convergence we seek follows easily. It is interesting to note that the Riemann-Hilbert approach, which is typically a tool to obtain sharp quantitative asymptotics, is here used as a means to solve a qualitative question namely the convergence of the interpolants.

The interpolation points shall be the same as in \cite{uBY3}, namely the zeros of a sequence of pseudo-rational functions adapted to the arc. Such an interpolation scheme will prove convergent for all Cauchy integrals with sufficiently smooth density with respect to a Jacobi weight on the arc at the same time. This provides us with a varying weight  which is not of power type, nor in general converging sufficiently fast to a weight of power type to take advantage of the results of \cite{Ap02}, where the Riemann-Hilbert approach is adapted to non-Hermitian orthogonality  with analytic weights on smooth $S$-arcs. Instead, when ``opening the lens'', we set up a sequence of Riemann-Hilbert problems with \emph{varying contours} whose solutions converge to the desired one by properties of the pseudo rational functions. 

We pay special attention to keep smoothness requirements low, in order to obtain as general a result as the method permits. Roughly speaking, the higher the Jacobi exponents the smoother the density should be, see the precise assumptions \eqref{eq:ab}. When the Jacobi exponents are negative, only a fraction of a derivative is needed, which compares not too badly with the Dini-continuity assumption in \cite{uBY3}. In the present setting, however, the density cannot vanish whereas some weak vanishing is still allowed in \cite{uBY3}. We are of course rewarded here with stronger asymptotics.

As the varying part of our weight is analytic, the extension inside the lens with controlled $\bar\partial$-estimates, introduced in \cite{McLM08} for power weights, needs only deal with the density defining the Cauchy integral we interpolate. This step is treated using either tools from real analysis, {\it e.g.} Muckenhoupt weights and Sobolev traces, or else classical H\"older estimates for singular integrals, whichever yields the best results granted  the Jacobi exponents.

Since we consider analytic arcs only, it is natural to ask how general our results with respect to the general class of Cauchy integrals over rectifiable Jordan arcs. It turns out that they are as general as can be, because if the Cauchy integral of a nontrivial Jacobi weight can be interpolated in a convergent way with a triangular scheme of interpolation points that stay away from the arc, then the arc is in fact \emph{analytic} \cite{uBY1}.

The paper is organized as follows. Section~\ref{sec:st} fixes notation and defines   pseudo rational functions as well as multipoint Pad\'e approximants before stating the main results. In Section~\ref{sec:g}, the contours that will later be instrumental for the solution of the Riemann-Hilbert problem are introduced. Section~\ref{sec:extension} contains preliminaries on smooth extensions from boundary data in domains with polygonal boundaries. Section~\ref{sec:bvp} is devoted to key estimates of certain singular integral operators that play a  main role in the extension of the weight. Section~\ref{sec:rhpbp} and \ref{sec:aa} deal with the analytic Riemann-Hilbert problem, while Section~\ref{sec:pbp} solves the $\overline{\partial}$ version thereof. Finally, in Section~\ref{sec:solution}, we gather the material developed so far to establish the asymptotics and the convergence of multipoint Pad\'e approximants stated in Section~\ref{sec:st}.

\section{Statements of Results}
\label{sec:st}

Let $\Delta$ be a closed analytic Jordan \emph{arc} with endpoints $-1$ and $1$. That is to say, there exists a holomorphic univalent function $\Xi$, defined in some domain $D_\Xi\supset[-1,1]$, such that
\[
\Delta = \Xi([-1,1]), \quad \Xi(\pm1)=\pm1.
\]
We call $\Xi$ an analytic parameterization of $\Delta$. We orient $\Delta$ from $-1$ to $1$ and, according to this orientation, we distinguish the left and the right sides of $\Delta$ denoted by $\Delta^+$ and $\Delta^-$, respectively. It will be convenient to introduce two unbounded arcs, say, $\Delta_l$ and $\Delta_r$, that respectively connect $-\infty$ to $-1$ and 1 to $+\infty$, in such a manner that $\Delta_l\cup\Delta\cup\Delta_r$ is a smooth unbounded Jordan arc that coincides with the real line in some neighborhood of infinity. Define on $\Delta$ the Jacobi weight
\begin{equation}
\label{eq:wab}
w(z) = w(\alpha,\beta;z) :=(1-z)^\alpha(1+z)^\beta, \quad \alpha,\beta>-1,
\end{equation}
where we choose branches of $(1-z)^\alpha$ and $(1+z)^\beta$ that are holomorphic outside of $\Delta_r$ and $\Delta_l$, respectively, and assume value 1 at the origin. In particular, $w$ is analytic across $\Delta^\circ:=\Delta\setminus\{\pm1\}$. Further, set 
\begin{equation}
\label{eq:sr}
\sr(z):=\sqrt{z^2-1}, \quad \sr(z)/z\to1, \quad \mbox{as} \quad z\to\infty,
\end{equation}
to be a holomorphic branch of the square root outside of $\Delta$. Then 
\begin{equation}
\label{eq:map}
\map(z):=z+\sr(z), \quad z\in D:=\overline\C\setminus\Delta,
\end{equation}
is holomorphic in $D\setminus\{\infty\}$, has continuous boundary values $\map^\pm$ on $\Delta^\pm$, respectively, and satisfies
\begin{equation}
\label{eq:mappr}
\map^+\map^-=1 \quad \mbox{on} \quad \Delta \quad \mbox{and} \quad \map(z)/2z \to 1 \quad \mbox{as} \quad z\to\infty.
\end{equation}
It is immediate that $\map$ is inverse of the Joukovski transformation $\jt(z):=(z^2+1)/2z$, i.e., $\jt(\map(z))=z$, $z\in D$. Moreover, $\map$ maps $D$ conformally onto an unbounded domain whose  boundary is an analytic Jordan curve \cite[Sec. 3.1]{uBY3} which is symmetric with respect to the transformation $z\mapsto1/z$. In particular, $\map$ does not vanish.

\subsection{Symmetric Contours}

Multipoint Pad\'e approximants to a given function $f$ are defined to be rational interpolants to $f$. In this paper we are interested in those functions $f$ that can be expressed as Cauchy integrals of Jacobi-type complex densities defined on $\Delta$ (see the smoothness assumptions in \eqref{eq:measure} and \eqref{eq:ab}). In order for multipoint Pad\'e approximants to converge to such a function, it is necessary to choose the interpolation schemes appropriately with respect to $\Delta$. We presently characterize these schemes in terms of the associated monic polynomials vanishing at the interpolation points.

Let $\{v_n\}$ be a sequence polynomials such that $\deg(v_n)\le2n$ and each $v_n$ has no zeros on $\Delta$. To this sequence we associate a sequence of ``pseudo-rational'' functions, say $\{r_n\}$, given by
\begin{equation}
\label{eq:rn}
r_n(z) := \left(\frac{1}{\map(z)}\right)^{2n-\deg(v_n)}\prod_{\{e:v_n(e)=0\}}\frac{\map(z)-\map(e)}{1-\map(e)\map(z)}, \quad z\in D,
\end{equation}
where the product is taken over all zeros of $v_n$ according to their multiplicities. It is easy to see that each function $r_n$ is holomorphic in $D$, has the same zeros as $v_n$ counting multiplicities, and vanishes at infinity with order $2n-\deg(v_n)$. Hence, each $r_n$ has exactly 2n zeros counting multiplicities. Moreover, the unrestricted boundary values $r^\pm_n$ exist continuously from each side of $\Delta$ and satisfy $r^+_nr^-_n \equiv 1$ by the first part of \eqref{eq:mappr}.

Hereafter, the normalized counting measure of a finite set is the probability measure that has equal mass at each point counting multiplicities. Below, the weak$^*$ topology refers to the duality between complex measures and continuous functions with compact support in~$\overline\C$. 

\begin{df}
\label{df:S}
We say that a sequence of polynomials $\{v_n\}$ with no zeros on $\Delta$ belongs to the class $\vp(\Delta)$ if the following conditions hold:
\begin{itemize}
\item[(1)] the associated functions $r_n$ via \eqref{eq:rn} satisfy $|r_n^\pm| = O(1)$ uniformly on $\Delta$ and $r_n=o(1)$ locally uniformly in $D$;
\item[(2)] there exists a neighborhood of $\Delta$ that contains no zeros of $r_n$ for all $n$ large enough;
\item[(3)] the normalized counting measures of zeros of $r_n$ form a weak$^*$ convergent sequence.
\end{itemize}
\end{df}

The third requirement in the definition of $\vp(\Delta)$ is purely technical and is placed only to simplify the forthcoming considerations since one can always proceed with subsequences as far as convergence is concerned.

Regarding the nature of the class $\vp(\Delta)$, the following result was obtained in \cite[Thm. 1]{uBY3}. For a closed analytic Jordan arc $\Delta$, there always exist sequences $\{v_n\}$ belonging to $\vp(\Delta)$ and they can be constructed explicitly granted the parameterization $\Xi$. A partial converse is also true. Namely, let $\Delta$ be a rectifiable Jordan arc with endpoint $\pm1$ such that for $x=\pm1$ and all $t\in\Delta$ sufficiently close to $x$ it holds that $|\Delta_{t,x}|\leq\const|x-t|^\beta$, $\beta>1/2$, where $|\Delta_{t,x}|$ is the length of the subarc of $\Delta$ joining $t$ and $x$ and ``$\const$'' is an absolute constant. If there exists a sequence of polynomials $\{v_n\}$ meeting the first two requirements of Definition~\ref{df:S}, then $\Delta$ is necessarily analytic. The class $\vp(\Delta)$ is also intimately related to the so-called \emph{symmetry property} of the contour $\Delta$ \cite{St85,St85b,uBY3}. 

For our investigation we need to detail further the properties of the just defined interpolation schemes. We gather them in the following theorem. We agree that the arcs involved have endpoints $\pm1$. Moreover, we say that two holomorphic functions are analytic continuations of each other if they are defined on domains that have nonempty intersection on which the functions coincide.

\begin{thm}
\label{thm:sp} Let $\Delta$ be a closed analytic Jordan arc and $\{v_n\}\in\vp(\Delta)$. Then there exists a sequence of closed analytic Jordan arcs $\{\Delta_n\}$ such that:
\begin{itemize}
\item [(i)] there exist analytic parametrizations $\Xi_n$ of $\Delta_n$ and $\Xi$ of $\Delta$ such that the functions $\Xi_n$ converge to $\Xi$ uniformly in some neighborhood of $[-1,1]$ as $n\to\infty$;
\item [(ii)] for each function $r_n$, associated to $v_n$ via \eqref{eq:rn}, there exists an analytic continuation $r_n^*$, holomorphic in $D_n:=\overline\C\setminus\Delta_n$, such that $|(r_n^*)^\pm|\equiv1$ on $\Delta_n$.
\end{itemize} 
\end{thm}

Let $\sr_n$ and $\map_n$ be defined relative to $\Delta_n$ as $\sr$ and $\map$ were defined in \eqref{eq:sr} and \eqref{eq:map} relative to $\Delta$. Clearly, $\sr_n$ and $\map_n$ are analytic continuations of $\sr$ and $\map$ to $D_n$. In fact, $r_n^*$ is simply the function associated to $v_n$ via \eqref{eq:rn} with $\map$ replaced by $\map_n$. It is apparent that $r_n^*$ is nothing  but the Blaschke product with respect to $D_n$ that has the same zeros as $r_n$.

\subsection{Multipoint Pad\'e Approximation}

Let $\mu$ be a complex Borel measure with compact support. We define the Cauchy transform of $\mu$ as
\begin{equation}
\label{eq:CauchyT}
f_\mu(z) := \int\frac{d\mu(t)}{z-t}, \quad z\in\overline\C\setminus\supp(\mu).
\end{equation}
Clearly, $f_\mu$ is a holomorphic function in $\overline\C\setminus\supp(\mu)$ that vanishes at infinity.

Classically, diagonal (multipoint) Pad\'e approximants to $f_\mu$ are rational functions of type $(n,n)$ that interpolate $f_\mu$ at a prescribed system of $2n+1$ points. However, when the approximated function is of the from (\ref{eq:CauchyT}), it is customary to place at least one interpolation point at infinity so as to let the approximants vanish at infinity as well by construction.

\begin{df}
\label{df:pade}
Let $f_\mu$ be given by \eqref{eq:CauchyT} and $\{v_n\}$ be a sequence of monic polynomials, $\deg(v_n)\leq2n$, with zeros in $\overline\C\setminus\supp(\mu)$. The $n$-th diagonal Pad\'e approximant to $f_\mu$ associated with $\{v_n\}$ is the unique rational function $\Pi_n=p_n/q_n$ satisfying:
\begin{itemize}
\item $\deg p_n\leq n$, $\deg q_n\leq n$, and $q_n\not\equiv0$; \smallskip
\item $\left(q_n(z)f_\mu(z)-p_n(z)\right)/v_n(z)$ is analytic in $\overline\C\setminus\supp(\mu)$; \smallskip
\item $\left(q_n(z)f_\mu(z)-p_n(z)\right)/v_n(z)=O\left(1/z^{n+1}\right)$ as $z\to\infty$.
\end{itemize}
\end{df}

A multipoint Pad\'e approximant always exists since the conditions for $p_n$ and $q_n$ amount to solving a system of $2n+1$ homogeneous linear equations with $2n+2$ unknown coefficients, no solution of which can be such that $q_n\equiv0$ (we may thus assume that $q_n$ is monic); note that the required interpolation at infinity is entailed by the last condition and therefore $\Pi_n$ is, in fact, of type $(n-1,n)$.

We consider only absolutely continuous measures that are supported on $\Delta$ and whose densities are Jacobi weights \eqref{eq:wab} multiplied by suitably smooth non-vanishing functions. This leads us to define smoothness classes $\cf^{m,\varsigma}$.

\begin{df}
\label{df:holder}
Let $K$ be an infinitely smooth closed Jordan arc or curve. We say that $\theta\in\cf^{m,\varsigma}(K)$ if $\theta$ is $m$-times continuously differentiable on $K$ with respect to the arclength and its $m$-th derivative is uniformly H\"older continuous with exponent $\varsigma$, i.e.,
\[
|\theta^{(m)}(t_1)-\theta^{(m)}(t_2)| \leq \const|t_1-t_2|^{\varsigma}, \quad t_1,t_2\in K.
\]
When $K=\Delta$, we simply write $\cf^{m,\varsigma}$ instead of $\cf^{m,\varsigma}(\Delta)$. We also write $\cf^\infty(K)$ for the space of infinitely differentiable functions on $K$.
\end{df}

Together with $\cf^{m,\varsigma}$, we also consider fractional Sobolev spaces.

\begin{df}
\label{df:sobolev}
Let $K$ be an infinitely smooth Jordan arc or curve. We say that $\theta\in\sof^{1-1/p}_p(K)$, $p\in(1,\infty)$, if
\[
\iint_{K\times K}\left|\frac{\theta(x)-\theta(y)}{x-y}\right|^p|dx||dy|<\infty.
\]
When $K=\Delta$, we simply write $\sof^{1-1/p}_p$ instead of $\sof^{1-1/p}_p(K)$.
\end{df}

We shall be interested only in the case $p\in(2,\infty)$ since in this range it holds that
\begin{equation}
\label{eq:simb}
\sof_p^{1-1/p} \subset \cf^{0,\varsigma}, \quad \varsigma = 1-\frac2p, \quad p\in(2,\infty).
\end{equation}
by Sobolev imbedding theorem (see  Section \ref{ss:sb}).

In what follows, we assume that the measure $\mu$ in \eqref{eq:CauchyT} is of the form
\begin{equation}
\label{eq:measure}
d\mu(t) = (wh)(t)dt, \quad h(t)=e^{\theta(t)}, \quad t\in\Delta,
\end{equation}
where the Jacobi weight $w=w(\alpha,\beta;\cdot)$ and the complex function $\theta$ are such that
\begin{equation}
\label{eq:ab}
\alpha,\beta \in (-s,s)\cap(-1,\infty)
\end{equation}
with
\[
s: =\left\{\begin{array}{lll}
1-\frac2p, & \mbox{if} \quad \theta\in\sof_p^{1-1/p}, & p\in(2,\infty) , \smallskip \\
2\varsigma-1, & \mbox{if} \quad \theta\in\cf^{0,\varsigma}, & \varsigma\in\left(\frac12,1\right], \smallskip \\
m+\varsigma, & \mbox{if} \quad \theta\in\cf^{m,\varsigma}, & m\in\N, \quad \varsigma\in(0,1].
\end{array}
\right.
\]
 
To describe the asymptotic behavior of the approximation error to functions $f_\mu$ by the multipoint Pad\'e approximants, we need to introduce complex geometric means and Szeg\H{o} functions. The {\it geometric mean} of $h=e^\theta$ is given by
\begin{equation}
\label{eq:geommean}
\gm_h := \exp\left\{\int\theta d\ed\right\}, \quad d\ed(t) := {\frac{idt}{\pi\sr^+(t)}}, \quad t\in\Delta.
\end{equation}
The measure $\ed$ is, in some sense, natural for the considered problem as suggested by the forthcoming Theorem~\ref{thm:pade}. Observe also that $\ed$ simply becomes the normalized arcsine distribution on $\Delta$ when $\Delta=[-1,1]$.  As $\int d\ed=1$, $\gm_h$ depends only on $h$ and is non-zero when  $\theta$ is H\"older continuous (see Section \ref{subsec:szego}). Moreover, in this case the \emph{Szeg\H{o} function} of $h$, given by
\begin{equation}
\label{eq:szf}
\szf_h(z) := \exp\left\{\frac{\sr(z)}{2}\int\frac{\theta(t)}{z-t}d\ed(t)-\frac12\int\theta d\ed\right\}, \quad z\in D,
\end{equation}
is the unique non-vanishing holomorphic function in $D$ that has continuous unrestricted boundary values on $\Delta$ from each side and satisfies
\begin{equation}
\label{eq:szego}
h = \gm_h\szf_h^+\szf_h^- \quad \mbox{on} \quad \Delta \quad \mbox{and} \quad \szf_h(\infty)=1.
\end{equation}
The main result of the paper is the following theorem.

\begin{thm}
\label{thm:pade}
Let $\Delta$ be a closed analytic Jordan arc connecting $\pm1$ and $\{v_n\}\in\vp(\Delta)$. Let also $f_\mu$ a Cauchy integral \eqref{eq:CauchyT} with $\mu$ given by \eqref{eq:measure} and \eqref{eq:ab}. Then $\{\Pi_n\}$, the sequence of diagonal multipoint Pad\'e approximants to $f_\mu$ associated with $\{v_n\}$, is such that
\[
(f_\mu-\Pi_n)\sr = \left[2\gm_{\dot\mu} + o(1)\right]\szf^2_{\dot\mu}\;r_n,
\]
with $o(1)$ satisfying
\begin{equation}
\label{eq:deltane}
|o(1)| \leq \frac{\const}{n^a}, \quad
a\in\left\{
\begin{array}{ll}
\left(0,\frac{s-\max\{|\alpha|,|\beta|\}}{2}\right), & s-\max\{|\alpha|,|\beta|\}\leq 1, \smallskip \\
\left(0,\frac12\right), & s-\max\{|\alpha|,|\beta|\}>1,
\end{array}
\right. 
\end{equation}
locally uniformly in $D$, where the constant $\const$ depends on $a$, $d\mu=\dot\mu d\ed$, i.e., $\dot\mu=-i\pi wh\sr^+$, and the functions $r_n$ are associated to the polynomials $v_n$ via \eqref{eq:rn} and hence converge to zero geometrically fast in $D$.
\end{thm}

The convergence theory of Pad\'e approximants to Cauchy integrals is strongly interwoven with asymptotic behavior of underlying orthogonal polynomials that are the denominators of $\Pi_n$. In fact, it is easy to show that $f_\mu-\Pi_n=R_n/q_n$, where $R_n$ is a function associated to $q_n$ via \eqref{eq:skind}, and $q_n$ satisfies the orthogonality relations of the form \eqref{eq:ortho}, \eqref{eq:varyingweight}  (see, for example, \cite[Thm.~4]{uBY3}). Hence, Theorem~\ref{thm:pade} follows from Theorem~\ref{thm:sa} below.

\subsection{Strong Asymptotics for non-Hermitian Orthogonal Polynomials}

In this section we investigate the asymptotic behavior of polynomials satisfying non-Hermitian orthogonality relations of the form
\begin{equation}
\label{eq:ortho}
\int_\Delta t^jq_n(t)w_n(t)dt=0, \quad j\in\{0,\ldots,n-1\},
\end{equation}
together with the asymptotic behavior of their functions of the second kind, i.e.,
\begin{equation}
\label{eq:skind}
R_n(z) := R_n(q_n;z) = \int_\Delta\frac{q_n(t)w_n(t)}{t-z}\frac{dt}{\pi i}, \quad z\in D,
\end{equation}
where $\{w_n\}$ is the sequence of varying weights specified in \eqref{eq:varyingweight}.

\begin{thm}
\label{thm:sa}
Let $\{q_n\}$, $\deg(q_n)\leq n$, be a sequence of polynomials satisfying orthogonality relations \eqref{eq:ortho} with weights given by
\begin{equation}
\label{eq:varyingweight}
w_n:=\frac{wh_nh}{v_n}, \quad h=e^{\theta}, \quad h_n=e^{\theta_n},
\end{equation}
where $\theta$ and $w=w(\alpha,\beta;\cdot)$ are as in \eqref{eq:ab}, $\{\theta_n\}$ is a normal family in some neighborhood of $\Delta$, and $\{v_n\}\in \vp(\Delta)$. Then, for all $n$ large enough, the polynomials $q_n$ have exact degree $n$ and therefore can be normalized to be monic. Under such a normalization, we have that
\begin{equation}
\label{eq:sa1}
\left\{
\begin{array}{lll}
q_n  &=& [1+o(1)]/\szf_n \smallskip \\
R_n\sr &=& [1+o(1)]\gamma_n\szf_n
\end{array}
\right.
\end{equation}
with $o(1)$ satsfying \eqref{eq:deltane} locally uniformly in $D$, where 
\begin{equation}
\label{eq:sngamman}
\szf_n := (2/\map)^n\szf_{w_n\sr^+}, \quad \gamma_n:=2^{1-2n}\gm_{w_n\sr^+}, 
\end{equation}
and $R_n$ was defined in \eqref{eq:skind}. Moreover, it holds that
\begin{equation}
\label{eq:sa2}
\left\{
\begin{array}{lll}
q_n &=&  [1+o(1)]/\szf_n^+ + [1+o(1)]/\szf_n^- \smallskip \\
(R_n\sr)^\pm &=& [1+o(1)]~\gamma_n\szf_n^\pm
\end{array}
\right.,
\end{equation}
where $o(1)$ satisfies \eqref{eq:deltane} locally uniformly in $\Delta^\circ$.
\end{thm}

The method of proof can also be used to derive the asymptotics of $q_n$ and $R_n$ around $\pm1$ as was done in \cite{KMcLVAV04}. However, the corresponding calculations are lengthy and do not impinge on the convergence of Pad\'e interpolants proper, which is why the authors decided to omit them here.

The appearance of the normal family $\{\theta_n\}$ in \eqref{eq:varyingweight} is not necessitated by Theorem~\ref{thm:pade} but is included for possible application to meromorphic approximation \cite{Y10}.

\section{Proof of Theorem \ref{thm:sp} and $g$-Functions}
\label{sec:g}

In this section we prove Theorem \ref{thm:sp}. The notion of $g$-function, which we introduce along the way, will be needed later on for the proof of Theorem \ref{thm:sa}.

\subsection{Parameterization $\Xi$, functions $g$ and $\widetilde g$}
\label{subsec:g}

Let $\{v_n\}\in\vp(\Delta)$ and $r_n$ be associated to $v_n$ by \eqref{eq:rn}. As required by Definition \ref{df:S}-(2) and (3), the normalized counting measures of the zeros of $r_n$ converge weak$^*$ to a Borel measure $\nu$, $\supp(\nu)\subset D$. Denote by $V_D^\nu$ the \emph{Green potential} of this measure with respect to $D$. It was shown in the course of the proof of \cite[Thm. 1, see (4.34)]{uBY3} that Definition~\ref{df:S}-(1) yields
\begin{equation}
\label{eq:vd}
V_D^\nu(z) = -\int\log\left|\frac{\map(z)-\map(t)}{1-\map(z)\map(t)}\right|d\nu(t).
\end{equation}
In other words, the Green kernel $-\log\left|\frac{\phi(z)-\phi(t)}{1-\overline{\phi(t)}\phi(z)}\right|$, where $\phi$ is the conformal map of $D$ onto $\{|z|>1\}$ such that $\phi(\infty)=\infty$ and $\phi^\prime(\infty)>0$, can be replaced by the one in \eqref{eq:vd} for this special measure $\nu$.

The Green potential $V_D^\nu$ is a positive harmonic function in $D\setminus\supp(\nu)$ whose boundary values vanish everywhere on $\Delta$. Let $L_\rho:=\{z:~V_D^\nu(z)=\log\rho\}$, $\rho>1$, be a level line of $V_D^\nu$ in $\Xi(D_\Xi)$, the range of $\Xi$. Without loss of generality we may assume that $\rho$ is a regular value and therefore $L_\rho$ is a smooth Jordan curve encompassing $\Delta=L_1$. Denote by $O$ the domain bounded by $L_\rho$ and $\Delta$. Set
\begin{equation}
\label{eq:zbarzpartials}
\partial f := \frac12\left(\partial_x f - i\partial_y f \right) \quad \mbox{and} \quad \bar\partial f := \frac12\left(\partial_x f + i\partial_y f \right).
\end{equation}
Since $\nu$ is a probability measure, it can be verified as in the proof of \cite[Thm. 1, see (4.39) and after]{uBY3} that the function
\begin{equation}
\label{eq:gpnu}
\Phi(z) := \exp\left\{2\int_1^z\frac{\partial V_D^\nu}{\partial z}(t)dt\right\} = \exp\left\{-\int\log\frac{\map(z)-\map(t)}{1-\map(z)\map(t)}d\nu(t)\right\}
\end{equation}
is well-defined in $O$ and maps it conformally onto the annulus $\{z:~1<|z|<\rho\}$ while $\Phi(\pm1)=\pm1$, where we take any path from $1$ to $z$ contained in $O\setminus\Delta$. Moreover, by direct examination of the kernel in \eqref{eq:gpnu}, we get that
\begin{equation}
\label{eq:reciprocityPhi}
\Phi^+=\overline{\Phi^-}=1/\Phi^- \quad \mbox{on} \quad \Delta.
\end{equation}
This, in particular, yields that $\jt\circ\Phi$ is holomorphic across $\Delta$, where $\jt(z)=(z+1/z)/2$ is the Joukovski transformation. Consequently, the inverse $(\jt\circ\Phi)^{-1}$ is a holomorphic univalent map in some neighborhood of $[-1,1]$ that analytically parametrizes $\Delta$. In what follows, we assume that $\Xi=(\jt\circ\Phi)^{-1}$.
  
Based on the conformal map $\Phi$, we define two more functions, $g$ and $\widetilde g$ as follows.  Set $L:=\Phi^{-1}([-\rho,-1])$, $\widetilde L:=\Phi^{-1}([1,\rho])$ (see Fig.~\ref{fig:1}), and define
\begin{equation}
\label{eq:gfun}
\begin{array}{lll}
g := \log\Phi, & \displaystyle \lim_{z\to1}g(z) = 0, & g\in\hf(O\setminus L), \smallskip \\
\widetilde g := \log\Phi-\pi i, & \displaystyle \lim_{z\to-1}\widetilde g(z) = 0, & \widetilde g\in\hf(O\setminus \widetilde L).
\end{array}
\end{equation}
It follows immediately from \eqref{eq:reciprocityPhi} that
\begin{equation}
\label{eq:bvg}
g^+ = -g^- \quad \mbox{and} \quad \widetilde g^+ = -\widetilde g^- \quad \mbox{on} \quad \Delta.
\end{equation}
Hence, $g^2$ and $\widetilde g^2$ are analytic in $O_g:=(O\cup\Delta)\setminus L$ and $O_{\widetilde g}:=(O\cup\Delta)\setminus \widetilde L$, respectively. Moreover, it holds that $g^2(\Delta)= \widetilde g^2(\Delta) = [-\pi^2,0]$ and $g^2(1)=\widetilde g^2(-1)=0$. It is also true that $g^2$ and $\widetilde g^2$ are univalent in $O_g$ and $O_{\widetilde g}$, respectively. Indeed, suppose that $g^2(z_1)=g^2(z_2)$, $z_1,z_2\in O_g$. Then either $\Phi(z_1)=\Phi(z_2)$ and therefore $z_1=z_2$ by conformality of $\Phi$ or $\Phi(z_1)=1/\Phi(z_2)$, which is possible only if $\Phi^+(z_1)=\overline{\Phi^-(z_2)}$, i.e., if $z_1=z_2\in\Delta$. The case of $\widetilde g^2$ is no different.

\subsection{Jordan arcs $\Delta_n$, functions $g_n$ and $\widetilde g_n$} 
\label{subsec:gn}

By Definition~\ref{df:S}-(2) and upon taking $\rho$ smaller if necessary, we may assume that functions $r_n$ have no zeros in $\overline O$. Moreover, as $r_n$ is holomorphic in $D\setminus\Delta$ and has $2n$ zeros in $D\setminus\overline O$, its winding number is equal to $-2n$ on any positively oriented curve homologous to $L_\rho$ and contained in $O$. In other words, $r_n$ has a continuous argument that decreases by $4n\pi$ as $\Delta$ is encompassed once in the positive direction. Thus, the functions
\[
\Phi_n := r_n^{-1/2n} = \exp\left\{-\int\log\frac{\map(z)-\map(t)}{1-\map(z)\map(t)}d\nu_n(t)\right\},
\]
are well-defined and analytic in $O$, where $\nu_n$ is the normalized counting measure of the zeros of $r_n$. Moreover, as the counting measures of zeros of $r_n$ converge weak$^*$ to $\nu$ by assumption, the functions $\Phi_n$ converge to $\Phi$ uniformly in $\overline O$, distinguishing the one-sided values on $\Delta^\pm$. 

Hence, we can define 
\begin{equation}
\label{eq:gnfun}
\begin{array}{lll}
g_n := \log\Phi_n, & \displaystyle \lim_{z\to1}g_n(z) = 0, & g_n\in\hf(O\setminus L), \smallskip \\
\widetilde g_n := \log\Phi_n - \pi i, & \displaystyle \lim_{z\to-1}\widetilde g_n(z) = 0, & \widetilde g_n\in\hf(O\setminus \widetilde L).
\end{array}
\end{equation}
By \eqref{eq:mappr}, it is straightforward to see that $\Phi_n^+\Phi_n^-\equiv1$ on $\Delta$, and therefore
\begin{equation}
\label{eq:bvgn}
g_n^+ = -g_n^- \quad \mbox{and} \quad \widetilde g_n^+ = -\widetilde g_n^- \quad \mbox{on} \quad \Delta.
\end{equation}
Thus, $g_n^2$ and $\widetilde g_n^2$ are analytic in $O_g$ and $O_{\widetilde g}$, respectively. We choose domains $O_L\subset O_g$ and $O_{\widetilde L} \subset O_{\widetilde g}$ in such a manner that $O_L\supset\widetilde L$, $O_{\widetilde L}\supset L$, and $O_L\cup O_{\widetilde L}$ is simply connected and contains $\Delta$ (see Fig.~\ref{fig:1}). Then it is an easy consequence of the convergence of $\Phi_n$ to $\Phi$ that $g_n^2$ and $\widetilde g_n^2$ converge uniformly to $g^2$ and $\widetilde g^2$ on $\overline O_L$ and $\overline O_{\widetilde L}$, respectively. 

\begin{figure}[!ht]
\centering
\includegraphics[scale=.45]{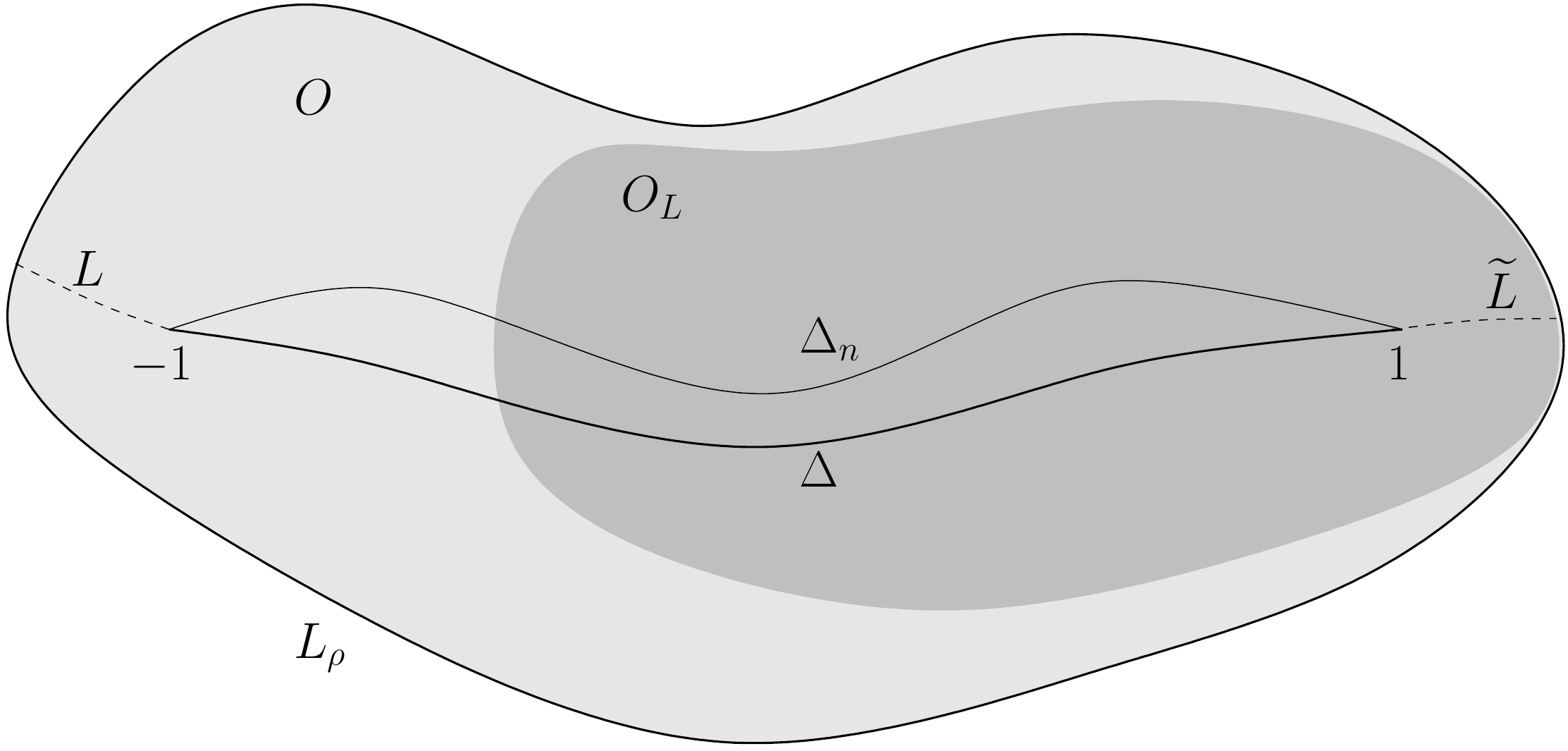}
\caption{\small The domain $O$ bounded by $\Delta$ and $L_\rho$(light grey), the domain $O_L\subset O\cup\Delta$ (dark grey), the cuts $L$ and $\widetilde L$ (dashed arcs).}
\label{fig:1}
\end{figure}

Next, we claim that $g_n^2$ and $\widetilde g_n^2$ are univalent for all $n$ large enough in $\overline O_L$ and $\overline O_{\widetilde L}$, respectively. Assume to the contrary that there exist two sequences of points $\{z_{1,n}\},\{z_{2,n}\}\subset\overline O_L$ such that $g_n^2(z_{1,n})=g_n^2(z_{2,n})$. As $\overline O_L$ is compact, we can assume that $z_{j,n}\to z_j\in \overline O_L$, $j=1,2$. Since $g_n^2$ converges to $g^2$ uniformly on $\overline O_L$, we have that $g^2(z_1)=g^2(z_2)$ and therefore $z_1=z_2$. Set $d_n(z):=(g^2_n(z)-g^2_n(z_{1,n}))/(z-z_{1,n})$. Then $d_n$ are analytic functions on $\overline O_L$ that converge uniformly to $d(z):=(g^2(z)-g^2(z_1))/(z-z_1)$. Moreover, the values $d_n(z_{2,n})$ are equal to $0$ and converge to $d(z_1)$. Thus, $(g^2)^\prime(z_1)=d(z_1)=0$, which is impossible since $g^2$ is univalent. This proves the claim as the case of $\widetilde g_n^2$ is no different. 

From the above we see that each $g_n^2$ maps $O_L$ conformally onto a neighborhood of zero as $g_n^2(1)=0$. Set $\Delta_{n,1}$ to be the preimage of the intersection of this neighborhood with $\Sigma_2:=\{\zeta:~\Arg(\zeta)=\pi\}$. Then $\Delta_{n,1}$ is an analytic arc with one endpoint being 1. Analogously, $\widetilde g_n^2$ maps $O_{\widetilde L}$ conformally into another neighborhood of zero. Thus, we can define $\Delta_{n,-1}$ to be again the preimage of the intersection of this neighborhood with $\Sigma_2$. Clearly, $\Delta_{n,-1}$ is an analytic arc with one endpoint being $-1$. Noticing that $g_n^2$ assumes negative values if and only if $\widetilde g_n^2$ assumes negative values on the common set of definition $O_L\cap O_{\widetilde L}$, we derive that $\Delta_n:=\Delta_{n,1}\cup \Delta_{n,-1}$ is an analytic arc with endpoint $\pm1$. 

\subsection{Parameterizations $\Xi_n$, functions $g_n^*$ and $\widetilde g_n^*$}

Now, define $\sr_n$ and $\map_n$ with respect to $\Delta_n$ like $\sr$ and $\map$ were defined in \eqref{eq:sr} and \eqref{eq:map} with respect to $\Delta$. Clearly, $\map_n$ is an analytic continuation of $\map$ from $D$ onto $D_n=\overline\C\setminus\Delta_n$. Further, let $r_n^*$ be defined by \eqref{eq:rn} with $\map$ replaced by $\map_n$ while keeping the same zeros as $r_n$. Hence, $r_n^*$ and $r_n$ are analytic continuations of each other defined in $D$ and $D_n$, respectively. Finally, set $\Phi_n^*$, $g_n^*$, and $\widetilde g_n^*$ to be the analytic continuations of $\Phi_n$, $g_n$, and $\widetilde g_n$ from $O$, $O\setminus L$, and $O\setminus\widetilde L$, onto $O_n:=(O\cup\Delta)\setminus \Delta_n$, $O_n\setminus L$, and $O_n\setminus\widetilde L$, defined in, by now, obvious manner. Hence
\[
g_n^2=(g_n^*)^2 \quad \mbox{and} \quad \widetilde g_n^2=(\widetilde g_n^*)^2,
\]
while $(g_n^*)^2$ is negative on $\Delta_n$. Since, in addition, $(g_n^*)^+=-(g_n^*)^-$ on $\Delta_n$, $(g_n^*)^\pm$ are pure imaginary on $\Delta_n$ and $|\Phi_n^*|\equiv1$ there. Therefore, $(\Phi_n^*)^+=\overline{(\Phi_n^*)^-}$ and $(r_n^*)^+=\overline{(r_n^*)^-}$.  Furthermore, $\Phi_n^*$ maps $O_n$ onto some annular domain having the unit circle as a component of its boundary. Arguing as was done after \eqref{eq:reciprocityPhi}, we derive that $\jt\circ\Phi_n^*$ is holomorphic across $\Delta_n$ and that $\Xi_n:=(\jt\circ\Phi_n^*)^{-1}$ is a holomorphic parameterization of $\Delta_n$, $\Xi_n([-1,1])=\Delta_n$. Moreover, as the functions $\Phi_n$ converge to $\Phi$ uniformly in some annular domain encompassing $\Delta$, we see that the functions $\jt\circ\Phi_n^*$ converge locally uniformly to $\jt\circ\Phi$ in some neighborhood of $\Delta$. Hence, the sequence of analytic parameterizations $\{\Xi_n\}$ of $\Delta_n$ converges uniformly to the analytic parameterization $\Xi$ of $\Delta$ in some neighborhood of $[-1,1]$. This finishes the proof of Theorem~\ref{thm:sp}.

\section{Trace Theorems and Extensions}
\label{sec:extension}

As is usual in the Riemann-Hilbert approach to asymptotics of orthogonal polynomials, we shall need to extend the weights of orthogonality from $\Delta$ into subsets of the complex plane. As the weights are not analytic, this extension will require a special construction that we carry out in this section.

\subsection{Domains with Smooth Boundaries} 
\label{ss:sb} 

In this section we suppose that $\Omega$ is a bounded simply connected domain with boundary $\Gamma$ which is infinitely smooth and contains $\Delta$, i.e., $\Delta\subset\Gamma$.

\begin{df}
\label{df:w}
Set $\lf^p(\Omega)$, $p\in[1,\infty)$, to be the space of all measurable functions $f$ such that $|f|^p$ is integrable over $\Omega$. The Sobolev space $\sof_p^1(\Omega)$, $p\in[1,\infty)$, is the subspace of $\lf^p(\Omega)$ that comprises of functions with weak partial derivatives also in $\lf^p(\Omega)$.
\end{df} 

Then the following theorem takes place \cite[Thm. 1.5.1.2]{Grisvard}.

\begin{thmT1}
\label{thm:T1}
For each $f\in\sof_p^{1-1/p}(\Gamma)$, $p\in(1,\infty)$, there exists $F\in\sof_p^1(\Omega)$ such that $F_{|\Gamma}=f$. Moreover, the extension operator can be made independent of $p$. Conversely, for every $F\in\sof^1_p(\Omega)$ it holds that $F_{|\Gamma}\in\sof^{1-1/p}_p(\Gamma)$.
\end{thmT1}

Together with the Sobolev spaces $\sof_p^1(\Omega)$, we consider smoothness classes $\cf^{m,\varsigma}(\overline\Omega)$.
\begin{df}
\label{df:c}
By $\cf^{m,\varsigma}(\overline\Omega)$, $m\in\Z_+$,  $\varsigma\in(0,1]$, we denote the space of all functions on $\overline\Omega$ whose partial derivatives up to the order $m$ are continuous on $\overline\Omega$ and whose partial derivatives of order $m$ are uniformly H\"older continuous on $\overline\Omega$ with exponent $\varsigma$. Moreover, $\cf^{m,\varsigma}_0(\overline\Omega)$ will stand for the subset of $\cf^{m,\varsigma}(\overline\Omega)$ consisting of functions whose partial derivatives up to order $m$, including the function itself, vanish on $\Gamma$. Finally, $\cf^\infty(\overline\Omega)$ will denote the space of functions on $\Omega$ whose partial derivatives of any order exist and are continuous on~$\overline\Omega$.
\end{df} 

It is known from Sobolev's imbedding theorem \cite[Thm. 5.4]{Adams} that
\begin{equation}
\label{eq:sobimb}
\sof_p^1(\Omega) \subset \left\{
\begin{array}{ll}
\lf^{2p/(2-p)}(\Omega), & p\in[1,2), \smallskip \\
\cf^{0,{1-2/p}}(\overline\Omega), & p\in(2,\infty).
\end{array}
\right.
\end{equation}
Hence, for $f\in\sof^{1-1/p}_p(\Gamma)$, $p\in(2,\infty)$, the function $F$ granted by Theorem~\hyperref[thm:T1]{T1} belongs to $\cf^{0,{1-2/p}}(\overline\Omega)$ and therefore $f\in\cf^{0,{1-2/p}}(\Gamma)$, which is exactly what was stated in \eqref{eq:simb}.

Later on, we shall need the following proposition.
\begin{prop}
\label{pr:smooth1}
Let $f$ be a continuous function on $\Delta$ such that $f(\pm1)=0$. If $f\in\sof^{1-1/p}_p$, $p\in(2,\infty)$, then there exists $F\in\sof^1_p(\Omega)$ such that $F_{|\Delta}=f$. Moreover, if $f\in\cf^{0,\varsigma}$, $\varsigma\in(1/2,1]$, then there exists $F\in\sof^1_q(\Omega)$ for any $q\in(2,\frac{1}{1-\varsigma})$ such that $F_{|\Delta}:=f$.
\end{prop}
\begin{proof} In both cases set $\tilde f=f$ on $\Delta$ and $\tilde f \equiv0$ on $\Gamma\setminus\Delta$. When $f\in\sof^{1-1/p}_p$, it is immediate to check that $\tilde f\in\sof^{1-1/p}_p(\Gamma)$ and therefore the first claim follows from Theorem~\hyperref[thm:T1]{T1}. When $f\in\cf^{0,\varsigma}$, it holds that $\tilde f\in\cf^{0,\varsigma}(\Gamma)$ and $\widetilde f\in\sof^{1-1/q}_q(\Gamma)$ for any $q\in(1,\frac{1}{1-\varsigma})$ by an easy estimate (see Definition~\ref{df:sobolev}). Hence, the second claim of the proposition again follows from Theorem~\hyperref[thm:T1]{T1}.
\end{proof}

To state a trace theorem for classes $\cf^{m,\varsigma}(\overline\Omega)$, we need to introduce the notion of a directional derivative. Namely, let $\xi$ be a continuous function on $\overline\Omega$ and $f\in\sof^1_p(\Omega)$. With the slight abuse of notation, we define the derivative of $f$ in the direction of the field $\xi$, denoted by $\partial_\xi$, as
\begin{equation}
\label{eq:directder}
\partial_\xi f := \bar\xi\bar\partial f + \xi\partial f = \re(\xi)\partial_xf+\im(\xi)\partial_yf,
\end{equation}
where $\nabla f$ is the gradient of $f$, $\vec{\xi}$ is the vector field with values in $\R^2$ corresponding to $\xi$.

As $\Gamma$ is infinitely smooth, any conformal map $\phi$ of $\Omega$ onto the unit disk belongs to $\cf^\infty(\overline\Omega)$ \cite[Thm. 3.6]{Pommerenke}. Moreover, it holds that $\phi^\prime\neq0$ in $\overline\Omega$. Thus, we may set
\begin{equation}
\label{eq:nfield}
n(z) := \frac{\phi(z)}{\phi^\prime(z)} \quad \mbox{and} \quad \mathbf{n}(z):=\phi(z)\frac{|\phi^\prime(z)|}{\phi^\prime(z)}, \quad z\in\overline\Omega.
\end{equation}
Then $n,\mathbf{n}\in\cf^\infty(\overline\Omega)$ and $n$ is holomorphic in $\Omega$. Moreover, for any $z\in\Gamma$, $\mathbf{n}(z)=n(z)/|n(z)|$ represents the complex number corresponding to the outer unit normal to $\Gamma$ at $z$. Then the following theorem takes place \cite[Thm. 6.2.6]{Grisvard}.

\begin{thmT21}
\label{thm:T21}
Let $\{f_k\}_{k=0}^m$ be such that $f_k\in\cf^{m-k,\varsigma}(\Gamma)$, $m\in\N$, $\varsigma\in(0,1]$, $k\in\{0,\ldots,m\}$. Then there exists $F\in\cf^{m,\varsigma}(\overline\Omega)$ such that $(\partial_\mathbf{n}^kF)_{|\Gamma}=f_k$, $k\in\{0,\ldots,m\}$.
\end{thmT21}

Now, observe that
\[
\partial^k_\mathbf{n} = |\phi^\prime|^k\partial^k_n + \sum_{j=1}^{k-1} c_{k,j} \partial^j_n,
\]
where the functions $c_{k,j}$ involve sums and products of the powers of the iterated directional derivatives of $|\phi^\prime|$ with respect to the field $n$ and therefore belong to $\cf^\infty(\overline\Omega)$. Set $c_{k,k}:=|\phi^\prime|^k$ and $c_{k,j}\equiv0$, $j\in\{k+1,\ldots,m\}$, $k\in\{1,\ldots,m\}$. Then the matrix $\mathscr{C}_\phi=[c_{k,j}]_{k,j=1}^m$ is such that $\det(\mathscr{C}_\phi)=|\phi^\prime|^{m(m+1)/2}$, which is non-vanishing at any $z\in\overline\C$, and
\[
(\partial_\mathbf{n},\ldots,\partial^m_\mathbf{n})^\mathrm{T} = \mathscr{C}_\phi (\partial_n,\ldots,\partial^m_n)^\mathrm{T}.
\]
Thus, to every family of functions $\{f_k\}_{k=0}^m$, $f_k\in\cf^{m-k,\varsigma}(\Gamma)$, there corresponds another family, say $\{\tilde f_k\}_{k=0}^m$, $\tilde f_k\in\cf^{m-k,\varsigma}(\Gamma)$, such that there exists $F\in\cf^{m,\varsigma}(\overline\Omega)$ satisfying $(\partial_\mathbf{n}^kF)_{|\Gamma}=f_k$ and $(\partial_n^kF)_{|\Gamma}=\tilde f_k$. Moreover, this correspondence is one-to-one and onto. Hence, Theorem~\hyperref[thm:T21]{T2(1)} can be equivalently reformulated as follows.

\begin{thmT22}
\label{thm:T22}
Let $\{f_k\}_{k=0}^m$ be such that $f_k\in\cf^{m-k,\varsigma}(\Gamma)$, $m\in\N$, $\varsigma\in(0,1]$, $k\in\{0,\ldots,m\}$. Then there exists $F\in\cf^{m,\varsigma}(\overline\Omega)$ such that $(\partial_n^kF)_{|\Gamma}=f_k$, $k\in\{0,\ldots,m\}$.
\end{thmT22}

Finally, we define $\tau:=in$. Clearly, $\tau(z)/|\tau(z)|$, $z\in\Gamma$, is the complex number corresponding to the positively oriented unit tangent to $\Gamma$ at $z$. Since $n$ and $\tau$ are holomorphic functions such that $\tau=in$, it is a simple computation to verify that $\partial_n\partial_\tau F=\partial_\tau\partial_n F$ in $\overline\Omega$. Then the following proposition holds.

\begin{prop}
\label{pr:smooth2}
Let $f\in\cf^{m,\varsigma}(\Delta)$, $m\in\N$, $\varsigma\in(0,1]$, $f^{(k)}(\pm1)=0$, $k\in\{0,\ldots,m\}$. Then there exists $F\in\cf^{m,\varsigma}(\overline\Omega)$ such that $F_{|\Delta}=f$ and $\bar\partial F\in\cf^{m-1,\varsigma}_0(\overline\Omega)$.
\end{prop}
\begin{proof}
Set $f_0=f$ on $\Delta$ and $f_0\equiv0$ on $\Gamma\setminus\Delta$. It is clear that $f_0\in\cf^{m,\varsigma}(\Gamma)$. Further, set $f_k:=(-i)^k\partial_\tau^kf_0$, $k\in\{1,\ldots,m\}$. As $f_k\in\cf^{m-k,\varsigma}(\Gamma)$, $k\in\{0,\ldots,m\}$, Theorem~\hyperref[thm:T22]{T2(2)} yields that there exists $F\in\cf^{m,\varsigma}(\overline\Omega)$ such that $(\partial^k_nF)_{|\Gamma}=f_k$. In particular, $F_{|\Delta}=f$. 

It remains only to show that $\bar\partial F\in\cf^{m-1,\varsigma}_0(\overline\Omega)$. It can be easily checked that
\begin{equation}
\label{eq:partialbar}
2\bar n \bar\partial F = \partial_n F + i \partial_\tau F =: H \quad \mbox{in} \quad \Omega^\prime,
\end{equation}
where $\Omega^\prime$ is an annular domain such that $\Gamma\subset\partial\Omega^\prime$ and $n$ is non-vanishing on this domain. As $n\in\cf^\infty(\overline\Omega)$ and is zero free in $\Omega^\prime$, it holds that $\bar\partial F\in\cf^{m-1,\varsigma}_0(\overline\Omega)$ if and only if $H\in\cf^{m-1,\varsigma}_0(\overline\Omega)$. Moreover, it is immediate from the construction of $F$ that $H\in\cf^{m-1,\varsigma}(\overline\Omega)$. Thus, it is only necessary to verify that all the partial derivatives of $H$ of order $k$, for any $k\in\{0,\ldots,m-1\}$, vanish on $\Gamma$. Since partial derivatives with respect to $n$ and $\tau$ commute, and these fields are non-vanishing and non-collinear in $\Omega^\prime$, it is enough to show that
\[
0 \equiv (\partial_\tau^{k_1}\partial_n^{k_2}H)_{|\Gamma} = \left(\partial_\tau^{k_1} \partial_n^{k_2+1} F\right)_{|\Gamma} + i \left(\partial_\tau^{k_1+1} \partial_n^{k_2} F\right)_{|\Gamma}
\]
for all $k_1+k_2\in\{0,\ldots,m-1\}$. The latter holds since
\[
\left(\partial_\tau^{k_1} \partial_n^{k_2+1} F\right)_{|\Gamma} = \partial_\tau^{k_1}f_{k_2+1} = (-i)^{k_2+1} \partial_\tau^{k_1+k_2+1}f_0 = -i\partial_\tau^{k_1+1}f_{k_2} = -i \left(\partial_\tau^{k_1+1}\partial_n^{k_2} F\right)_{|\Gamma},
\]
by the choice of $\{f_k\}$.
\end{proof}

\subsection{Domains with Polygonal Boundary}
\label{ss:pb}

The previous results also hold, with some modifications, for domains with polygonal boundaries. Namely, let $\Omega$ be a domain whose boundary is a curvilinear polygon consisting of two pieces, say $\Delta_1$ and $\Delta_2$, such that they might form corners at the joints. As we do not strive for generality at this point, we assume that each $\Delta_j$ is an analytic arc connecting $-1$ and 1.

The first trace theorem of this section states the following \cite[Thm. 1.5.2.3]{Grisvard}.

\begin{thmT3}
\label{thm:T3}
Given $f_j\in\sof^{1-1/p}_p(\Delta_j)$, $j=1,2$, satisfying $f_1(\pm1)=f_2(\pm1)$, there exists $F\in\sof^1_p(\Omega)$ such that $F_{|\Delta_j}=f_j$, $j=1,2$. The choice of $F$ can be made in such a way that it depends only on $f_j$ and not on $p$.
\end{thmT3}

To state an analogous theorem for the classes $\cf^{m,\varsigma}(\Delta_j)$, we again need to define normal fields on $\Delta_j$, $j=1,2$. Let $\Gamma_j$ be an infinitely smooth Jordan curve such that $\Delta_j\subset\Gamma_j$. Moreover, assume that the interior domain of $\Gamma_j$, say $\Omega_j$, contains $\Omega$, $j=1,2$. Define $n_j$  for $\Omega_j$ as it was done in \eqref{eq:nfield}. Composing with a self-map of the disk if necessary, we can choose conformal maps in \eqref{eq:nfield} so $n_j$ does not vanish in $\Omega$. Further, set $\tau_j:= in_j$ if $\Omega$ lies on the left side of $\Delta_j$ and $\tau_j:= -in_j$ otherwise. In particular, the fields $\tau_j$ and $n_j$ commute, are infinitely smooth, non-vanishing and non-collinear. Finally, observe that \eqref{eq:partialbar} holds with $n$, $\tau$, and $\Omega^\prime$ replaced by $n_j$, $\tau_j$, $\Omega$, and the plus sign replaced by the minus sign in the right-hand side of \eqref{eq:partialbar} when $\tau_j=-in_j$. 

With all the necessary material at hand, we can state a special case of the trace theorem for smoothness classes on domains with polygonal boundary\footnote{In Theorem~\hyperref[thm:T4]{T4} we use non-unit normal fields $n_j$ rather than fields that are unit on $\Delta_j$ as it was done in the original reference. However, we have already explained after Theorem~\hyperref[thm:T21]{T2(1)} that these formulations are equivalent.} \cite[Cor. 6.2.8]{Grisvard}. 

\begin{thmT4}
\label{thm:T4}
Given $\{f_{jk}\}_{k=0}^m$, $f_{jk}\in\cf^{m-k,\varsigma}(\Delta_j)$, $m\in\N$, $\varsigma\in(0,1]$, $j=1,2$, satisfying $f_{jk_1}^{(k_2)}(\pm1)=0$, $k_1+k_2\in\{0,\ldots,m\}$, there exists $F\in\cf^{m,\varsigma}(\overline\Omega)$ such that $(\partial^k_{n_j}F)_{|\Delta_j}=f_{jk}$, $k\in\{0,\ldots,m\}$, $j=1,2$.
\end{thmT4}

Now, as in Section~\ref{ss:sb}, we shall make Theorems~\hyperref[thm:T3]{T3} and \hyperref[thm:T4]{T4} suit our needs. Let $\Delta$ be a closed analytic Jordan arc and $\Delta_\pm$ be two closed analytic Jordan arcs with endpoints $\pm1$ such that the interior domain of $\Delta\cup\Delta_+$, say $\Omega_+$, is simply connected and lies to the left of $\Delta$ while the interior domain of $\Delta\cup\Delta_-$, say $\Omega_-$, is again simply connected and lies to the right of $\Delta$. Then the following proposition holds.

\begin{prop}
\label{pr:corners1}
Let $f$ be a continuous function on $\Delta$ such that $f(\pm1)=0$. If $f\in\sof^{1-1/p}_p$, $p\in(2,\infty)$, then there exist $F_\pm\in\sof^1_p(\Omega_\pm)$ such that
\begin{equation}
\label{eq:Fpm}
F_{\pm|\Delta}=\pm f \quad \mbox{and} \quad F_{\pm|\Delta_\pm}\equiv0.
\end{equation}
Moreover, if $f\in\cf^{0,\varsigma}$, $\varsigma\in(1/2,1]$, then there exist $F_\pm\in\sof^1_q(\Omega_\pm)$ for any $q\in(2,\frac{1}{1-\varsigma})$ satisfying \eqref{eq:Fpm}.
\end{prop}
\begin{proof}
This proposition follows from Theorem~\hyperref[thm:T3]{T3} in the same fashion as Proposition~\ref{pr:smooth1} followed from Theorem~\hyperref[thm:T1]{T1}.
\end{proof}

Finally, we state the counterpart of Proposition~\ref{pr:smooth2} for domains with corners.

\begin{prop}
\label{pr:corners2}
Let $f\in\cf^{m,\varsigma}(\Delta)$, $m\in\N$, $\varsigma\in(0,1]$, $f^{(k)}(\pm1)=0$, $k\in\{0,\ldots,m\}$. Then there exist functions $F_\pm\in\cf^{m,\varsigma}(\overline\Omega_\pm)$ such that
\begin{equation}
\label{eq:FOmegaPM}
F_{\pm|\Delta}=\pm f, \quad F_{\pm|\Delta_\pm}\equiv0, \quad \mbox{and} \quad  \bar\partial F_\pm\in\cf^{m-1,\varsigma}_0(\overline\Omega_\pm).
\end{equation}
\end{prop}
\begin{proof}
First, we consider the case of $\Omega_+$. By setting $f_{1k}:=(-i)^kf^{(k)}$, $k\in\{0,\ldots,m\}$, we see that $f_{1k}\in\cf^{s-k}(\Delta)$, $k\in\{0,\ldots,m\}$. Moreover, after putting $f_{2k}\equiv0$, we observe that
\[
f_{jk_1}^{(k_2)}(\pm1)=0, \quad k_1+k_2\in\{0,\ldots,m\}, \quad j=1,2.
\]
Then the existence of $F_+$ in $\cf^{m,\varsigma}(\overline\Omega_+)$ follows from Theorem~\hyperref[thm:T4]{T4}. The fact that $\bar\partial F_+\in\cf^{m-1,\varsigma}_0(\overline\Omega_+)$ can be shown exactly as in Proposition~\ref{pr:smooth2}. In the case of $\Omega_-$ the only difference is that we need to set $f_{1k}:=-i^kf^{(k)}$ since this time the normal and tangent on $\Delta$ satisfy $\tau=-in$.
\end{proof}

\section{Scalar Boundary Value Problems}
\label{sec:bvp}

In this section we dwell on smoothness properties of certain integral operators. 

\subsection{Integral Operators}
\label{subsec:io}

Below we introduce contour and area integral operators and explain the solution of a certain $\bar\partial$-problem.

Let $\phi$ be an $\lf^p$, $p>1$, function on $\Delta$, where $\lf^p=\lf^p(\Delta)$ stands for the space of functions with $p$-summable modulus on $\Delta$ with respect to arclength differential $|dt|$. The Cauchy integral operator on $\Delta$ is defined by
\begin{equation}
\label{eq:oc}
\oc\phi(z) := \oc_\Delta\phi(z) = \frac{1}{2\pi i}\int_\Delta\frac{\phi(t)}{t-z}dt, \quad z\in D.
\end{equation}
It is known that $\oc\phi$ is a holomorphic function in $D$ with $\lf^p$ traces on $\Delta$, i.e., non-tangential limits a.e. on $\Delta$, from above and below denoted by $\oc^{\pm}\phi$. These traces are connected by the Sokhotski-Pemelj formulae \cite[Sec. I.4.2]{Gakhov}, i.e.,
\begin{equation}
\label{eq:sokhotski}
\oc^+\phi-\oc^-\phi = \phi \quad \mbox{and} \quad \oc^+\phi+\oc^-\phi = \os\phi, \quad \mbox{a.e. on} \quad \Delta,
\end{equation}
where $\os$ is the singular integral operator on $\Delta$ given by
\begin{equation}
\label{eq:os}
\os\phi(\tau) := \os_\Delta\phi(\tau)= \frac{1}{\pi i}\int_\Delta\frac{\phi(t)}{t-\tau}dt, \quad \tau\in \Delta^\circ,
\end{equation}
with the integral being understood in the sense of the principal value.

Let now $\Omega$ be a simply connected bounded domain with smooth boundary $\Gamma$. We define $\oc_\Gamma$ and $\os_\Gamma$ by \eqref{eq:oc} and \eqref{eq:os} integrating this time over $\Gamma$ rather than $\Delta$. The Sokhotski-Plemelj formulae \eqref{eq:sokhotski} still hold for $\phi\in\lf^p(\Gamma)$, $p>1$, with the only difference that now $\oc_\Gamma\phi$ is a sectionally holomorphic function and therefore $\oc_\Gamma^+\phi$ is the trace of $\oc_\Gamma\phi$ from within $\Omega$ and $\oc_\Gamma^-\phi$ is the trace of $\oc_\Gamma\phi$ from within $\overline\C\setminus\overline\Omega$.

Concerning the smoothness of $\oc_\Gamma\phi$ the following is known. If $\phi\in\cf^{0,\varsigma}(\Gamma)$, $\varsigma\in(0,1)$, then $\oc_\Gamma\phi\in\cf^{0,\varsigma}(\overline\Omega)$ \cite[Sec. 5.5.1]{Gakhov}. In particular, this means that $\oc_\Gamma\phi$ extends continuously from $\Omega$ to $\Gamma$. Further, if $\phi$ is continuously differentiable on $\Gamma$, then $\oc^\prime_\Gamma\phi=\oc_\Gamma\phi^\prime$ \cite[Sec. 4.4.4]{Gakhov}. Thus, we may conclude that when $\phi\in\cf^{m,\varsigma}(\Gamma)$, $m\in\Z_+$, $\varsigma\in(0,1)$, then $\oc_\Gamma\phi\in\cf^{m,\varsigma}(\overline\Omega)$.

Let now $\phi\in\lf^p(\Omega)$. The Cauchy area integral on $\Omega$ is defined as
\begin{equation}
\label{eq:ok}
\ok\phi(z) := \frac{1}{2\pi i}\iint_\Omega\frac{\phi(\zeta)}{\zeta-z}d\zeta\wedge d \bar\zeta, \quad z\in\Omega.
\end{equation}
Then it is well-known \cite[Sec. 4.9]{AstalaIwaniecMartin} that
\begin{equation}
\label{eq:okreverse}
\bar\partial\ok\phi = \phi \quad \mbox{and} \quad \partial\ok\phi=\ob\phi,
\end{equation}
in the distributional sense, where $\ob$ is the Beurling transform, i.e.,
\begin{equation}
\label{eq:beurling}
\ob\phi(z) := \frac{1}{2\pi i}\iint_\Omega\frac{\phi(\zeta)}{(\zeta-z)^2}d\zeta\wedge d \bar\zeta, \quad z\in\Omega,
\end{equation}
and the integral is understood in the sense of the principal value. 

The transformation $\ok$ defines a bounded operator from $\lf^p(\Omega)$ into $\lf^{2p/(2-p)}(\Omega)$ for $p\in(1,2)$, \cite[Thm. 4.3.8]{AstalaIwaniecMartin}, and into $\cf^{1-2/p}(\overline\Omega)$ for $p\in(2,\infty)$, \cite[Thm. 4.3.13]{AstalaIwaniecMartin}. Since nothing prevents us from taking $z$ outside of $\overline\Omega$, $\ok\phi$ is, in fact, defined throughout $\overline\C$ and is clearly holomorphic outside of $\overline\Omega$ and vanishes at infinity. Moreover, $\ok\phi$ is continuous across $\Gamma$ when $p\in(2,\infty)$. The latter can be easily seen if we continue $\phi$ by zero to a larger domain, say $\widetilde\Omega$, and observe that this extension is in $\lf^p(\widetilde\Omega)$. 

The Beurling transform $\ob$ defines a bounded operator from a weighted space $\lf^p_v(\C):=\{f:f^pv\in\lf^p(\C)\}$ into itself when the non-negative function $v$ is an \emph{$A_p$-weight (Muckenhoupt weight)}, $p\in(1,\infty)$ \cite[Thm. 4.9.6]{AstalaIwaniecMartin}. Let $\phi\in\lf^p(\Omega)$. We can suppose that $\phi\in\lf^p(\C)$ with $\phi\equiv0$ outside of $\Omega$ and therefore $\phi/\sr\in\lf^p_{|\sr|^p}(\C)$. It holds that $|\sr|^p$ is an $A_p$-weight for $p>2$ \cite[Sec. 9.1.b]{Grafakos}. Thus, $\ob(\phi/\sr)\in\lf^p_{|\sr|^p}(\C)$ and therefore 
\begin{equation}
\label{eq:beurlinglp}
\phi\in\lf^p(\Omega) \quad \mbox{implies} \quad \sr\ob(\phi/\sr)\in\lf^p(\Omega), \quad p>2.
\end{equation} 

Finally, we point out that $\phi\in\sof^1_p(\Omega)$ can be recovered by means $\oc_\Gamma$ and $\ok$ in the following fashion:
\begin{equation}
\label{eq:cthnaf}
\phi = \oc_\Gamma\phi + \ok\bar\partial\phi \quad \mbox{a.e. in} \quad \Omega,
\end{equation}
which is the Cauchy-Green formula for Sobolev functions.

\subsection{Functions of the Second Kind}
\label{subsec:secondkind}

Let $R_n$ be given by \eqref{eq:skind} with $q_n$ satisfying \eqref{eq:ortho} and $w_n$ defined as in \eqref{eq:varyingweight}. Clearly, $R_n$ is holomorphic in $D$, and it vanishes at infinity with order at least $n+1$, i.e., $R_n = O(z^{-n-1})$ as $z\to\infty$, on account of \eqref{eq:ortho}. It is also clear that $R_n=2\oc(q_nw_n)$. Thus, it holds by \eqref{eq:sokhotski} that 
\[
R_n^+-R_n^- = 2q_nw_n.
\]
Further, since $q_nw_n/w=q_nh_nh/v_n$ is H\"older continuous by the conditions of Theorem~\ref{thm:sa}, we have that
\begin{equation}
\label{eq:rnnear1}
R_n = \left\{
\begin{array}{ll}
O(|1-z|^\alpha), & \mbox{if} \quad \alpha<0, \smallskip \\
O(\log|1-z|), & \mbox{if} \quad \alpha=0, \smallskip \\
O(1), & \mbox{if} \quad \alpha>0,
\end{array}
\right.
\end{equation}
and analogous asymptotics holds near $-1$. Indeed, the case $\alpha<0$ follows from \cite[Sec. I.8.3 and I.8.4]{Gakhov}. (Observe that we defined $(1-t)^\alpha$, $t\in \Delta^\circ$, as the values on $\Delta$ of $(1-z)^\alpha$, where the latter is holomorphic outside of the branch cut taken along $\Delta_r$. However, $(1-t)^\alpha$ equivalently can be regared as the boundary values of $(1-z)^\alpha$ on $\Delta^+$, where the latter is holomorphic outside of the branch cut taken along $\Delta_l\cup\Delta$. Hence, the analysis in \cite[Sec. I.8.3]{Gakhov} indeed applies to the present situation.) The case $\alpha=0$ follows from \cite[Sec. I.8.1 and I.8.4]{Gakhov}. Finally, the case $\alpha>0$ holds since $R_n(1)$ exists for such $\alpha$ as $w_n(t)/(t-1)$ is integrable near 1 in this situation.

\subsection{Szeg\H{o} Functions}
\label{subsec:szego}

Let $\theta\in\cf^{m,\varsigma}$, $m\in\Z_+$, $\varsigma(0,1]$, and $h:=e^\theta$. The definition of the Szeg\H{o} function given in \eqref{eq:szf} can be rewritten as
\[
\szf_h = \exp\left\{\sr\oc\left(\frac{\theta}{\sr^+}\right)-\frac12\int\theta d\ed\right\}.
\]
Note that decomposition \eqref{eq:szego} easily follows from the Sokhotski-Plemelj formulae \eqref{eq:sokhotski}. Moreover, as the lemma in the next section shows, the traces $\szf_h^\pm$ belong to $\cf^{m,\varsigma^\prime}$, $0<\varsigma^\prime<\varsigma$, and $\szf_h^+(\pm1)=\szf_h^-(\pm1)$. In particular, the functions\footnote{Here we slightly abuse the notation and use superscripts $+$ and $-$ as a part of the symbol for the function. However, in Lemma~\ref{lem:extension} we shall construct a function $c_h$ whose traces on $\Delta$ will coincide with $c_h^\pm$.} $c_h^+:=\szf_h^+/\szf_h^-$ and $c_h^-:=\szf_h^-/\szf_h^+$ are continuous on $\Delta$ and assume value 1 at $\pm1$. It also follows from the Sokhotski-Plemelj formulae that
\begin{equation}
\label{eq:c}
c_h^\pm = \exp\left\{\sr^\pm\os\left(\frac{\theta}{\sr^+}\right)\right\}.
\end{equation}

The following facts are explained in detail in \cite[Sec. 3.2 and 3.3]{uBY3}. First, if $\theta_1,\theta_2\in\cf^{m,\varsigma}$, then $\szf_{h_1h_2} = \szf_{h_1}\szf_{h_2}$. Second, if $\{\theta_n\}$ is a normal family in some neighborhood of $\Delta$ then $\{\szf_{h_n}\}$ is a normal family in $D$. If, in addition, $\{\theta_n\}$ converges then $\{\szf_{h_n}\}$ converges as well and the convergence is \emph{uniform} on the closure of $D$, that is, including the boundary values from each side. 

Third, the uniqueness of decomposition \eqref{eq:szego}, which was shown, for instance, in \cite[eq. (2.7) and after]{Y10}, implies the following formula for the Szeg\H{o} function of the polynomial $v_n$, $\deg(v_n)\leq2n$, with zeros in~$D$:
\begin{equation}
\label{eq:szfpoly}
\szf_{v_n^2}=\szf_{v_n}^2 = \frac{1}{\gm_{v_n}}\frac{v_n}{r_n\map^{2n}},
\end{equation}
where $r_n$ was defined in \eqref{eq:rn}.

Fourth, observe that it is possible to define continuous arguments of $(z+1)/\map(z)$ and $(z-1)/\map(z)$ that vanish on the real axis in some neighborhood of infinity. Therefore it holds that
\begin{equation}
\label{eq:szfw}
\szf_w(z) = \left(2\frac{z-1}{\map(z)}\right)^{\alpha/2} \left(2\frac{z+1}{\map(z)}\right)^{\beta/2} \quad \mbox{and} \quad \gm_w = 2^{-(\alpha+\beta)},
\end{equation}
where $w$ was defined in \eqref{eq:wab} and the branches of the power functions are taken such that the positive reals are mapped into the positive reals.

Finally, using \eqref{eq:szfw} with $w=w(1/2,1/2;\cdot)$ we have that
\[
\szf_w^+(t)\szf_w^-(t)=2\sqrt{1-t^2} = -2i\sr^+(t) \quad t\in\Delta.
\]
Hence, we get that
\begin{equation}
\label{eq:szfsr}
\szf_{\sr^+} = \sqrt{2\sr/\map} \quad \mbox{and} \quad \gm_{\sr^+}=i/2,
\end{equation}
where, as usual, the branch of the square root is chosen so that $\szf_{\sr^+}$ is positive for large positive reals. It will be useful for us later to note that
\begin{equation}
\label{eq:szfsrpm}
\left(\map\szf_{\sr^+}\right)^{\pm} = (\szf_{\sr^+}^\pm)^2\frac{\map^\pm}{\szf_{\sr^+}^\pm\szf_{\sr^+}^\mp}\szf_{\sr^+}^\mp = (\szf_{\sr^+}^\pm)^2\frac{i\map^\pm}{\pm2\sr^\pm}\szf_{\sr^+}^\mp = \pm i\szf_{\sr^+}^\mp,
\end{equation}
where we used \eqref{eq:szego} and \eqref{eq:szfsr}.

\subsection{Smoothness of a Singular Integral Operator}

In this section we show that the boundary values on $\Delta$ of the Szeg\H{o} function of $e^\theta$ have essentially the same Sobolev or H\"older smoothness as $\theta$. 

We start with the case of functions in $\sof^{1-1/p}_p$, $p\in(2,\infty)$. Observe that $\sof^{1-1/q}_q\supset\sof^{1-1/p}_p$ when $q<p$, which is immediate from Definition~\ref{df:sobolev}.

\begin{prop}
\label{lem:smooth1}
Let $\theta\in\sof^{1-1/p}_p$, $p\in(2,\infty)$. Then 
\[
\sr^\pm\os(\theta/\sr^+) = \pm d +\sr^\pm\ell, \quad d(\pm1)=0, 
\]
where $d\in\sof^{1-1/q}_q$ for any $q\in(2,p)$ and $\ell$ is a polynomial, $\deg(\ell)\leq1$.
\end{prop}
\begin{proof}
It follows immediately from Cauchy integral formula and the Sokhotski-Plemelj formulae \eqref{eq:sokhotski} that
\begin{equation}
\label{eq:singzero}
\oc\left(\frac{1}{\sr^+}\right) = \frac{1}{2\sr^+} \quad \mbox{and} \quad \os\left(\frac{1}{\sr^+}\right) = 0.
\end{equation}
Hence, for any polynomial $\ell_0$ and $\tau\in\Delta^\circ$ it holds that
\[
\os\left(\frac{\ell_0}{\sr^+}\right)(\tau) = \frac{1}{\pi i}\int_\Delta\frac{\ell_0(t)}{t-\tau}\frac{dt}{\sr^+(t)} = \frac{1}{\pi i}\int_\Delta\frac{\ell_0(t)-\ell_0(\tau)}{t-\tau}\frac{dt}{\sr^+(t)} = \ell_1(\tau),
\]
where $\ell_1$ is a polynomial, $\deg(\ell_1)<\deg(\ell_0)$, since $\frac{\ell_0(\cdot)-\ell_0(\tau)}{\cdot-\tau}$ is a polynomial in $\tau$. Choose $\ell_0$ to be the polynomial interpolating $\theta$ at $\pm1$, $\deg(\ell_0)\leq1$. Then
\[
\sr^\pm\os\left(\frac{\theta}{\sr^+}\right) = \sr^\pm\os\left(\frac{\theta-\ell_0}{\sr^+}\right) + \sr^\pm\ell_1.
\]
Thus, it holds by \eqref{eq:sokhotski} that
\[
\sr^\pm\os\left(\frac{\theta}{\sr^+}\right) = 2\sr^\pm\left(\oc^+\left(\frac{\theta-\ell_0}{\sr^+}\right)-\ell_2\right) \mp (\theta-\ell_0) +\sr^\pm(2\ell_2+\ell_1),
\]
where $\ell_2$, $\deg(\ell_2)\leq1$ will be chosen later. Set $\ell:=\ell_1+2\ell_2$ and
\[
d := 2\sr^+\left(\oc^+\left(\frac{\theta-\ell_0}{\sr^+}\right)-\ell_2\right) - (\theta-\ell_0) =: 2d_1-(\theta-\ell_0).
\]
Clearly, to prove the proposition, we need to show that $d_1\in\sof^{1-1/q}_q$ and $d_1(\pm1)=0$.

Let $\Gamma$ be any infinitely smooth curve containing $\Delta$. Assume also that the inner domain of $\Gamma$, say $\Omega$, lies to the left of $\Delta$, i.e., $\Delta^+$ is accessible from within $\Omega$. Define $\theta_{e|\Delta}:=\theta-\ell_0$ and $\theta_{e|\Gamma\setminus\Delta}\equiv0$. It is clear that $\theta_e\in\sof^{1-1/p}_p(\Gamma)$. Moreover, since $\theta_e$ is identically zero on $\Gamma\setminus\Delta$, it holds by \eqref{eq:sokhotski} that
\[
d_1 = \sr^+\left(\oc^+\left(\frac{\theta_e}{\sr^+}\right)-\ell_2\right) = \sr\left(\oc_\Gamma^+\left(\frac{\theta_e}{\sr}\right)-\ell_2\right),
\]
where from now on we agree that $\sr_{|\Gamma}$ is the trace of $\sr$ from within $\Omega$, i.e. it is equal to $\sr^+$ on $\Delta$. Furthermore, we can regard $d_1$ as a function holomorphic in $\Omega$. Thus, by Theorem~\hyperref[thm:T1]{T1}, to show that $d_{1|\Delta}\in\sof^{1-1/q}_q$ it is enough to prove that $d_1\in\sof^1_q(\Omega)$, $q\in(2,p)$. As $d_1$ is holomorphic in $\Omega$, it is, in fact, sufficient to get that $\partial d_1\in\lf^q(\Omega)$.

Now, Proposition~\ref{pr:smooth1} insures that there exists $\Theta\in\sof^1_p(\Omega)$ such that $\Theta_{|\Gamma}=\theta_e$. Observe that $\Theta/\sr\in\sof^1_s(\Omega)$ for any $s\in[1,\frac{4p}{p+4})$ since $\partial\Theta/\sr,\bar\partial\Theta/\sr\in\lf^s(\Omega)$ by H\"older inequality and $\Theta/\sr^3\in\lf^s(\Omega)$ by the estimate
\[
\left|\frac{z\Theta(z)}{\sr^3(z)}\right| \leq \const|z^2-1|^{-2/p-1/2}, \quad z\in\overline\Omega,
\]
where we used the definition of $\Theta$ and \eqref{eq:sobimb}. Thus, Cauchy-Green formula \eqref{eq:cthnaf} applied to $\Theta/\sr$ implies that
\[
\partial d_1 = \partial\left(\Theta - \sr\ok\left(\frac{\bar\partial\Theta}{\sr}\right) - \sr\ell_2\right) = \partial\Theta - \partial\left(\sr\ok\left(\frac{\bar\partial\Theta}{\sr}\right)\right) - \sr\ell_2^\prime - \frac{z\ell_2}{\sr}
\]
a.e. in $\Omega$, where we used that $\bar\partial\sr=0$. Since $\partial\Theta\in\lf^p(\Omega)$ and $\sr\ell_2^\prime$ is bounded, it is only necessary to show that
\begin{eqnarray}
\partial\left(\sr\ok\left(\frac{\bar\partial\Theta}{\sr}\right)\right)- \frac{z\ell_2}{\sr} &=& \frac{z}{\sr}\left(\ok\left(\frac{\bar\partial\Theta}{\sr}\right)-\ell_2\right) + \sr\partial\ok\left(\frac{\bar\partial\Theta}{\sr}\right) \nonumber \\
{} &=& \frac{z}{\sr}\left(\ok\left(\frac{\bar\partial\Theta}{\sr}\right)-\ell_2\right) + \sr\ob\left(\frac{\bar\partial\Theta}{\sr}\right) \nonumber
\end{eqnarray}
belongs to $\lf^q(\Omega)$, $q\in(2,p)$, where we used \eqref{eq:okreverse}. The fact that $\sr\ob(\bar\partial\Theta/\sr)\in\lf^p(\Omega)$ follows from \eqref{eq:beurlinglp}. Now, to show that $(1/\sr)\left(\ok(\bar\partial\Theta/\sr)-\ell_2\right)\in\lf^q(\Omega)$, $q\in(2,p)$, recall that $\bar\partial\Theta/\sr\in\lf^s(\Omega)$ for any $s\in[1,\frac{4p}{p+4})$. So, as mentioned after \eqref{eq:beurling}, $\ok(\bar\partial\Theta/\sr)\in\lf^{2s/2-s}(\Omega)$ when $p\in(2,4]$, i.e., $s\in(\frac43,2)$; and $\ok(\bar\partial\Theta/\sr)\in\cf^{0,1-2/s}(\Omega)$ when $p\in(4,\infty)$, i.e., $s$ can be chosen to lie in $\left(2,\frac{4p}{p+4}\right)$. In the first case, we get that $(1/\sr)\ok(\bar\partial\Theta/\sr)\in\lf^q(\Omega)$, $q\in(2,p)$, simply by applying H\"older inequality once more. This shows that $\partial d_1\in\lf^q(\Omega)$, $q\in(2,p)$, when $p\in(2,4]$ with $\ell_2\equiv0$. In the second case, let $\ell_2$ be the polynomial interpolating $\ok(\bar\partial\Theta/\sr)$ at $\pm1$, $\deg(\ell_2)\leq1$. Then
\[
\left|\frac{z}{\sr(z)}\left(\ok\left(\frac{\bar\partial\Theta}{\sr}\right)-\ell_2\right)(z)\right| \leq \const|z^2-1|^{(s-4)/2s}, \quad z\in\overline\Omega.
\]
Since $\frac{s-4}{2s}\in\left(-\frac12,-\frac{2}{p}\right)$, it holds that $(1/\sr)\left(\ok(\bar\partial\Theta/\sr)-\ell_2\right)\in\lf^q(\Omega)$, $q\in(2,p)$, which shows that $\partial d_1\in\lf^q(\Omega)$, $q\in(2,p)$, when $p\in(4,\infty)$.

It only remains to show that $d_1(\pm1)=0$. As $\theta_e\in\cf^{1-2/p}(\Gamma)$ by \eqref{eq:simb} and $\theta_e(\pm1)=0$, the function $\oc^+(\theta_e/\sr^+)$ is either bounded near $\pm1$ or blows up there with the order strictly less than $1/2$ \cite[Sec. I.8.4]{Gakhov}, see also \eqref{eq:rnnear1}. Thus, $\sr^+\oc^+(\theta_e/\sr^+)$ vanishes at $\pm1$ and so does $d_1$.
\end{proof}

We continue with the case of functions in $\cf^{m,\varsigma}$.

\begin{prop}
\label{lem:smooth2}
Let $\theta\in\cf^{m,\varsigma}$, $m\in\Z_+$, $\varsigma\in(0,1]$. Then 
\[
\sr^\pm\os(\theta/\sr^+) = \pm d +\sr^\pm\ell, \quad  d^{(k)}(\pm1)=0, \quad k\in\{0,\ldots,m\},
\]
where $\ell$ is a polynomial, $\deg(\ell)\leq2m+1$, and $d\in\cf^{m,\varsigma}$ when $\varsigma\in\left(0,\frac12\right)\cup\left(\frac12,1\right)$, while $d\in\cf^{m,\varsigma-\epsilon}$ for arbitrarily  small $\epsilon>0$ when $\varsigma=\frac12,1$.
\end{prop}

When $\varsigma\in\left(0,\frac12\right)$, the conclusion of the proposition follows from \cite[Thm. 3]{DudSp00}. Therefore, we are required\footnote{The authors were surprised not to find this case in the literature.} to prove Proposition~\ref{lem:smooth2} only for $\varsigma\in\left[\frac12,1\right]$. To this end, we shall need several geometrical lemmas. In all of them we assume that $\Gamma$ is as in Proposition \ref{lem:smooth1} and we omit superscript $+$ for $\sr$ when dealing with the values of $\sr$ on $\Delta$. By $C_\Gamma$ we shall denote a constant such that $|\tau|\leq C_\Gamma$, $\tau\in\Gamma$. Moreover, $\tau_1$ and $\tau_2$ will stand for two different points on $\Gamma$ satisfying
\begin{equation}
\label{eq:cond2}
|1-\tau_1^2|\leq|1-\tau_2^2|.
\end{equation}

\begin{lem}
\label{lem:smooth21}
Let $N\in\N$ or $N=-1$. Then
\begin{equation}
\label{eq:geom3}
\left|\frac{1}{\sr^N(\tau_1)}-\frac{1}{\sr^N(\tau_2)}\right| \leq C_1|N|\max\left\{\frac{1}{|1-\tau_1^2|^{1+N/2}},\frac{1}{|1-\tau_2^2|^{1+N/2}}\right\} |\tau_1-\tau_2|,
\end{equation}
where $C_1$ is a constant depending only on $\Gamma$.
\end{lem}
\begin{proof}
If $N$ is an even integer, then
\begin{eqnarray}
\left|\frac{1}{\sr^N(\tau_1)}-\frac{1}{\sr^N(\tau_2)}\right| &=& \frac{1}{|1-\tau_1^2|^{N/2}}\left|1-\left(\frac{1-\tau_1^2}{1-\tau_2^2}\right)^{N/2}\right| \nonumber \\
{} &=& \frac{|\tau_1^2-\tau_2^2|}{|1-\tau_1^2|^{N/2}|1-\tau_2^2|} \left|\sum_{j=0}^{N/2-1}\left(\frac{1-\tau_1^2}{1-\tau_2^2}\right)^j\right| \leq \frac{NC_\Gamma|\tau_1-\tau_2|}{|1-\tau_1^2|^{1+N/2}} \nonumber
\end{eqnarray}
by \eqref{eq:cond2}.  If $N$ is an odd integer, then
\begin{eqnarray}
\left|\frac{1}{\sr^N(\tau_1)}-\frac{1}{\sr^N(\tau_2)}\right| &=& \frac{1}{|1-\tau_1^2|^{N/2}|1-\tau_2^2|^{1/2}}\left|\sr(\tau_2)-\sr(\tau_1)\left(\frac{1-\tau_1^2}{1-\tau_2^2}\right)^{(N-1)/2}\right| \nonumber \\
{} &=& \frac{|\sr(\tau_1)-\sr(\tau_2)|}{|1-\tau_1^2|^{N/2}|1-\tau_2^2|^{1/2}}  +  \frac{(N-1)C_\Gamma|\tau_1-\tau_2|}{|1-\tau_1^2|^{N/2}|1-\tau_2^2|} \nonumber
\end{eqnarray}
by the first estimate and \eqref{eq:cond2}. Clearly, it only remains to prove the lemma for $N=-1$. It can be readily verified that it is enough to consider $|\tau_1-\tau_2|$ small enough. As $\sr$ is zero free in $\Omega$, interior of $\Gamma$, an argument function, say $a$, is well-defined and continuous in $\Omega$.  Since $\sr$ extends holomorphically across $\Delta^\circ$ and $\Gamma\setminus\Delta$, the trace of $a$ is uniformly continuous on any compact subset of $\Gamma\setminus\{\pm1\}$. Moreover, it also has one-sided limits at $\pm1$ with the jumps of magnitude $\pi/2$. Thus, there exists $\delta>0$ such that for all $|\tau_1-\tau_2|<\delta$ it holds that $|a(\tau_1)-a(\tau_2)|<\frac{2\pi}{3}$. Then 
\[
|\sr(\tau_1)+\sr(\tau_2)| \geq |\sr(\tau_2)|-\frac12|\sr(\tau_1)| \geq \frac12|\sr(\tau_2)|
\]
for $|\tau_1-\tau_2|<\delta$ by \eqref{eq:cond2} for the last inequality and therefore
\[
|\sr(\tau_1)-\sr(\tau_2)| \leq 2C_\Gamma\frac{|\tau_1-\tau_2|}{|\sr(\tau_1)+\sr(\tau_2)|} \leq 4C_\Gamma\frac{|\tau_1-\tau_2|}{|1-\tau_2^2|^{1/2}},
\]
which finishes the proof of the lemma.
\end{proof}

\begin{lem}
\label{lem:smooth22}
Let  $\varrho\in\cf^{0,\varsigma}(\Gamma)$, $\varsigma\in\left(\frac12,1\right)$, and $\varrho(\pm1)=0$. Then $(\varrho/\sr)\in\cf^{0,\varsigma-1/2}(\Gamma)$ and
\begin{equation}
\label{eq:claim1}
\left|\frac{\varrho(\tau_1)}{\sr(\tau_1)}-\frac{\varrho(\tau_2)}{\sr(\tau_2)}\right| \leq C_2 \min\left\{\frac{1}{|1-\tau_1^2|^{1/2}},\frac{1}{|1-\tau_2^2|^{1/2}}\right\} |\tau_1-\tau_2|^{\varsigma},
\end{equation}
where $C_2$ is a constant depending only on $\Gamma$.
\end{lem}
\begin{proof}
We start by proving \eqref{eq:claim1}. By the condition of the lemma it holds that
\begin{equation}
\label{eq:cond1}
\left\{
\begin{array}{l}
|\varrho(\tau_1)-\varrho(\tau_2)| \leq M|\tau_1-\tau_2|^\varsigma \smallskip \\ 
|\varrho(\tau)| \leq M|1-\tau^2|^\varsigma, 
\end{array}
\right.
\quad \tau_1,\tau_2,\tau\in\Gamma,
\end{equation}
for some finite constant $M$. Set, for brevity, $\kappa:=\varrho/\sr$. First, let $\tau_1=1$. Observe that
\[
|\kappa(\tau)| \leq M|1-\tau^2|^{\varsigma-1/2}  \leq M(1+C_\Gamma)\frac{|1-\tau|^\varsigma}{|1-\tau^2|^{1/2}}
\]
by \eqref{eq:cond1}. Thus, $\kappa(1)=0$ by continuity and it holds that
\begin{equation}
\label{eq:h-1}
|\kappa(\tau_1)-\kappa(\tau_2)| \leq  M(1+C_\Gamma)\frac{|\tau_1-\tau_2|^\varsigma}{|1-\tau_2^2|^{1/2}}.
\end{equation}
Clearly, an analogous bound holds when $\tau_1=-1$.

Second, let $|\tau_1-\tau_2|\geq|1-\tau_1^2|$. In this case, we also have that
\begin{equation}
\label{eq:geom-1}
|1-\tau_2^2| \leq |\tau_1^2-\tau_2^2| + |1-\tau_1^2| \leq C_\Gamma^*|\tau_1-\tau_2|, \quad C_\Gamma^*:=1+2C_\Gamma.
\end{equation}
Then it follows from \eqref{eq:cond1}, \eqref{eq:cond2}, and \eqref{eq:geom-1} that
\begin{eqnarray}
|\kappa(\tau_1)-\kappa(\tau_2)| &\leq& |\kappa(\tau_1)|+|\kappa(\tau_2)| \leq M\left(|1-\tau_1^2|^{\varsigma-1/2} + |1-\tau_2^2|^{\varsigma-1/2}\right) \nonumber \smallskip \\
\label{eq:h-2}
{} &\leq& 2M|1-\tau_2^2|^{\varsigma-1/2} \leq 2M(C_\Gamma^*)^\varsigma\frac{|\tau_1-\tau_2|^\varsigma}{|1-\tau_2^2|^{1/2}}.
\end{eqnarray}

Third, let $|\tau_1-\tau_2|\leq|1-\tau_1^2|$. Then, it also holds that
\begin{equation}
\label{eq:geom2}
|1-\tau_2^2| \leq |1-\tau_1^2| + |\tau_1^2-\tau_2^2| \leq C_\Gamma^*|1-\tau_1^2|.
\end{equation}
Thus, \eqref{eq:cond1} and Lemma \ref{lem:smooth21} imply that
\begin{eqnarray}
|\kappa(\tau_1)-\kappa(\tau_2)| &\leq& |\varrho(\tau_1)|\left|\frac{1}{\sr(\tau_1)}-\frac{1}{\sr(\tau_2)}\right| + \frac{|\varrho(\tau_1)-\varrho(\tau_2)|}{|\sr(\tau_2)|} \nonumber \\
{} &\leq& M|1-\tau_1^2|^\varsigma C_1\frac{|\tau_1-\tau_2|}{|1-\tau_1^2|^{3/2}} + M\frac{|\tau_1-\tau_2|^\varsigma}{|1-\tau_2^2|^{1/2}}. \nonumber
\end{eqnarray}
Using the conditions $|\tau_1-\tau_2|\leq|1-\tau_1^2|$ and \eqref{eq:geom2}, we obtain that
\begin{eqnarray}
|\kappa(\tau_1)-\kappa(\tau_2)| &\leq& M|\tau_1-\tau_2|^\varsigma\left(C_1\frac{|\tau_1-\tau_2|^{1-\varsigma}}{|1-\tau_1^2|^{1-\varsigma}}\frac{1}{|1-\tau_1^2|^{1/2}}+\frac{1}{|1-\tau_2^2|^{1/2}}\right) \nonumber \smallskip \\
\label{eq:h-3}
{} &\leq& M(C_1\sqrt{C_\Gamma^*}+1)\frac{|\tau_1-\tau_2|^\varsigma}{|1-\tau_2^2|^{1/2}}.
\end{eqnarray}
This finishes the proof of \eqref{eq:claim1}.

Finally, it can be readily verified that the equations leading to \eqref{eq:h-1}, \eqref{eq:h-2}, and \eqref{eq:h-3} also yield that $\kappa\in\cf^{0,\varsigma-1/2}(\Gamma)$ and hence, the lemma is proved.
\end{proof}

\begin{lem}
\label{lem:smooth23}
Let $\varrho$ be as in Lemma \ref{lem:smooth22}. Then $\os_\Gamma(\varrho/\sr)\in\cf^{0,\varsigma-1/2}(\Gamma)$ and
\begin{equation}
\label{eq:claim2}
\left|\os_\Gamma\left(\frac{\varrho}{\sr}\right)(\tau_1)-\os_\Gamma\left(\frac{\varrho}{\sr}\right)(\tau_2)\right| \leq C_3\min\left\{\frac{1}{|1-\tau_1^2|^{1/2}},\frac{1}{|1-\tau_2^2|^{1/2}}\right\} |\tau_1-\tau_2|^{\varsigma},
\end{equation}
where $C_3$ is a constant depending only on $\Gamma$.
\end{lem}
\begin{proof}
Since $\kappa:=\varrho/\sr\in\cf^{0,\varsigma-1/2}(\Gamma)$, $\os_\Gamma\kappa\in\cf^{0,\varsigma-1/2}(\Gamma)$ as well by \cite[Sec. I.5.1]{Gakhov}. To prove \eqref{eq:claim2}, one needs to trace the local character of the proof in \cite[Sec. I.5.1]{Gakhov}. This is a tedious job but the authors felt compelled to carry it out for the reader.

Define
\[
\os(\tau) := \os_\Gamma\kappa(\tau) - \kappa(\tau) = \frac{1}{\pi i}\int_\Gamma\frac{\kappa(t)-\kappa(\tau)}{t-\tau}dt, \quad \tau\in\Gamma.
\]
In the light of \eqref{eq:claim1}, it is enough to show \eqref{eq:claim2} with $\os_\Gamma\kappa$ replaced by $\os$. Observe also that the integral that defines $\os$ is no longer singular as $\kappa(\pm1)=0$.

Denote by $\Gamma^*$ the connected component of $\Gamma\cap\{\tau:|\tau_1-\tau|\leq 2|\tau_1-\tau_2|\}$ that contains $\tau_1$. we order $\tau_1$ and $\tau_2$  so that \eqref{eq:cond2} holds. Then we can write
\begin{eqnarray}
\os(\tau_2)-\os(\tau_1) &=& \frac{1}{\pi i}\int_{\Gamma^*}\frac{\kappa(\tau)-\kappa(\tau_2)}{\tau-\tau_2}d\tau - \frac{1}{\pi i}\int_{\Gamma^*}\frac{\kappa(\tau)-\kappa(\tau_1)}{\tau-\tau_1}d\tau \nonumber \\
{} && + \frac{1}{\pi i}\int_{\Gamma\setminus\Gamma^*}\frac{\kappa(\tau_1)-\kappa(\tau_2)}{\tau-\tau_1}d\tau + \frac{1}{\pi i}\int_{\Gamma\setminus\Gamma^*}\frac{(\tau_2-\tau_1)(\kappa(\tau)-\kappa(\tau_2))}{(\tau-\tau_1)(\tau-\tau_2)}d\tau \nonumber \smallskip \\
{} &=& I_1 + I_2 + I_3 + I_4. \nonumber
\end{eqnarray}
Before we continue, observe that there exists a finite constant $M$ such that
\begin{equation}
\label{eq:cond4}
|\Gamma(t,\tau)|\leq M|t-\tau|, \quad t,\tau\in\Gamma,
\end{equation}
since $\Gamma$ is a smooth Jordan curve, where $|\Gamma(t,\tau)|$ is the arclength of the smallest subarc of $\Gamma$ connecting $t$ and $\tau$. 

First, let us estimate $I_1$. We get from \eqref{eq:claim1} and \eqref{eq:cond4} that
\begin{eqnarray}
|I_1| &\leq& \frac{C_2}{\pi}\sup_{\tau\in\Gamma^*}\min\left\{\frac{1}{|1-\tau^2|^{1/2}},\frac{1}{|1-\tau_2^2|^{1/2}}\right\}\int_{\Gamma^*}\frac{|d\tau|}{|\tau-\tau_2|^{1-\varsigma}}  \nonumber \\
\label{eq:h-4}
{} &\leq& \frac{C_2M^{1-\varsigma}}{\pi|1-\tau_2^2|^{1/2}} \int_0^{4M|\tau_1-\tau_2|}\frac{ds}{s^{1-\varsigma}} \leq \frac{C_24^\varsigma M}{\varsigma\pi}\frac{|\tau_1-\tau_2|^\varsigma}{|1-\tau_2^2|^{1/2}}.
\end{eqnarray}
Clearly, an analogous estimate can be made for $I_2$.

Now, we shall estimate $I_3$. It holds that
\[
|I_3| = \left|(\kappa(\tau_1)-\kappa(\tau_2))\log\left(\frac{\tau_a-\tau_1}{\tau_b-\tau_1}\right)\right|,
\]
where $\tau_a$ and $\tau_b$ are the endpoints of $\Gamma^*$. As $|\tau_a-\tau_1|=|\tau_b-\tau_1|$, we have that
\begin{equation}
\label{eq:h-5}
|I_3| \leq \const\frac{|\tau_1-\tau_2|^\varsigma}{|1-\tau_2^2|^{1/2}},
\end{equation}
where $\const$ is the product of $C_2$ and the maximum of the argument of $\frac{\tau_a-\tau_1}{\tau_b-\tau_1}$ for all possible choices of $\tau_a$ and $\tau_b$.

Finally, let us estimate $I_4$. Observe that $|\tau-\tau_1|\leq2|\tau-\tau_2|$, $\tau\in\Gamma\setminus\Gamma^*$, since
\[
|\tau-\tau_1| \leq |\tau-\tau_2| + |\tau_1-\tau_2| \leq |\tau-\tau_2| + \frac12|\tau-\tau_1|.
\]
Then we get from \eqref{eq:claim1} and the bound above that
\begin{eqnarray}
|I_4| &\leq& \frac{C_2}{\pi}\max_{\tau\in\Gamma\setminus\Gamma^*}\min\left\{\frac{1}{|1-\tau^2|^{1/2}},\frac{1}{|1-\tau_2^2|^{1/2}}\right\} \int_{\Gamma\setminus\Gamma^*}\frac{|\tau_2-\tau_1||d\tau|}{|\tau-\tau_1||\tau-\tau_2|^{1-\varsigma}}  \nonumber \\
{} &=& \frac{C_2}{\pi}\frac{|\tau_2-\tau_1|}{|1-\tau_2^2|^{1/2}}\int_{\Gamma\setminus\Gamma^*}\left|\frac{\tau-\tau_1}{\tau-\tau_2}\right|^{1-\varsigma}\frac{|d\tau|}{|\tau-\tau_1|^{2-\varsigma}} \nonumber \\
\label{eq:h-6}
{} &\leq& \frac{C_22^{1-\varsigma}M^{2-\varsigma}}{\pi}\frac{|\tau_2-\tau_1|}{|1-\tau_2^2|^{1/2}}\int_{2|\tau_1-\tau_2|}^\infty\frac{ds}{s^{2-\varsigma}} = \frac{C_2M^{2-\varsigma}}{(1-\varsigma)\pi}\frac{|\tau_2-\tau_1|^\varsigma}{|1-\tau_2^2|^{1/2}}.
\end{eqnarray}
Combining \eqref{eq:h-4}, \eqref{eq:h-5}, and \eqref{eq:h-6} with \eqref{eq:claim1}, we see that \eqref{eq:claim2} holds.
\end{proof}

\begin{lem}
\label{lem:smooth24}
Let $\varrho\in\cf^{0,\upsilon}$, $\upsilon\in\left(0,\frac12\right)$, $\varrho(\pm1)=0$, be such that
\[
\left|\varrho(\tau_1)-\varrho(\tau_2)\right| \leq C_4\min\left\{\frac{1}{|1-\tau_1^2|^{1/2}},\frac{1}{|1-\tau_2^2|^{1/2}}\right\} |\tau_1-\tau_2|^{\upsilon+1/2},
\]
$\tau_1,\tau_2\in\Gamma$, $\tau_1\neq\tau_2$, where $C_4$ is a constant depending only on $\Gamma$. Then $\sr\varrho\in\cf^{0,\upsilon+1/2}(\Gamma)$. Further, let $\varrho\in\cf^{N,\upsilon}(\Gamma)$, $N\in\N$, $\varrho^{(j)}(\pm1)=0$, $j\in\{0,\ldots,N\}$. Then
\begin{equation}
\label{eq:claim3}
\left\{
\begin{array}{ll}
(\varrho/\sr^{2N-1})\in\cf^{0,\upsilon+1/2}(\Gamma) & \mbox{if } \upsilon\in\left(0,\frac12\right), \smallskip \\
(\varrho/\sr^{2N+1})\in\cf^{0,\upsilon-1/2}(\Gamma) & \mbox{if } \upsilon\in\left(\frac12,1\right).
\end{array}
\right.
\end{equation}
\end{lem}
\begin{proof}
To verify the first claim, assume first that $|\tau_1-\tau_2|\geq|1-\tau_1^2|$. Since $\varrho\in\cf^{0,\upsilon}(\Gamma)$ and vanishes at $\pm1$, it holds for some finite constant $M$ that
\begin{equation}
\label{eq:cond3}
|\varrho(\tau)| \leq M|1-\tau^2|^\upsilon.
\end{equation}
Then we get from \eqref{eq:geom-1} and the inequality above that
\begin{equation}
\label{eq:h-7}
|(\sr\varrho)(\tau_1)-(\sr\varrho)(\tau_2)| \leq M\left(|1-\tau_1^2|^\varsigma+|1-\tau_2^2|^\varsigma\right) \leq M(1+C_\Gamma^*)|\tau_1-\tau_2|^\varsigma.
\end{equation}
Assume now that $|\tau_1-\tau_2|\leq|1-\tau_1^2|$. Then we get by \eqref{eq:cond3}, \eqref{eq:geom3}, and the conditions of the lemma that
\begin{eqnarray}
|(\sr\varrho)(\tau_1)-(\sr\varrho)(\tau_2)| &\leq & |\sr(\tau_2)|\left|\varrho(\tau_1)-\varrho(\tau_2)\right| + |\varrho(\tau_1)||\sr(\tau_1)-\sr(\tau_2)| \nonumber \\
{} &\leq& |1-\tau_2^2|^{1/2}C_4\frac{|\tau_1-\tau_2|^{\upsilon+1/2}}{|1-\tau_2^2|^{1/2}} + M|1-\tau_1^2|^\upsilon C_1\frac{|\tau_1-\tau_2|}{|1-\tau_1^2|^{1/2}} \nonumber \\
\label{eq:h-8}
{} &\leq& (C_4+C_1M)|\tau_1-\tau_2|^{\upsilon+1/2}.
\end{eqnarray}
Equations \eqref{eq:h-7} and \eqref{eq:h-8} show that $(\sr\varrho)\in\cf^{0,\upsilon+1/2}(\Gamma)$.

It remains to prove \eqref{eq:claim3}. Suppose first that $\upsilon\in\left(0,\frac12\right)$. Then, by the assumptions on $\varrho$, it holds for some finite constant $M$ that
\begin{equation}
\label{eq:cond5}
\left\{
\begin{array}{l}
\left|\varrho(\tau_2)-\sum_{j=1}^N\varrho^{(j)}(\tau_1)(\tau_2-\tau_1)^j\right| \leq M|\tau_2-\tau_1|^{N+\upsilon}, \smallskip \\
\left|\varrho^{(j)}(\tau)\right| \leq M|1-\tau^2|^{N-j+\upsilon},
\end{array}
\right.
\quad \tau_1,\tau_2,\tau\in\Gamma.
\end{equation}
Thus, for $|\tau_1-\tau_2|\geq|1-\tau_1^2|$, we have from \eqref{eq:cond5} and \eqref{eq:geom-1} that
\[
\left|\frac{\varrho(\tau_1)}{\sr^{2N-1}(\tau_1)}-\frac{\varrho(\tau_2)}{\sr^{2N-1}(\tau_2)}\right| \leq \left|\frac{\varrho(\tau_1)}{\sr^{2N-1}(\tau_1)}\right|+\left|\frac{\varrho(\tau_2)}{\sr^{2N-1}(\tau_2)}\right| \leq M(1+(C_\Gamma^*)^\varsigma)|\tau_1-\tau_2|^{\upsilon+1/2}.
\]
For $|\tau_1-\tau_2|\leq|1-\tau_1^2|$, we have from \eqref{eq:cond5} that
\begin{eqnarray}
\left|\frac{\varrho(\tau_1)}{\sr^{2N-1}(\tau_1)}-\frac{\varrho(\tau_2)}{\sr^{2N-1}(\tau_2)}\right| &=& \left|\frac{\varrho(\tau_1)}{\sr^{2N-1}(\tau_1)}  \pm \sum_{j=0}^N\frac{\varrho^{(j)}(\tau_1)(\tau_2-\tau_1)^j}{\sr^{2N-1}(\tau_2)} - \frac{\varrho(\tau_2)}{\sr^{2N-1}(\tau_2)}\right|\nonumber \\
{} &\leq& \left|\frac{\varrho(\tau_1)}{\sr^{2N-1}(\tau_1)}  - \sum_{j=0}^N\frac{\varrho^{(j)}(\tau_1)(\tau_2-\tau_1)^j}{\sr^{2N-1}(\tau_2)}\right| +\frac{M|\tau_2-\tau_1|^{N+\upsilon}}{|1-\tau_2^2|^{N-1/2}} \nonumber \\
{} &\leq& I + M|\tau_2-\tau_1|^{\upsilon+1/2}. \nonumber
\end{eqnarray}
Furthermore, it holds by \eqref{eq:geom3} and \eqref{eq:cond5} that
\begin{eqnarray}
I &\leq& \sum_{j=0}^N\left|\varrho^{(j)}(\tau_1)(\tau_2-\tau_1)^j\left(\frac{1}{\sr^{2N-1}(\tau_1)}-\frac{1}{\sr^{2N-1}(\tau_2)}\right)\right| + \sum_{j=1}^N\left|\frac{\varrho^{(j)}(\tau_1)(\tau_2-\tau_1)^j}{\sr^{2N-1}(\tau_1)}\right| \nonumber \\
{} &\leq& (2N-1)C_1M \sum_{j=0}^N\frac{|1-\tau_1^2|^{N-j+\upsilon}|\tau_2-\tau_1|^{j+1}}{|1-\tau_1^2|^{N+1/2}} + M\sum_{j=1}^k\frac{|1-\tau_1^2|^{N-j+\upsilon}|\tau_2-\tau_1|^j}{|1-\tau_1^2|^{N-1/2}}  \nonumber \\
{} &\leq& 2NC_1M\sum_{j=1}^{N+1}\left|\frac{\tau_2-\tau_1}{1-\tau_1^2}\right|^{j-\upsilon-1/2} |\tau_2-\tau_1|^{\upsilon+1/2} \leq 2N^2C_1M|\tau_2-\tau_1|^{\upsilon+1/2}. \nonumber
\end{eqnarray}
Clearly, the case $\upsilon\in\left(\frac12,1\right)$ can be handled in a similar fashion.
\end{proof}

\begin{proof}[Proof of Proposition \ref{lem:smooth2}]
Clearly, we need to prove the proposition only for $\varsigma\neq\frac12,1$ as these two cases follow from the obvious inclusion $\cf^{m,\varsigma-\epsilon}\subset\cf^{m,\varsigma}$.

Let $\ell_0$, $\deg(\ell_0)\leq2m+1$, be the polynomial interpolating $\theta$ at $\pm1$ up to and including the order $m$. Throughout the proof we assume that $\Gamma$ is as in Proposition \ref{lem:smooth1} and that $\theta$ is extended to $\Gamma\setminus\Delta$ by $\ell_0$. Clearly, this implies that $\theta\in\cf^{m,\varsigma}(\Gamma)$. As $\ell_1:=\os(\ell_0/\sr^+)$ is a polynomial of degree $2m$, we have that
\[
\os\left(\frac{\theta}{\sr}\right) = \os\left(\frac{\theta-\ell_0}{\sr}\right) + \ell_1 = \os_\Gamma\left(\frac{\theta-\ell_0}{\sr}\right)_{|\Delta} + \ell_1.
\]

Assume first that $m=0$. Then $\theta-\ell_0$ satisfies the conditions of $\varrho$ in Lemma \ref{lem:smooth22} and therefore Lemma \ref{lem:smooth23} holds with $\theta-\ell_0$ in place of $\varrho$. Let $\ell_2$ be the polynomial interpolating $\os_\Gamma((\theta-\ell_0)/\sr)$ at $\pm1$ and set 
\begin{equation}
\label{eq:setd}
d:= \sr\left(\os_\Gamma\left(\frac{\theta-\ell_0}{\sr}\right)-\ell_2\right) \quad \mbox{and} \quad \ell:=\ell_1+\ell_2.
\end{equation}
Clearly, $d(\pm1)=0$. Then the conclusion of the proposition follows from Lemma \ref{lem:smooth24} applied with $\varrho=\os_\Gamma((\theta-\ell_0)/\sr)-\ell_2$.

Assume now that $m\in\N$. Since the derivative of a singular integral is the singular integral of the derivative \cite[Sec. I.4.4]{Gakhov}, observe that
\[
\os_\Gamma^{(m)}\left(\frac{\theta-\ell_0}{\sr}\right) = \os_\Gamma\left(\left(\frac{\theta-\ell_0}{\sr}\right)^{(m)}\right) =\sum_{j=0}^m\binom{m}{j}\os_\Gamma\left(\frac{v_j(\theta-\ell_0)^{(m-j)}}{\sr^{2j+1}}\right),
\]
where $v_j$ are polynomials. Then $\os_\Gamma((\theta-\ell_0)/\sr)\in\cf^{m,\varsigma-1/2}(\Gamma)$ by \eqref{eq:claim3}, applied with $\varrho=v_j(\theta-\ell_0)^{(m-j)}$, $N=j$, and $\upsilon=\varsigma$, and the fact that singular integrals preserve H\"older smoothness \cite[Sec. I.5.1]{Gakhov}. Thus, $\os_\Gamma((\theta-\ell_0)/\sr)$ has $m$ continuous derivatives on $\Gamma$. Let then $d$ and $\ell$ be defined by \eqref{eq:setd}, where $\ell_2$ is the polynomial interpolating $\os_\Gamma((\theta-\ell_0)/\sr)$ at $\pm1$ up to and including the order $m$. Once more, since singular integral commutes with differentiation, we get
\begin{eqnarray}
d^{(m)} &=& \sum_{k=0}^m\binom{m}{k}\frac{u_k}{\sr^{2k-1}}\left(\os_\Gamma^{(m-k)}\left(\frac{\theta-\ell_0}{\sr}\right)-\ell_2^{(m-k)}\right) \nonumber \\
{} &=& \sum_{k=0}^m \sum_{j=0}^{m-k} \binom{m}{k} \binom{m-k}{j}\frac{u_k}{\sr^{2k-1}}\left(\os_\Gamma\left(\frac{v_j(\theta-\ell_0)^{(m-k-j)}}{\sr^{2j+1}}\right)-\ell_{j,k}\right), \nonumber
\end{eqnarray}
where $u_k$ are polynomials and the polynomials $\ell_{j,k}$ interpolate the corresponding term in the parenthesis. Then
\[
\os_\Gamma^{(k)}\left(\frac{v_j(\theta-\ell_0)^{(m-k-j)}}{\sr^{2j+1}}\right) = \sum_{l=0}^k\binom{k}{l}\os_\Gamma\left(\frac{v_{j+l}(\theta-\ell_0)^{(m-j-l)}}{\sr^{2(j+l)+1}}\right) \in \cf^{0,\varsigma-1/2}(\Gamma),
\]
by \eqref{eq:claim3}, applied with $\varrho=v_{j+l}(\theta-\ell_0)^{(m-j-l)}$, $N=j+l$, $\upsilon=\varsigma$, and the fact that singular integrals preserve H\"older smoothness. Thus, \eqref{eq:claim3} applied once more, now with $\varrho=\os_\Gamma(v_j(\theta-\ell_0)^{(m-k-j)}/\sr^{2j+1})-\ell_{j,k}$, $N=k$, and, $\upsilon=\varsigma-1/2$, yields that $d^{(m)}\in\cf^{0,\varsigma}(\Gamma)$, which finishes the proof of the lemma.
\end{proof}

\section{Riemann-Hilbert-$\bar\partial$ Problem}
\label{sec:rhpbp}

In what follows, we adopt the notation $\phi^{m\sigma_3}$ for the diagonal matrix $\left(\begin{array}{cc} \phi^m&0\\0&\phi^{-m} \end{array}\right)$, where $\phi$ is a function, $m$ is a constant, and $\sigma_3$ is the Pauli matrix $\displaystyle \sigma_3 = \left(\begin{array}{cc} 1&0\\0&-1 \end{array}\right)$.

\subsection{Initial Riemann-Hilbert Problem} 

Let $\mathscr{Y}$ be a $2\times2$ matrix function and $w_n$ be given by \eqref{eq:varyingweight}. Consider the following Riemann-Hilbert problem for $\mathscr{Y}$ (\rhy):
\begin{itemize}
\item[(a)] $\mathscr{Y}$ is analytic in $\C\setminus\Delta$ and $\displaystyle \lim_{z\to\infty} \mathscr{Y}(z)z^{-n\sigma_3} = \mathscr{I}$, where $\mathscr{I}$ is the identity matrix;
\item[(b)] $\mathscr{Y}$ has continuous traces from each side of $\Delta^\circ$, $\mathscr{Y}_\pm$, and $\displaystyle \mathscr{Y}_+ = \mathscr{Y}_- \left(\begin{array}{cc}1&2w_n\\0&1\end{array}\right);$
\item[(c)] $\mathscr{Y}$ has the following behavior near $z=1$:
\[
\mathscr{Y} = \left\{
\begin{array}{ll}
\displaystyle O\left(\begin{array}{cc}1&|1-z|^\alpha\\1&|1-z|^\alpha\end{array}\right), & \mbox{if} \quad \alpha<0, \smallskip \\
\displaystyle O\left(\begin{array}{cc}1&\log|1-z|\\1&\log|1-z|\end{array}\right), & \mbox{if} \quad \alpha=0, \smallskip \\
\displaystyle O\left(\begin{array}{cc}1&1\\1&1\end{array}\right), & \mbox{if} \quad \alpha>0, \\
\end{array}
\right. \quad \mbox{as} \quad D\ni z\to1;
\]
\item[(d)] $\mathscr{Y}$ has the same behavior when $D\ni z\to-1$ as in (c) only with $\alpha$ replaced by $\beta$ and $1-z$ replaced by $1+z$.
\end{itemize}

The connection between \rhy~ and polynomials orthogonal with respect to $w_n$ was first realized by Fokas, Its, and Kitaev \cite{FIK91,FIK92} and lies in the following.
\begin{lem}
\label{lem:rhy}
Let $q_n$ be a polynomial satisfying orthogonality relations \eqref{eq:ortho} and $R_n$ be the corresponding function of the second kind given by \eqref{eq:skind}. Further, let $q_{n-1}^*$ be a polynomial satisfying
\[
\int_\Delta t^jq_{n-1}^*(t)w_n(t)dt=0, \quad j\in\{0,\ldots,n-2\},
\]
and $R_{n-1}^*=R_{n-1}(q_{n-1}^*;\cdot)$ be the function of the second kind for $q_{n-1}^*$. If a solution of \rhy~ exists then it is unique. Moreover, in this case $\deg(q_n)=n$, $R_{n-1}^*(z)=O(z^{-n})$ as $z\to\infty$, and the solution of \rhy~ is given by
\begin{equation}
\label{eq:y}
\mathscr{Y} = \left(\begin{array}{cc}
q_n & R_n \\
m_nq_{n-1}^* & m_nR_{n-1}^*
\end{array}\right),
\end{equation}
where $m_n$ is a constant such that $m_nR_{n-1}^*(z)=z^{-n}[1+o(1)]$ near infinity. Conversely, if $\deg(q_n)=n$ and $R_{n-1}^*(z)=O(z^{-n})$ as $z\to\infty$, then $\mathscr{Y}$ defined in \eqref{eq:y} solves \rhy~.
\end{lem}
\begin{proof} 
As only the smoothness properties of the function $w_n$ were used in \cite[Lem. 2.3]{KMcLVAV04}, this lemma translates without change to the case of a general closed analytic arc and yields the uniqueness of the solution of \rhy~ whenever the latter exists. 

Suppose now that the solution, say $\mathscr{Y}=[\mathscr{Y}_{jk}]_{j,k=1}^2$, of \rhy~ exists. Then $\mathscr{Y}_{11}=z^n+${ lower order terms} by the normalization in \rhy(a). Moreover, by \rhy(b), $\mathscr{Y}_{11}$ has no jump on $\Delta$ and hence is holomorphic in the whole complex plane. Thus, $\mathscr{Y}_{11}$ is necessarily a polynomial of degree $n$ by Liouville's theorem. Further, since $\mathscr{Y}_{12}=O(z^{-n-1})$ and satisfies \rhy(b), it holds that $\mathscr{Y}_{12}=2\oc(\mathscr{Y}_{11}w_n)$. From the latter, we easily deduce that $\mathscr{Y}_{11}$ satisfies orthogonality relations \eqref{eq:ortho}. Applying the same arguments to the second row of $\mathscr{Y}$, we obtain that $\mathscr{Y}_{21}=q_{n-1}^*$ and $\mathscr{Y}_{22}=m_nR_{n-1}^*$ with $m_n$ well-defined.

Conversely, let $\deg(q_n)=n$ and $R_{n-1}^*(z)=O(z^{-n})$ as $z\to\infty$. Then it can be easily checked by direct examination  of \rhy(a)-(d), using the material in Section~\ref{subsec:secondkind}, that $\mathscr{Y}$, given by \eqref{eq:y}, solves \rhy.
\end{proof}

\subsection{Renormalized Riemann-Hilbert Problem} 

Throughout, unless specified otherwise, we follow the convention $\sqrt z=\sqrt{|z|}\exp\{i\Arg(z)/2\}$, $\Arg(z)\in(-\pi,\pi]$. Set
\begin{equation}
\label{eq:en}
\epsilon_n := \sqrt{\gm_{(v_n/hh_n)}/2} \quad \mbox{and} \quad E_n := \epsilon_n\map^n\szf_{(v_n/hh_n)}.
\end{equation}
Then $E_n$ has continuous boundary values on each side of $\Delta$ that satisfy
\[
E_n^+E_n^- = \frac{v_n}{2hh_n} = \frac{w}{2w_n}
\]
due to \eqref{eq:mappr}, \eqref{eq:szego}, and \eqref{eq:varyingweight}. Further, put
\begin{equation}
\label{eq:cpmcnpm}
c^+:=\szf_{h}^+/\szf_{h}^-, \quad c_n^+ := \szf_{h_n}^+/\szf_{h_n}^-, \quad c^-:=1/c^+, \quad \mbox{and} \quad c_n^- := 1/c_n^+.
\end{equation}
Then we get on account of (\ref{eq:szego}), (\ref{eq:szfpoly}), and the multiplicativity property of the Szeg\H{o} functions that
\begin{equation}
\label{eq:encn}
\frac{E_n^-}{E_n^+} = \frac{\left(\szf_{v_n}\map^n\right)^-}{\left(\szf_{v_n}\map^n\right)^+} \frac{\szf_{hh_n}^+}{\szf_{hh_n}^-} = \frac{v_nc_n^+c^+}{\gm_{v_n}\left(\szf_{v_n}^2\map^{2n}\right)^+} = (r_nc_nc)^+ \quad \mbox{and} \quad \frac{E_n^+}{E_n^-}=(r_nc_nc)^-,
\end{equation}
where we slightly abuse the notation by writing $(r_nc_nc)^\pm$ instead of $r_n^\pm c_n^\pm c^\pm$. Since any Szeg\H{o} function assumes value one at infinity and $\map(z)/2z\to1$ as $z\to\infty$, it holds that $E_n(z)/[\epsilon_n(2z)^n]\to1$ as $z\to\infty$. Then it is a quick computation to check that
\[
\left(E_n^-\right)^{\sigma_3} \left(\begin{array}{cc}1 & 2w_n \\ 0 & 1 \end{array}\right) \left(E_n^+\right)^{-\sigma_3} = \left(\begin{array}{cc} (r_nc_nc)^+ & w \\ 0 & (r_nc_nc)^- \end{array}\right)
\]
and
\[
\lim_{z\to\infty} (2^n\epsilon_n)^{\sigma_3}\mathscr{Y}E_n^{-\sigma_3}(z) = \mathscr{I}.
\]

Suppose now that \rhy~ is solvable and $\mathscr{Y}$ is the solution. Define
\begin{equation}
\label{eq:t}
\mathscr{T} := (2^n\epsilon_n)^{\sigma_3}\mathscr{Y}E_n^{-\sigma_3}.
\end{equation}
Then $\mathscr{T}$ solves the following Riemann-Hilbert problem (\rht):
\begin{itemize}
\item[(a)] $\mathscr{T}$ is analytic in $D$ and $\mathscr{T}(\infty)=\mathscr{I}$;
\item[(b)] $\mathscr{T}$ has continuous traces, $\mathscr{T}_\pm$, on $\Delta^\circ$ and $\displaystyle \mathscr{T}_+ = \mathscr{T}_- \left(\begin{array}{cc} (r_nc_nc)^+ & w \\ 0 & (r_nc_nc)^- \end{array}\right)$;
\item[(c)] $\mathscr{T}$ has the following behavior near $z=1$:
\[
\mathscr{T} = \left\{
\begin{array}{ll}
\displaystyle O\left(\begin{array}{cc}1&|1-z|^\alpha\\1&|1-z|^\alpha\end{array}\right), & \mbox{if} \quad \alpha<0, \smallskip \\
\displaystyle O\left(\begin{array}{cc}1&\log|1-z|\\1&\log|1-z|\end{array}\right), & \mbox{if} \quad \alpha=0, \smallskip \\
\displaystyle O\left(\begin{array}{cc}1&1\\1&1\end{array}\right), & \mbox{if} \quad \alpha>0, \\
\end{array}
\right. \quad \mbox{as} \quad D\ni z\to1;
\]
\item[(d)] $\mathscr{T}$ has the same behavior when $D\ni z\to-1$ as in (c) only with $\alpha$ replaced by $\beta$ and $1-z$ replaced by $1+z$.
\end{itemize}

Trivially, the following lemma holds.

\begin{lem}
\label{lem:rht}
\rht~ is solvable if and only if \rhy~ is solvable. When solutions of \rht~ and \rhy~ exist, they are unique and connected by \eqref{eq:t}.
\end{lem}

\subsection{Opening the Lenses, Contours $\Sigma_{ext}$, $\Sigma_n$, and $\Sigma_n^{md}$}
\label{subsec:lenses}

As is standard in the Riemann-Hilbert approach, the second transformation of ~\rhy~ is based on the following factorization of the jump matrix in \rht(b):
\[
\left(\begin{array}{cc} (r_nc_nc)^+ & w \\ 0 & (r_nc_nc)^-\end{array}\right) = 
\left(\begin{array}{cc} 1 & 0 \\ (r_nc_nc)^-/w & 1 \end{array}\right) \left(\begin{array}{cc} 0 &w \\ -1/w & 0 \end{array}\right) \left(\begin{array}{cc} 1 & 0 \\ (r_nc_nc)^+/w & 1 \end{array}\right),
\]
where we took into account that $r_n^+r_n^-\equiv1$ on $\Delta$. This factorization leads us to consider a new Riemann-Hilbert problem with three jumps on a lens-shaped contour $\Sigma_n$ (see Fig.~\ref{fig:2}). 
\begin{figure}[!ht]
\label{fig:a}
\centering
\includegraphics[scale=.55]{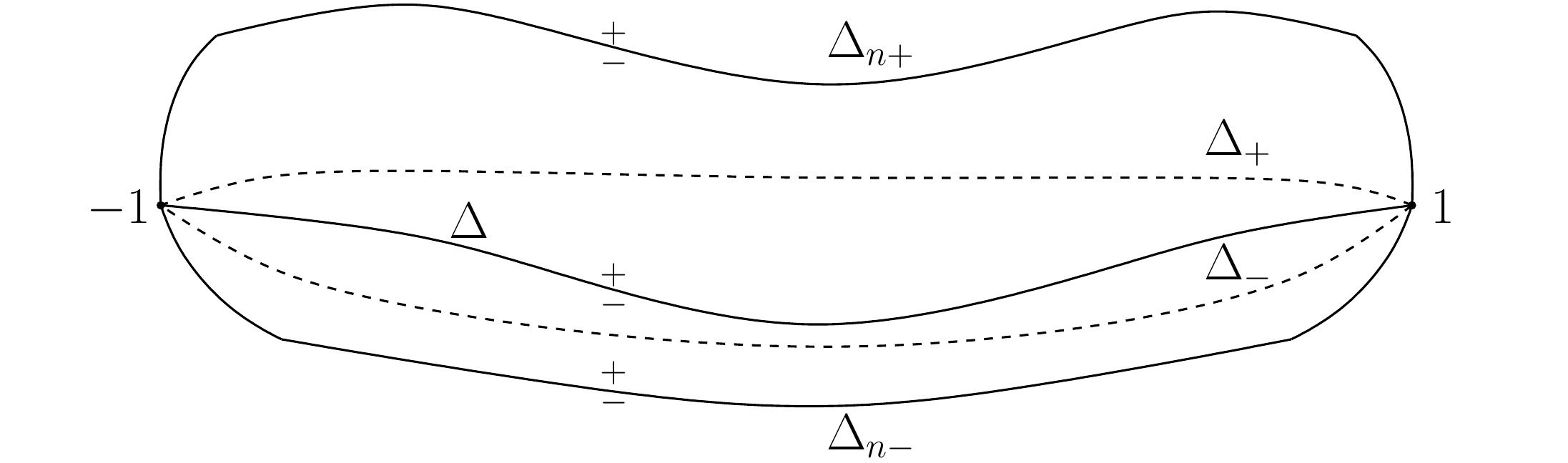}
\caption{\small The contour $\Sigma_n:=\Delta_{n+}\cup\Delta\cup\Delta_{n-}\subset \Xi(D_\Xi)$ (solid lines). The extension contour $\Sigma_{ext}:=\Delta_+\cup\Delta\cup\Delta_-$ (dashed lines and $\Delta$).}
\label{fig:2}
\end{figure}
However, to proceed with such a decomposition, we need to extend $c^\pm$ and $c_n^\pm$ to the complex plane. We shall do it in such a manner that the extended functions, denoted by $c$ and $c_n$, are analytic outside of a fixed lens $\Sigma_{ext}$ (see Fig.~\ref{fig:2}). We postpone this task until the next section and describe here the construction of the lenses $\Sigma_{ext}$ and $\Sigma_n$.

We start from $\Sigma_{ext}$. When $\Delta=[-1,1]$, fix $x>0$ and set $\Delta_+$ to be the subarc of the circle $\{z:~|z-ix|=|x+1|\}$ that lies in the upper half plane. Clearly, $\Delta_+$ joins $-1$ and $1$ and can be made as uniformly close to $[-1,1]$ as we want by taking $x$ sufficiently large. We set $\Delta_-$ to be the reflection of $\Delta_+$ across the real axis. We also denote by $\Omega_+$ and $\Omega_-$ the upper and the lower parts of the lens $\Sigma_{ext}$, i.e., $\Omega_+$ (resp. $\Omega_-$) is a domain bounded by $\Delta_+$ (resp. $\Delta_-$) and $\Delta$. When $\Delta$ is a general closed analytic arc parametrized by $\Xi$, set $\Sigma_{ext}$ to be the image under $\Xi$ of the corresponding lens for $[-1,1]$ (the latter can always be made small enough to lie in $D_\Xi$).

We continue by constructing the lens $\Sigma$, which will we the limit lens for the sequence $\{\Sigma_n\}$. Let $g^2$ and $O_g\subset \Xi(D_\Xi)$ be defined in \eqref{eq:gfun}. Then $1\in O_g$, $g^2(1)=0$, and $g^2$ is conformal in $O_g$. Set $U_\delta:=\{z:|z-1|<\delta\}$. Choose $\delta_0>0$ to be so small that $U_\delta\subset O_g$ and $g^2(U_\delta)$ is convex for any $\delta<\delta_0$. We require the same conditions to be fulfilled by $\delta_0$ and $\widetilde U_\delta:=\{z:|z+1|<\delta\}$ with respect to $\widetilde g^2$ and $O_{\widetilde g}$ also defined in \eqref{eq:gfun}. Fix $\delta<\delta_0$. Let Jordan arcs $K_j$, $j=1,3$, be the preimages of $\Sigma_1:=\{\zeta:~ \Arg(\zeta)=2\pi/3\}$ and $\Sigma_3 := \{\zeta:~ \Arg(\zeta)=-2\pi/3\}$ under $g^2$ in $U_\delta$. Let also $\widetilde K_j$, $j=1,3$, be the preimages of $\Sigma_1$ and $\Sigma_3$ under $\widetilde g^2$ in $\widetilde U_\delta$. Set $K_+:=K_1\cup K_2 \cup \widetilde K_3$, where $K_2$ is the image under $\Xi$ of the line segment  that joins $\Xi^{-1}(K_1\cap U_\delta)$. Set also $\widetilde \Xi^{-1}(\widetilde K_3\cap \widetilde U_\delta)$, and $K_-:=K_3\cup K_4 \cup \widetilde K_1$, where $K_4$ is the image under $\Xi$ of the line segment  that joins $\Xi^{-1}(K_3\cap U_\delta)$ and $\widetilde \Xi^{-1}(\widetilde K_1\cap \widetilde U_\delta)$. Then $\Delta_\pm$ are Jordan arcs that with endpoints $\pm1$. We define $\Sigma:=\Delta\cup K_+\cup K_-$ (see Fig.~\ref{fig:3}).

Let $g_n$ and $\widetilde g_n$ be defined by \eqref{eq:gnfun}. Assume that $\delta$ is small enough that $U_\delta\subset O_L$ and $\widetilde U_\delta\subset O_{\widetilde L}$. Then we construct the lens $\Sigma_n:=\Delta\cup\Delta_{n+}\cup\Delta_{n-}$ exactly as we constructed $\Sigma$ only with $g$, $\widetilde g$, and $\Xi$ replaced by $g_n$, $\widetilde g_n$, and $\Xi_n$, where we also employ the notation $\Delta_{n\pm}$ for the upper and lower lips of the lens.
\begin{figure}[!ht]
\label{fig:b}
\centering
\includegraphics[scale=.5]{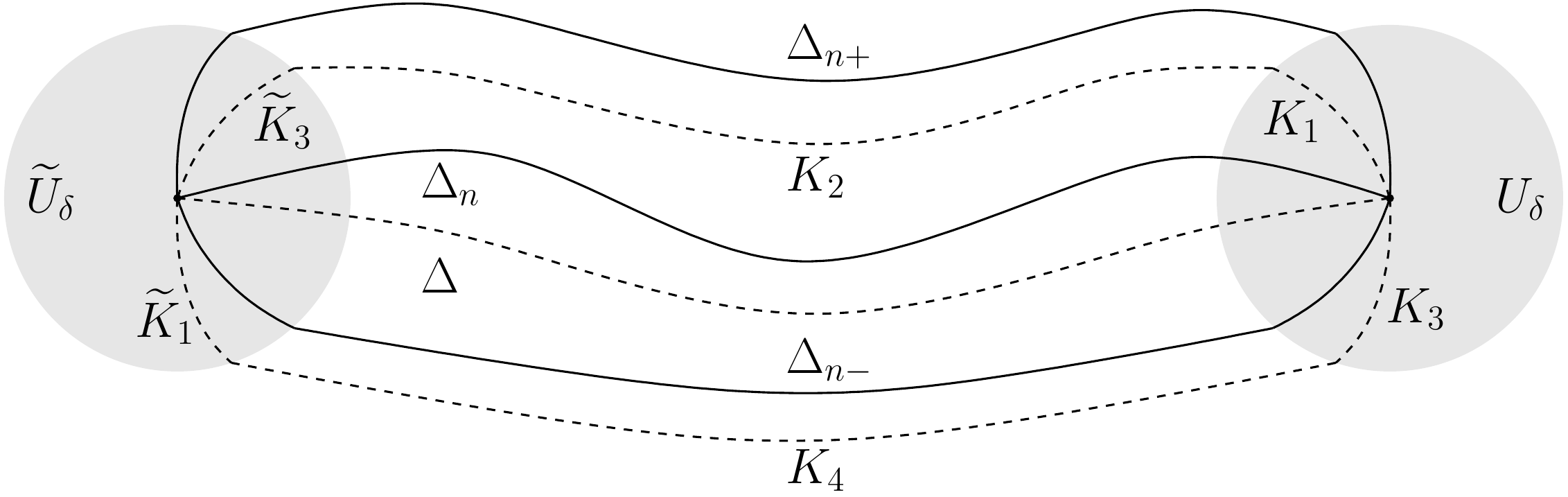}
\caption{\small Contours $\Sigma$ (dashed lines) and $\Sigma_n^{md}$ (solid lines). Neighborhoods $U_\delta$ and $\widetilde U_\delta$ (disks around $\pm1$).}
\label{fig:3}
\end{figure}
It can be easily seen that the arcs $\Delta_{n\pm}$ and $\Delta$ intersect only at $\pm1$ for all $n$ large enough since $\Delta_n$ approach $\Delta$ in a uniform manner by Theorem~\ref{thm:sp} and $\Delta_{n\pm}$ and $\Delta_n$ form angles $\pi/3$ at $1$ and $-1$ by construction.

Finally, it will be useful for us later to define one more system of contours, say $\Sigma_n^{md}$. The lens $\Sigma_n^{md}$ is obtained from $\Sigma_n$ simply by replacing $\Delta$ by $\Delta_n$ (see Fig.~\ref{fig:3}). We also require the lens $\Sigma_{ext}$ to be contained within each lens $\Sigma_n$ (see Fig.~\ref{fig:2}).

\subsection{Extension with Controlled $\bar\partial$ Derivative}
\label{subsec:ext}

Without loss of generality we may assume that $\Sigma_n\subset D_\Xi$ and all the functions $h_n$ are holomorphic in $D_\Xi$. By the very definition of $c_n^\pm$ we have that
\[
c_n^\pm=\gm_{h_n}\left(\szf_{h_n}^\pm\right)^2h_n^{-1}.
\]
Thus, there is a natural holomorphic extension of each $c_n^\pm$ given by 
\begin{equation}
\label{eq:extcn}
c_n:=\gm_{h_n}\szf_{h_n}^2h_n^{-1} \quad \mbox{in} \quad D_\Xi\setminus\Delta.
\end{equation}
Concerning the extension of $c$, we can prove the following.

\begin{lem}
\label{lem:extension}
Let $\theta\in\sof^{1-1/p}_p$, $p\in(2,\infty)$, or $\theta\in\cf^{m,\varsigma}$, $m\in\Z_+$, $\varsigma\in(0,1]$, $m+\varsigma>\frac12$. Then there exists a function $c$, continuous in $\C\setminus\Delta$ and up to $\Delta^\pm$, satisfying 
\[
c_{|\Delta^\pm} = c^\pm, \quad c = \exp\{\sr\ell\} \quad \mbox{in} \quad \overline\C\setminus(\overline{\Omega_+\cup\Omega_-}), \quad \mbox{and} \quad \bar\partial c=cf,
\]
where $\ell$ is a polynomial, $\deg(\ell)\leq2m+1$, $f\in\lf^p(\Omega_\pm)$ when $\theta\in\sof^{1-1/p}_p$, $f\in\lf^q(\Omega_\pm)$, $q\in\left(2,\frac{1}{1-\varsigma}\right)$ when $\theta\in\cf^{0,\varsigma}$, and $f\in\cf^{m-1,\varsigma-\epsilon}_0(\overline{\Omega_\pm})$ when $\theta\in\cf^{m,\varsigma}$, $m\in\N$, and $\epsilon\in(0,\varsigma)$.
\end{lem}
\begin{proof}
This lemma is a straightforward consequence of \eqref{eq:c}, inclusion $\cf^{0,\varsigma}\subset\sof^{1-1/q}_q$ for $q\in\left(2,\frac{1}{1-\varsigma}\right)$, and Propositions~\ref{lem:smooth1} and \ref{lem:smooth2}  combined with Propositions~\ref{pr:corners1} and \ref{pr:corners2}.
\end{proof}

\subsection{Formulation of Riemann-Hilbert-$\bar\partial$ Problem}

In this section we reformulate \rht~ as a Riemann-Hilbert-$\bar\partial$ problem. In what follows, we understand under $c$ and $c_n$ the extensions obtained in Section \ref{subsec:ext} above. Suppose that \rht~ is solvable and $\mathscr{T}$ is the solution. We define a matrix function $\mathscr{S}$ on $\overline\C\setminus\Sigma_n$ as follows:
\begin{equation}
\label{eq:s}
\mathscr{S}:= \left\{
\begin{array}{ll}
\mathscr{T} \left(\begin{array}{cc}1&0 \\ \mp r_nc_nc/w & 1 \end{array}\right), & \mbox{in} \ \Omega_{n\pm}, \smallskip \\
\mathscr{T}, & \mbox{outside the lens} \ \Sigma_n,
\end{array}
\right.
\end{equation}
where the upper part, $\Omega_{n+}$, (resp. lower part, $\Omega_{n-}$) of the lens $\Sigma_n$ is a domain bounded by $\Delta_{n+}$ (resp. $\Delta_{n-}$) and $\Delta$. This new matrix function is no longer analytic in general in the whole domain $D$ since $c$ may not be analytic inside the extension lens $\Sigma_{ext}$. Recall, however, that by the very construction, $c$ coincides with a holomorphic function $\widetilde c=\exp\{\sr\ell\}$ outside the lens $\Sigma_{ext}$. To capture the non-analytic character of $\mathscr{S}$, we introduce the following matrix function that will represent the deviation from analyticity:
\begin{equation}
\label{eq:dfaw0}
\mathscr{W}_0 := \left\{
\begin{array}{ll}
\displaystyle \left(\begin{array}{cc} 0 & 0 \\ \mp r_nc_n \bar\partial c/w & 0 \end{array}\right), & \mbox{in} \ \Omega_{\pm}, \smallskip \\
\left(\begin{array}{cc} 0 & 0 \\ 0 & 0\end{array}\right), & \mbox{outside the lens} \ \Sigma_{ext}.
\end{array}
\right.
\end{equation}
Then $\mathscr{S}$ solves the following Riemann-Hilbert-$\bar\partial$ problem (\rhds):
\begin{itemize}
\item[(a)] $\mathscr{S}$ is a continuous matrix function in $\overline\C\setminus\Sigma_n$ and $\mathscr{S}(\infty)=\mathscr{I}$;
\item[(b)] $\mathscr{S}$ has continuous boundary values, $\mathscr{S}_\pm$, on $\Sigma_n^\circ:=\Sigma_n\setminus\{\pm1\}$ and
\begin{eqnarray}
\mathscr{S}_+ &=& \mathscr{S}_- \left(\begin{array}{cc}1&0 \\ r_nc_nc/w & 1 \end{array}\right) \quad \mbox{on} \quad \Delta_{n+}^\circ\cup\Delta_{n-}^\circ, \nonumber \\
\mathscr{S}_+ &=& \mathscr{S}_- \left(\begin{array}{cc} 0 & w \\ -1/w & 0 \end{array}\right) \quad \mbox{on} \quad\Delta^\circ; \nonumber
\end{eqnarray}
\item[(c)] For $\alpha<0$, $\mathscr{S}$ has the following behavior near $z=1$:
\[
\mathscr{S}(z) = O\left(\begin{array}{cc}1&|1-z|^\alpha\\1&|1-z|^\alpha\end{array}\right),  \quad \mbox{as} \quad \C\setminus\Sigma_n\ni z\to1.
\]
For $\alpha=0$, $\mathscr{S}$ has the following behavior near $z=1$:
\[
\mathscr{S}(z) =  O\left(\begin{array}{cc} \log|1-z|&\log|1-z| \\ \log|1-z|&\log|1-z| \end{array}\right) \mbox{as} \quad \C\setminus\Sigma_n\ni z\to1.
\]
For $\alpha>0$, $\mathscr{S}$ has the following behavior near $z=1$:
\[
\mathscr{S}(z) = \left\{
\begin{array}{ll}
\displaystyle O\left(\begin{array}{cc}1&1\\1&1\end{array}\right), & \mbox{as} \ z\to1 \ \mbox{outside the lens} \ \Sigma_n, \smallskip \\
\displaystyle O\left(\begin{array}{cc}|1-z|^{-\alpha}&1\\|1-z|^{-\alpha}&1\end{array}\right), & \mbox{as} \ z\to1 \ \mbox{inside the lens} \ \Sigma_n;
\end{array}
\right.
\]
\item[(d)] $\mathscr{S}$ has the same behavior when $\C\setminus\Sigma_n\ni z\to-1$ as in (c) only with $\alpha$ replaced by $\beta$ and $1-z$ replaced by $1+z$;
\item[(e)] $\mathscr{S}$ deviates from an analytic matrix function according to $\bar\partial\mathscr{S}=\mathscr{S}\mathscr{W}_0$.
\end{itemize}

Then the following lemma holds.

\begin{lem}
\label{lem:rhs}
\rhds~ is solvable if and only if \rht~ is solvable. When solutions of \rhds~ and \rht~ exist, they are unique and connected by \eqref{eq:s}.
\end{lem}
\begin{proof}
By construction, the solution of \rht~ yields a solution of \rhds. Conversely, let $\mathscr{S}^*$ be a solution of \rhds. It is easy to check using the Leibnitz's rule that $\bar\partial\mathscr{T}^*$ is equal to the zero matrix outside of $\Sigma_n$, where $\mathscr{T}^*$ is obtained from $\mathscr{S}^*$ by inverting (\ref{eq:s}). Thus, $\mathscr{T}^*$ is an analytic matrix function in $\overline\C\setminus\Sigma_n$ with continuous boundary values on each side of $\Sigma_n^\circ$. Moreover, it can be readily verified that $\mathscr{T}^*$ has no jumps on $\Delta_{n\pm}^\circ$ and therefore is, in fact, analytic in $D$. It is aslo obvious that it equals to the identity matrix at infinity and has a jump on $\Delta$ described by \rht(b). Thus, $\mathscr{T}^*$ complies with \rht(a)--(b). 

Now, if $\alpha,\beta<0$ then it follows from \rhds(c)--(d) and (\ref{eq:s}) that $\mathscr{T}^*$ has the same behavior near endpoints $\pm1$ as $\mathscr{S}^*$. Therefore, $\mathscr{T}^*$ solves \rht~ in this case. When either $\alpha$ or $\beta$ is nonnegative, it is no longer immediate that the first column of $\mathscr{T}^*$ has the behavior near $\pm1$ required by \rht(c)--(d). This difficulty was resolved in \cite[Lem. 4.1]{KMcLVAV04} by considering $\mathscr{T}^*\mathscr{T}^{-1}$, where $\mathscr{T}$ is the unique solution of \rht. However, in the present case it is not clear that such a solution exists (see Lemma~\ref{lem:rhy}). Thus, we are bound to consider the first column of $\mathscr{T}^*$ by itself.

Denote by $\mathscr{T}^*_{11}$ and $\mathscr{T}^*_{21}$ the $11$- and $21$-entries of $\mathscr{T}^*$. Then $\mathscr{T}^*_{11}$ and $\mathscr{T}^*_{21}$ are analytic functions in $D$ with the following behavior near $1$: 
\begin{equation}
\label{eq:neartheendentires}
\mathscr{T}^*_{j1}(z) = \left\{
\begin{array}{ll}
O(1), & \mbox{if} \quad \alpha<0 \smallskip \\
O(\log|1-z|), & \mbox{if} \quad \alpha=0, \smallskip \\
O(|1-z|^{-\alpha}), & \mbox{if} \quad \alpha>0 \ \mbox{and} \ z \ \mbox{is inside the lens}, \smallskip \\
O(1), & \mbox{if} \quad \alpha>0 \ \mbox{and} \ z \ \mbox{is outside the lens},
\end{array}
\right.
\end{equation}
for $j=1,2$. The behavior near $-1$ is identical only with $\alpha$ replaced by $\beta$ and $1-z$ replaced by $1+z$. Moreover, each $\mathscr{T}^*_{j1}$ solves the following scalar boundary value problem:
\begin{equation}
\label{eq:bvprn}
\phi^+ = \phi^-(r_nc_nc)^+ \quad \mbox{on} \quad\Delta, \quad \phi\in\hf(D).
\end{equation}
Now, recall that $r_n^+r_n^-\equiv1$ on $\Delta$ and $r_n$ has $2n$ zeros in $D$ that lie away from the lens $\Sigma_n$. Hence, the argument of $r_n^+$ increases by $2\pi n$ when $\Delta$ is traversed from $-1$ to $1$. Moreover, for $c^+$ and each $c_n^+$ a branch of the argument can be taken continuous and vanishing at $\pm1$ (it is the imaginary part of $\sr^+\os(\theta/\sr^+)$, which is continuous and vanishing at $\pm1$ by Propositions \ref{lem:smooth1} and \ref{lem:smooth2}). Define $\varrho := \log (r_nc_nc)^+$, $\varrho(-1)=0$. This normalization is possible since $r_n^+(-1)=1$ as $r_n^+$ is a product of $2n$ factors each of which is equal to $-1$ at $-1$. Furthermore, this normalization necessarily yields that $\varrho(1)=2\pi ni$ and that the so-called canonical solution of the problem (\ref{eq:bvprn}) is given by \cite[Sec. 43.1]{Gakhov}
\[
\phi_c(z) := (z-1)^{-n}\exp\left\{\oc(\varrho;z)\right\}, \quad z\in D.
\]
Recall that $\phi_c$ is bounded above and below in the vicinities of 1 and $-1$, has a zero of order $n$ at infinity, and otherwise is non-vanishing. Hence, the functions $\phi_j:=\mathscr{T}^*_{j1}/\phi_c$, $j=1,2$, are analytic in $\C\setminus\{\pm1\}$. Moreover, according to (\ref{eq:neartheendentires}), the singularities of these functions at 1 and $-1$ cannot be essential, they are either removable or polar. In fact, since $\phi_j(z)=O(1)$ or $\phi_j(z)=O(\log|1\pm z|)$ when $z$ approaches 1 or $-1$ outside of the lens, $\phi_j$ can have only removable singularities at these points. Hence, $\phi_j(z)=O(1)$ and subsequently $\mathscr{T}^*_{j1}=O(1)$  near 1 and $-1$. Thus, $\mathscr{T}^*$ satisfies \rht(c)--(d) for all $\alpha$ and $\beta$, which means that $\mathscr{T}^*$ is the solution of \rht. Therefore, indeed, the problems \rht~ and \rhds~ are equivalent.
\end{proof}

\section{Analytic Approximation of \rhds}
\label{sec:aa}

Elaborating on the path developed in \cite{McLM08}, we put \rhds~ aside for a while and consider an analytic approximation of this problem. In other words, we seek the solution of the following Riemann-Hilbert problem (\rha):
\begin{itemize}
\item[(a)] $\mathscr{A}$ is a holomorphic matrix function in $\overline\C\setminus\Sigma_n$ and $\mathscr{A}(\infty)=\mathscr{I}$;
\item[(b)] $\mathscr{A}$ has continuous traces, $\mathscr{A}_\pm$, on $\Sigma_n^\circ$ that satisfy the same relations as in \rhds(b);
\item[(c)] the behavior of $\mathscr{A}$ near 1 is described by \rhds(c);
\item[(d)] the behavior of $\mathscr{A}$ near $-1$ is described by \rhds(d).
\end{itemize}
Before we proceed, observe that the function $c$ coincides on $\Delta_{n\pm}$ with the analytic function $\widetilde c := \exp\{\sr\ell\}$, where $\ell$ is a polynomial, by construction. Hence, we can assume that the jump matrix in \rha(b) is expressed in terms of $\widetilde c$ rather than $c$.

\subsection{Modified \rha}
The problem above almost falls into the scope of the classical approach to asymptotics of orthogonal polynomials. We say ``almost'' because it is not generally true that the functions $r_n$ can be written as the $2n$-th power of a single function, even up to a normal family as is the case in \cite[Thm. 2]{Ap02}. This will explain why we constructed another lens, $\Sigma_n^{md}$, in Section \ref{subsec:lenses}. 

Consider the following Riemann-Hilbert problem (\rhb):
\begin{itemize}
\item[(a)] $\mathscr{B}$ is a holomorphic matrix function in $\overline\C\setminus\Sigma_n^{md}$ and $\mathscr{B}(\infty)=\mathscr{I}$;
\item[(b)] $\mathscr{B}$ has continuous traces, $\mathscr{B}_{\pm}$, on $(\Sigma_n^{md})^\circ$ that satisfy
\[
\begin{array}{llll}
\mathscr{B}_+ &=& \mathscr{B}_- \left(\begin{array}{cc}1 & 0 \\ r_nc_n\widetilde c/w & 1 \end{array}\right) & \mbox{on} \quad \Delta_{n+}^\circ\cup\Delta_{n-}^\circ, \smallskip \\
\mathscr{B}_+ &=& \mathscr{B}_- \left(\begin{array}{cc} 0 & w \\ -1/w & 0 \end{array}\right) & \mbox{on} \quad \Delta_n^\circ;
\end{array}
\]
\item[(c)] the behavior of $\mathscr{B}$ near 1 is described by \rha(c) with respect to the lens $\Sigma_n^{md}$;
\item[(d)] the behavior of $\mathscr{B}$ near $-1$ is described by \rha(d), again, with respect to $\Sigma_n^{md}$.
\end{itemize}
In fact, this new problem is equivalent to \rha.
\begin{lem}
\label{lem:rhab}
The problems \rha~ and \rhb~ are equivalent.
\end{lem}
\begin{proof}
Suppose that \rhb~ is solvable and $\mathscr{B}$ is a solution. As before, let $\Omega_{n+}$ (resp. $\Omega_{n-}$) be the upper (resp. lower) part of the lens $\Sigma_n$. Analogously define $\Omega_{n\pm}^{md}$ and set
\begin{equation}
\label{eq:a}
\mathscr{A}^* := \left\{
\begin{array}{ll}
\mathscr{B} \left(\begin{array}{cc} 0 & w \\ -1/w & 0 \end{array}\right), & \mbox{in} \quad \Omega_{n+}\cap\Omega_{n-}^{md}, \smallskip \\
\mathscr{B} \left(\begin{array}{cc} 0 & -w \\ 1/w & 0 \end{array}\right), & \mbox{in} \quad \Omega_{n-}\cap\Omega_{n+}^{md}, \smallskip \\
\mathscr{B}, & \mbox{elsewhere}.
\end{array}
\right.
\end{equation}
Observe that $\Omega_{n\pm}\cap\Omega^{md}_{n\mp}$ is a finite, possibly empty, union of Jordan domains by analyticity of $\Delta$ and $\Delta_n$. It is a routine exercise to verify that $\mathscr{A}^*$ complies with \rha(a) and~(b). Moreover, within $\Omega_{n+}\cap\Omega_{n-}^{md}$ and $\Omega_{n-}\cap\Omega_{n+}^{md}$ we have that for $\alpha<0$, $\mathscr{A}^*$ has the following behavior near $z=1$:
\[
\mathscr{A}^*(z) = O\left(\begin{array}{cc} 1 & |1-z|^\alpha \\ 1 & |1-z|^\alpha\end{array}\right)  O\left(\begin{array}{cc} 0 & |1-z|^\alpha \\ |1-z|^{-\alpha} & 0 \end{array}\right) = O\left(\begin{array}{cc} 1 & |1-z|^\alpha \\ 1 & |1-z|^\alpha\end{array}\right),  
\]
as $z\to1$; for $\alpha=0$, $\mathscr{A}^*$ has the same behavior near $z=1$ as $\mathscr{B}$ since the latter is multiplied by a bounded matrix near $1$; for $\alpha>0$, $\mathscr{A}^*$ has the following behavior near $z=1$:
\[
\mathscr{A}^*(z) = O\left(\begin{array}{cc} |1-z|^{-\alpha} & 1 \\ |1-z|^{-\alpha} & 1 \end{array}\right) O\left(\begin{array}{cc} 0 & |1-z|^\alpha \\ |1-z|^{-\alpha} & 0 \end{array}\right) = O\left(\begin{array}{cc} |1-z|^{-\alpha} & 1 \\ |1-z|^{-\alpha} & 1 \end{array}\right).
\]
Hence, $\mathscr{A}^*$ has exactly the behavior near $1$ required by \rha(c). In the same fashion one can check that $\mathscr{A}^*$ satisfies \rha(d)~ and therefore it is, in fact, a solution of \rha. Clearly, the arguments above could be reversed and hence each solution of \rha~ yields a solution of \rhb. 
\end{proof}

Let us now alleviate the notation by making a slightly abuse of it. Throughout this section, we shall understand under $\map$, $r_n$, $g_n$, $\widetilde g_n$, $c_n$, $\widetilde c$, $\sr$, $\szf_{h_n}$, $\szf_{\sr^+}$, and $\szf_{w}$ their holomorphic continuations that are analytic outside of $\Delta_n$ rather than $\Delta$. Note that outside the bounded set with boundary $\Delta\cup \Delta_n$ these continued functions coincide with the original ones. Moreover, their values considered within the interior domain of $\Delta_n\cup\Delta$ can be obtained through analytic continuation of the original functions across $\Delta$.

\subsection{Auxiliary Riemann-Hilbert Problems}

In this subsection we define the necessary objects to solve \rhb. This material essentially appeared in \cite{KMcLVAV04} for the case $\Delta=[-1,1]$.

\subsubsection{Parametrix away from the endpoints} As $r_n$ converges to zero geometrically fast away from $\Delta_n$, the jump matrix in \rhb(b)~ is close to the identity on $\Delta_{n+}^\circ$ and $\Delta_{n-}^\circ$. Thus, the main term of the asymptotics for $\mathscr{B}$ in $D_n=\overline\C\setminus\Delta_n$ is determined by the following Riemann-Hilbert problem (\rhn):
\begin{itemize}
\item[(a)] $\mathscr{N}$ is a holomorphic matrix function in $D_n$ and $\mathscr{N}(\infty)=\mathscr{I}$;
\item[(b)] $\mathscr{N}$ has continuous traces, $\mathscr{N}_{\pm}$, on $\Delta_n^\circ$ and $\mathscr{N}_+ = \mathscr{N}_- \left(\begin{array}{cc} 0 & w \\ -1/w & 0 \end{array}\right)$;
\end{itemize}
It can be easily checked using \eqref{eq:szfsrpm} that a solution of \rhn~ when $w\equiv1$ is given by
\begin{equation}
\label{eq:nstar}
\mathscr{N_*} = \left(\begin{array}{cc} \displaystyle \szf_{\sr^+}^{-1} & \displaystyle i(\map\szf_{\sr^+})^{-1} \smallskip \\ \displaystyle -i(\map\szf_{\sr^+})^{-1}  & \szf_{\sr^+}^{-1} \end{array}\right).
\end{equation}
Then a solution of \rhn~ for arbitrary $w$ is given by
\begin{equation}
\label{eq:n}
\mathscr{N} = (\gm_w)^{\sigma_3/2}\mathscr{N}_*(\gm_w\szf_w^2)^{-\sigma_3/2}.
\end{equation}

\subsubsection{Auxiliary parametrix near the endpoints} The following construction was introduced in \cite[Thm. 6.3]{KMcLVAV04}. Let $I_\alpha$ and $K_\alpha$ be the modified Bessel functions and $H_\alpha^{(1)}$ and $H_\alpha^{(2)}$ be the Hankel functions \cite[Ch. 9]{AbramowitzStegun}. Set $\Psi$ to be the following sectionally holomorphic matrix function:
\[
\Psi(\zeta) = \Psi(\zeta;\alpha) := \left(
\begin{array}{cc}
I_\alpha\left(2\zeta^{1/2}\right) & \frac i\pi K_\alpha\left(2\zeta^{1/2}\right) \smallskip \\
2\pi i\zeta^{1/2} I_\alpha^\prime\left(2\zeta^{1/2}\right) & -2\zeta^{1/2} K_\alpha^\prime\left(2\zeta^{1/2}\right) 
\end{array}
\right)
\]
for $|\Arg(\zeta)|<2\pi/3$;
\[
\Psi(\zeta) := \left(\begin{array}{cc}
\frac12 H_\alpha^{(1)}\left(2(-\zeta)^{1/2}\right) & \frac12 H_\alpha^{(2)}\left(2(-\zeta)^{1/2}\right) \smallskip \\
\pi\zeta^{1/2}\left(H_\alpha^{(1)}\right)^\prime\left(2(-\zeta)^{1/2}\right) & \pi\zeta^{1/2}\left( H_\alpha^{(2)}\right)^\prime\left(2(-\zeta)^{1/2}\right)
\end{array}
\right)e^{\frac12 \alpha\pi i\sigma_3}
\]
for $2\pi/3<\Arg(\zeta)<\pi$;
\[
\Psi(\zeta) := \left(\begin{array}{cc}
\frac12 H_\alpha^{(2)}\left(2(-\zeta)^{1/2}\right) & -\frac12 H_\alpha^{(1)}\left(2(-\zeta)^{1/2}\right) \smallskip \\
-\pi\zeta^{1/2}\left(H_\alpha^{(2)}\right)^\prime\left(2(-\zeta)^{1/2}\right) & \pi\zeta^{1/2}\left( H_\alpha^{(1)}\right)^\prime\left(2(-\zeta)^{1/2}\right)
\end{array}
\right)e^{-\frac12 \alpha\pi i\sigma_3}
\]
for $-\pi<\Arg(\zeta)<-2\pi/3$, where $\Arg(\zeta)\in(-\pi,\pi]$ is the principal determination of the argument of $\zeta$. Let further $\Sigma_1$, $\Sigma_2$, and $\Sigma_3$ be the rays $\{\zeta:\Arg(\zeta)=2\pi/3\}$, $\{\zeta:\Arg(\zeta)=\pi\}$, and $\{\zeta:\Arg(\zeta)=-2\pi/3\}$, respectively, oriented from infinity to zero. Using known properties of $I_\alpha$, $K_\alpha$, $H_\alpha^{(1)}$, $H_\alpha^{(2)}$, and their derivatives, it can be checked that $\Psi$ is the solution of the following Riemann-Hilbert problem \rhpsi:
\begin{itemize}
\item[(a)] $\Psi$ is a holomorphic matrix function in $\C\setminus(\Sigma_1\cup\Sigma_2\cup\Sigma_3)$;
\item[(b)] $\Psi$ has continuous traces, $\Psi_\pm$, on $\Sigma^\circ_j$, $j\in\{1,2,3\}$, and 
\[
\begin{array}{llll}
\Psi_+ &=& \Psi_- \left(\begin{array}{cc} 1 & 0 \\ e^{\alpha\pi i} & 1 \end{array}\right) & \mbox{on} \quad \Sigma_1^\circ, \smallskip \\
\Psi_+ &=& \Psi_- \left(\begin{array}{cc} 0 & 1 \\ -1 & 0 \end{array}\right) & \mbox{on} \quad \Sigma_2^\circ, \smallskip \\
\Psi_+ &=& \Psi_- \left(\begin{array}{cc} 1 & 0 \\ e^{-\alpha\pi i} & 1 \end{array}\right) & \mbox{on} \quad \Sigma_3^\circ;
\end{array}
\]
\item[(c)] $\Psi$ has the following behavior near $\infty$:
\[
\Psi(\zeta) = \left(2\pi\zeta^{1/2}\right)^{-\sigma_3/2} \frac{1}{\sqrt2}\left(
\begin{array}{cc}
1 + O\left(\zeta^{-1/2}\right) & i + O\left(\zeta^{-1/2}\right) \smallskip \\
i + O\left(\zeta^{-1/2}\right) & 1 + O\left(\zeta^{-1/2}\right)
\end{array}
\right)e^{2\zeta^{1/2}\sigma_3}
\]
uniformly in $\C\setminus(\Sigma_1\cup\Sigma_2\cup\Sigma_3)$;
\item[(d)] For $\alpha<0$, $\Psi$ has the following behavior near $0$:
\[
\Psi = O\left(\begin{array}{cc} |\zeta|^{\alpha/2} & |\zeta|^{\alpha/2} \\ |\zeta|^{\alpha/2} & |\zeta|^{\alpha/2} \end{array}\right) \ \mbox{as} \ \zeta\to0;
\]
For $\alpha=0$, $\Psi$ has the following behavior near $0$:
\[
\Psi = O\left(\begin{array}{cc} \log|\zeta| & \log|\zeta| \\ \log|\zeta| & \log|\zeta| \end{array}\right) \ \mbox{as} \ \zeta\to0;
\]
For $\alpha>0$, $\Psi$ has the following behavior near $0$:
\[
\Psi = \left\{
\begin{array}{ll}
\displaystyle O\left(\begin{array}{cc} |\zeta|^{\alpha/2} & |\zeta|^{-\alpha/2} \\ |\zeta|^{\alpha/2} & |\zeta|^{-\alpha/2} \end{array}\right) & \mbox{as} \ \zeta\to0 \ \mbox{in} \ |\Arg(\zeta)|<2\pi/3, \smallskip \\
\displaystyle O\left(\begin{array}{cc} |\zeta|^{-\alpha/2} & |\zeta|^{-\alpha/2} \\ |\zeta|^{-\alpha/2} & |\zeta|^{-\alpha/2} \end{array}\right) & \mbox{as} \ \zeta\to0 \ \mbox{in} \ 2\pi/3<|\Arg(\zeta)|<\pi.
\end{array}
\right.
\]
\end{itemize}

Further, if we set 
\[
\widetilde \Psi := \sigma_3\Psi(\cdot;\beta)\sigma_3,
\]
then this matrix function satisfies \rhpsi~ with $\alpha$ replaced by $\beta$ and reversed orientation for $\Sigma_j$, $j\in\{1,2,3\}$.

\subsubsection{Parametrix near $1$} As shown in \cite{KMcLVAV04}, the main term of the asymptotics of the solution of \rhb~ near the endpoints of $\Delta_n$ is described by a solution of a special Riemann-Hilbert problem for which $\Psi$ will be instrumental. Let $U_\delta$ be as in the construction of $\Sigma_n$ (see Section~\ref{subsec:lenses}). Below we describe the solution of the following Riemann-Hilbert problem (\rhp):
\begin{itemize}
\item[(a)] $\mathscr{P}$ is a holomorphic matrix function in $U_\delta\setminus\Sigma_n^{md}$;
\item[(b)] $\mathscr{P}$ has continuous boundary values, $\mathscr{P}_\pm$, on $U_\delta\cap(\Sigma_n^{md})^\circ$ and
\[
\begin{array}{llll}
\mathscr{P}_+ &=& \mathscr{P}_- \left(\begin{array}{cc}1&0 \\ r_nc_n\widetilde c/w & 1 \end{array}\right) & \mbox{on} \quad U_\delta\cap(\Delta_{n+}^\circ\cup\Delta_{n-}^\circ), \smallskip \\
\mathscr{P}_+ &=& \mathscr{P}_- \left(\begin{array}{cc} 0 & w \\ -1/w & 0\end{array}\right) & \mbox{on} \quad U_\delta\cap \Delta_n^\circ; 
\end{array}
\]
\item[(c)] $\mathscr{P}\mathscr{N}^{-1}=\mathscr{I}+O(1/n)$ uniformly on $\partial U_\delta$;
\item[(d)] For $\alpha<0$, $\mathscr{P}$ has the following behavior near $z=1$:
\[
\mathscr{P} = O\left(\begin{array}{cc}1&|1-z|^\alpha\\1&|1-z|^\alpha\end{array}\right),  \quad \mbox{as} \quad U_\delta\setminus\Sigma_n^{md}\ni z\to1;
\]
For $\alpha=0$, $\mathscr{P}$ has the following behavior near $z=1$:
\[
\mathscr{P} = O\left(\begin{array}{cc} \log|1-z| & \log|1-z| \\ \log|1-z| & \log|1-z| \end{array}\right), \quad \mbox{as} \quad U_\delta\setminus\Sigma_n^{md}\ni z\to1;
\]
For $\alpha>0$, $\mathscr{P}$ has the following behavior near $z=1$:
\[
\mathscr{P} = \left\{
\begin{array}{ll}
\displaystyle O\left(\begin{array}{cc}1&1\\1&1\end{array}\right), & \mbox{as} \ z\to1 \ \mbox{outside the lens} \ \Sigma_n^{md}, \smallskip \\
\displaystyle O\left(\begin{array}{cc}|1-z|^{-\alpha}&1\\|1-z|^{-\alpha}&1\end{array}\right), & \mbox{as} \ z\to1 \ \mbox{inside the lens} \ \Sigma_n^{md}.
\end{array}
\right.
\]
\end{itemize}

To present a solution of \rhp, we need to introduce more notation. Denote by $U_\delta^+$ and $U_\delta^-$ the subsets of $U_\delta$ that are mapped by $g_n^2$ into the upper and lower half planes, respectively. Without loss of generality we may assume that functions $\theta_n$ are holomorphic in $U_\delta$ and the branch cut of $w$ in $U_\delta$ coincides with the preimage of the positive reals under $g^2_n$. In particular, we have that $w$ is analytic in $U^+_\delta$ and $U_\delta^-$ and therefore across $\Delta_{n\pm}^\circ$. Set
\[
A_n(z) := \frac{\exp\left\{\frac12\left(\theta_n(z)-\sr(z)\ell(z)\right)\right\}}{\sqrt{\gm_{h_n}}\szf_{h_n}(z)} (z-1)^{\alpha/2}(z+1)^{\beta/2},
\]
where we use the same branch of $(z+1)^{\beta/2}$ as in definition of $w$ and a branch of $(z-1)^{\alpha/2}$ analytic in $U_\delta\setminus\Delta_n$ and positive for positive reals large enough. Then
\begin{equation}
\label{eq:An1}
A_n^2 = \left\{
\begin{array}{ll}
e^{\alpha\pi i}w/c_n\widetilde c, & \mbox{in} \quad U_\delta^+, \smallskip \\
e^{-\alpha\pi i}w/c_n\widetilde c, & \mbox{in} \quad U_\delta^-,
\end{array}
\right.
\end{equation}
by the definition of $\widetilde c$ and on account of \eqref{eq:extcn}. Moreover, it readily follows from \eqref{eq:An1} and \eqref{eq:cpmcnpm} that
\begin{equation}
\label{eq:An2}
A^+_nA^-_n = w \quad \mbox{on} \quad \Delta_n^\circ.
\end{equation}
Observe that $g_n^{1/2}$ is a holomorphic function on $U_\delta\setminus \Delta_n$ such that
\begin{equation}
\label{eq:gni}
(g_n^{1/2})^+=i(g_n^{1/2})^- \quad \mbox{on} \quad \Delta_n
\end{equation}
by \eqref{eq:bvgn}. Then the following lemma holds.
\begin{lem}
\label{lem:rhp}
A solution of \rhp~ is given by
\[
\mathscr{P} = \mathscr{E}\Psi\left(\frac{n^2g_n^2}{4}\right)A_n^{-\sigma_3}r_n^{\sigma_3/2}, \quad \mathscr{E} := \mathscr{N}A_n^{\sigma_3} \frac{1}{\sqrt2} \left(\begin{array}{cc} 1 & -i \\ -i & 1 \end{array}\right) \left(\pi ng_n\right)^{\sigma_3/2}.
\]
\end{lem}
\begin{proof} Except for some technical differences, the proof is analogous to the considerations in \cite[eqn. (6.27) and after]{KMcLVAV04}. First, we must show that $\mathscr{E}$ is holomorphic in $U_\delta$. This is clearly true in $U_\delta\setminus \Delta_n$. It is also clear that $\mathscr{E}$ has continuous boundary values on each side of $\Delta_n^\circ$. Since
\begin{eqnarray}
\mathscr{E}_+ &=& \mathscr{N}_+\left(A_n^+\right)^{\sigma_3} \frac{1}{\sqrt2} \left(\begin{array}{cc} 1 & -i \\ -i & 1 \end{array}\right) (\pi n)^{\sigma_3/2}\left(\left(g_n^{1/2}\right)^+\right)^{\sigma_3} \nonumber \\
{} &=& \mathscr{N}_- \left(\begin{array}{cc} 0 & w \\ -1/w & 0 \end{array}\right) \left(\frac{w}{A_n^-}\right)^{\sigma_3} \frac{1}{\sqrt2} \left(\begin{array}{cc} 1 & -i \\ -i & 1 \end{array}\right) (\pi n)^{\sigma_3/2} \left(\begin{array}{cc} i & 0 \\ 0 & -i \end{array}\right) \left(\left(g_n^{1/2}\right)^-\right)^{\sigma_3} \nonumber \\
{} &=& \mathscr{N}_-\left(A_n^-\right)^{\sigma_3} \left(\begin{array}{cc} 0 & 1 \\ -1 & 0 \end{array}\right) \frac{1}{\sqrt2} \left(\begin{array}{cc} i & -1 \\ 1 & -i \end{array}\right) (\pi n)^{\sigma_3/2}  \left(\left(g_n^{1/2}\right)^-\right)^{\sigma_3} = \mathscr{E}_-, \nonumber 
\end{eqnarray}
where we used \rhn(b), \eqref{eq:An2}, and \eqref{eq:gni}, $\mathscr{E}$ is holomorphic across $\Delta_n^\circ$. Thus, it remains to show that $\mathscr{E}$ has no singularity at 1. For this observe that
\[
\left(g_n^{1/2}(z)\right)^{\sigma_3} = O\left(\begin{array}{cc} |1-z|^{1/4} & 0 \\ 0 & |1-z|^{-1/4} \end{array}\right), \quad \mbox{as} \quad z\to1,
\]
since $g_n^2$ has a simple zero at 1. Furthermore, by the very definition it holds that
\[
\mathscr{N}_* = O\left(\begin{array}{cc} |1-z|^{-1/4} & |1-z|^{-1/4} \\ |1-z|^{-1/4} & |1-z|^{-1/4} \end{array}\right), \quad \mbox{as} \quad z\to1.
\]
Finally, $\displaystyle (A_n/\szf_w)(z)\to 2^{-(\alpha+\beta)/2}$ as $z\to1$ by \eqref{eq:szfw}. Hence, the entries of $\mathscr{E}$ can have at most square-root singularity at 1, which is impossible since $\mathscr{E}$ is analytic in $U_\delta\setminus\{1\}$, and therefore $\mathscr{E}$ is analytic in the whole disk $U_\delta$.

The analyticity of $\mathscr{E}$ implies that the jumps of $\mathscr{P}$ are those of $\Psi\left(n^2g_n^2/4\right)A_n^{-\sigma_3}r_n^{\sigma_3/2}$. Clearly, the latter has jumps on $\Sigma_n^{md}\cap U_\delta$ by the very definition of $g_n^2$ and $\Psi$. Moreover, it is a routine exercise, using \rhpsi(b) and \eqref{eq:An1}, to verify that these jumps are described exactly by \rhp(b). It is also clear that \rhp(a) is satisfied. Further, we get directly from \rhpsi(c) that the behavior of $\Psi\left(n^2g_n^2/4\right)$ on $\partial U_\delta$ can be described by
\[
\Psi\left(\frac{n^2g_n^2}{4}\right) = (\pi ng_n)^{-\sigma_3/2} \left(
\begin{array}{cc}
\displaystyle 1 + O\left(1/n\right) & \displaystyle i + O\left(1/n\right) \smallskip \\
\displaystyle i + O\left(1/n\right) & \displaystyle 1 + O\left(1/n\right)
\end{array}
\right) r_n^{-\sigma_3/2},
\]
where the property $O(1/n)$ holds uniformly on $\partial U_\delta$. Hence, using that the diagonal matrices $A_n^{-\sigma_3}$ and $r_n^{\sigma_3/2}$ commute, we get that
\begin{eqnarray}
\mathscr{P}\mathscr{N}^{-1} &=& \mathscr{E}\left(\pi ng_n\right)^{-\sigma_3/2}\left(
\begin{array}{cc}
\displaystyle 1 + O\left(1/n\right) & \displaystyle i + O\left(1/n\right) \smallskip \\
\displaystyle i + O\left(1/n\right) & \displaystyle 1 + O\left(1/n\right)
\end{array}
\right) A_n^{-\sigma_3}\mathscr{N}^{-1} \nonumber \\
{} &=& \mathscr{N}A_n^{\sigma_3} \left(\mathscr{I}+O\left(1/n\right)\right)A_n^{-\sigma_3}\mathscr{N}^{-1} = \mathscr{I}+O\left(1/n\right) \nonumber
\end{eqnarray}
since the moduli of all the entries of $\mathscr{N}A_n^{\sigma_3}$ are uniformly bounded above and away from zero on $\partial U_\delta$. Thus, \rhp(c)~ holds. Finally, \rhp(d)~ follows immediately from \rhpsi(d)~ upon recalling that $|g_n^2(z)|=O(|1-z|)$ and $|A_n(z)|\sim|1-z|^{\alpha/2}$ as $z\to1$.
\end{proof}

\subsubsection{Parametrix near $-1$} In this section we describe the solution of the Riemann-Hilbert problem that plays the same role with respect to $-1$ as \rhp~ did for 1. Below we describe the solution of the following Riemann-Hilbert problem  (\rhpt):
\begin{itemize}
\item[(a)] $\mathscr{\widetilde P}$ is a holomorphic matrix function in $\widetilde U_\delta\setminus\Sigma_n^{md}$;
\item[(b)] $\mathscr{\widetilde P}$ has continuous boundary values, $\mathscr{\widetilde P}_\pm$, on $\widetilde U_\delta\cap(\Sigma_n^{md})^\circ$ and
\[
\begin{array}{llll}
\mathscr{\widetilde P}_+ &=& \mathscr{\widetilde P}_- \left(\begin{array}{cc}1&0 \\ r_nc_n\widetilde c/w & 1 \end{array}\right) & \mbox{on} \quad \widetilde U_\delta\cap(\Delta_{n+}^\circ\cup\Delta_{n-}^\circ), \smallskip \\
\mathscr{\widetilde P}_+ &=& \mathscr{\widetilde P}_- \left(\begin{array}{cc} 0 & w \\ -1/w & 0\end{array}\right) & \mbox{on} \quad \widetilde U_\delta\cap \Delta_n^\circ; 
\end{array}
\]
\item[(c)] $\mathscr{\widetilde P}\mathscr{N}^{-1}=\mathscr{I}+O(1/n)$ uniformly on $\partial \widetilde U_\delta$;
\item[(d)] For $\beta<0$, $\mathscr{\widetilde P}$ has the following behavior near $z=-1$:
\[
\mathscr{\widetilde P} = O\left(\begin{array}{cc}1&|1+z|^\beta \\ 1 & |1+z|^\beta \end{array}\right),  \quad \mbox{as} \quad \widetilde U_\delta\setminus\Sigma_n^{md}\ni z\to-1;
\]
For $\beta=0$, $\mathscr{\widetilde P}$ has the following behavior near $z=-1$:
\[
\mathscr{\widetilde P} = O\left(\begin{array}{cc} \log|1+z| & \log|1+z| \\ \log|1+z| & \log|1+z| \end{array}\right), \quad \mbox{as} \quad \widetilde U_\delta\setminus\Sigma_n^{md}\ni z\to-1;
\]
For $\beta>0$, $\mathscr{\widetilde P}$ has the following behavior near $z=-1$:
\[
\mathscr{\widetilde P} = \left\{
\begin{array}{ll}
\displaystyle O\left(\begin{array}{cc}1&1\\1&1\end{array}\right), & \mbox{as} \ z\to-1 \ \mbox{outside the lens} \ \Sigma_n^{md}, \smallskip \\
\displaystyle O\left(\begin{array}{cc}|1+z|^{-\beta} & 1 \\ |1+z|^{-\beta} & 1\end{array}\right), & \mbox{as} \ z\to-1 \ \mbox{inside the lens} \ \Sigma_n^{md}.
\end{array}
\right.
\]
\end{itemize}

This problem is solved exactly in the same manner as  \rhp. Thus, we set
\[
\widetilde A_n(z) := \frac{\exp\left\{\frac12\left(\theta_n(z)-\sr(z)\ell(z)\right)\right\}}{\sqrt{\gm_{h_n}}\szf_{h_n}(z)}(1-z)^{\alpha/2}(-1-z)^{\beta/2},
\]
where the branch of $(1-z)^{\alpha/2}$ is the same as the corresponding one in $w$, and $(-1-z)^{\beta/2}$ is holomorphic in $\widetilde U_\delta\setminus\Delta_n$ and positive for negative real large enough. As in \eqref{eq:An1}, we have that
\[
\widetilde A_n^2 = \left\{
\begin{array}{ll}
e^{\beta\pi i}w/c_n\widetilde c, & \mbox{in} \quad \widetilde U_\delta^+, \smallskip \\
e^{-\beta\pi i}w/c_n\widetilde c, & \mbox{in} \quad \widetilde U_\delta^-,
\end{array}
\right.
\]
where $\widetilde U_\delta^\pm$ have the same meaning as in the previous section. However, here one needs to be cautious since $\widetilde g_n$ reverses the orientation on $\Sigma_2$, i.e. $\Sigma_2$ is now oriented from zero to infinity, and therefore $\widetilde U_\delta^+$ is mapped into $\{z:~\im(z)<0\}$ and $\widetilde U_\delta^-$ into $\{z:~\im(z)>0\}$. Again, it can be checked that
\[
\widetilde A^+_n\widetilde A^-_n = w \quad \mbox{on} \quad \Delta_n^\circ.
\]
The following lemma can be proven exactly as Lemma \ref{lem:rhp} using that $(-1)^nr_n^{1/2}=e^{n\widetilde g_n}$.
\begin{lem}
\label{lem:rhpt}
The solution of \rhpt~ is given by
\[
\mathscr{\widetilde P} = \mathscr{\widetilde E}\widetilde \Psi\left(\frac{n^2\widetilde g_n^2}{4}\right)\widetilde A_n^{-\sigma_3}(-1)^{n\sigma_3}r_n^{\sigma_3/2}, \ \mathscr{\widetilde E} := \mathscr{N}\widetilde A_n^{\sigma_3} \frac{1}{\sqrt2} \left(\begin{array}{cc} 1 & i \\ i & 1 \end{array}\right) \left(\pi n\widetilde g_n\right)^{\sigma_3/2}.
\]
\end{lem}

Finally, we are prepared to solve \rha.

\subsection{Solution of \rha}

Denote by $\Sigma^{rd}_n$ the reduced system of contours obtained from $\Sigma_n^{md}$ by removing $\Delta_n$ and $\Delta_{n\pm}\cap(U_\delta\cup\widetilde U_\delta)$ and adding $\partial U_\delta \cup \partial \widetilde U_\delta$. For this new system of contours we consider the following Riemann-Hilbert problem (\rhr):
\begin{itemize}
\item[(a)] $\mathscr{R}$ is a holomorphic matrix function in $\overline\C\setminus\Sigma_n^{rd}$ and $\mathscr{R}(\infty)=\mathscr{I}$;
\item[(b)] the traces of $\mathscr{R}$, $\mathscr{R}_\pm$, are continuous on $\Sigma^{rd}_n$ except for the branching points of $\Sigma^{rd}_n$, where they have definite limits from each sector and along each branch of $\Sigma_n^{rd}$. Moreover, $\mathscr{R}_\pm$ satisfy 
\[
\mathscr{R}_+  = \mathscr{R}_- \left\{
\begin{array}{ll}
\mathscr{P}\mathscr{N}^{-1} & \mbox{on} \quad \partial U_\delta, \smallskip \\
\mathscr{\widetilde P}\mathscr{N}^{-1} & \mbox{on} \quad \partial \widetilde U_\delta \smallskip \\
\mathscr{N} \left(\begin{array}{cc} 1 & 0 \\ r_nc_n\widetilde c/w & 1 \end{array}\right) \mathscr{N}^{-1} & \mbox{on} \quad \Sigma_n^{rd}\setminus(\partial U_\delta \cup \partial \widetilde U_\delta).
\end{array}
\right.
\]
\end{itemize}
Then the following lemma takes place.
\begin{lem}
\label{lem:rhr}
The solution of \rhr~ exists for all $n$ large enough and satisfies
\begin{equation}
\label{eq:r}
\mathscr{R}=\mathscr{I}+O\left(1/n\right),
\end{equation}
where $O(1/n)$ holds uniformly in $\overline\C$. Moreover, $\det(\mathscr{R})=1$.
\end{lem}
\begin{proof}
By \rhp(c) and \rhpt(c), we have that \rhr(b) can be written as
\begin{equation}
\label{eq:smalljump}
\mathscr{R}_+ = \mathscr{R_-}\left(\mathscr{I}+O\left(1/n\right)\right)
\end{equation}
uniformly on $\partial U_\delta\cup\partial\widetilde U_\delta$. Further, since the jump of $\mathscr{R}$ on $\Sigma_n^{rd}\setminus(\partial U_\delta \cup \partial \widetilde U_\delta)$ is analytic, it allows us to deform the problem \rhr~ to a fixed contour, say $\Sigma^{rd}$, obtained from $\Sigma$ like $\Sigma_n^{rd}$ was obtained from $\Sigma_n^{md}$  (the solutions exist, are simultaneously unique, and can be easily expressed through each other as in \eqref{eq:a}). Moreover, by the properties of $r_n$, the jump of $\mathscr{R}$ on $\Sigma^{rd}\setminus(\partial U_\delta \cup \partial \widetilde U_\delta)$ is geometrically uniformly close to $\mathscr{I}$. Hence, \eqref{eq:smalljump} holds uniformly on $\Sigma^{rd}$. Thus, by \cite[Cor. 7.108]{Deift}, \rhr~ is solvable for all $n$ large enough and $\mathscr{R}_\pm$ converge to zero on $\Sigma^{rd}$ in $L^2$-sense as fast as $1/n$. The latter yields \eqref{eq:r} locally uniformly in $\overline\C\setminus\Sigma^{rd}$. To show that \eqref{eq:r} holds at $z\in\Sigma^{rd}$, deform $\Sigma^{rd}$ to a  new contour that avoids $z$ (by making $\delta$ smaller or chosing different arcs to connect $U_\delta$ and $\widetilde U_\delta$). As the jump in \rhr~ is given by analytic matrix functions, one can state an equivalent problem on this new contour, the solution to which is an analytic continuation of $\mathscr{R}$. However, now we have that \eqref{eq:r} holds locally around $z$. Compactness of $\Sigma^{rd}$ finishes the proof of \eqref{eq:r}.

Finally, as $\mathscr{N}$, $\mathscr{P}$, and $\mathscr{\widetilde P}$ have determinants equal to 1 throughout $\C$ \cite[Rem. 7.1]{KMcLVAV04}, $\det\mathscr{R}$ is an analytic function in $\overline\C\setminus\Sigma^{rd}$ that is equal to 1 at infinity, has equal boundary values on each side of $\Sigma^{rd}\setminus\{\mbox{branching points}\}$, and is bounded near the branching points. Thus, $\det(\mathscr{R})\equiv1$.
\end{proof}

Finally, we provide the solution of \rha.
\begin{lem}
\label{lem:rhar}
The solution of \rha~  exists for all $n$ large enough and is given by \eqref{eq:a} with
\begin{equation}
\label{eq:b}
\mathscr{B} := \left\{
\begin{array}{ll}
\mathscr{R}\mathscr{N}, & \mbox{in} \quad \overline\C\setminus(\overline U_\delta \cup \overline{\widetilde U_\delta} \cup \Sigma_n), \smallskip \\
\mathscr{R}\mathscr{P}, & \mbox{in} \quad  U_\delta, \smallskip \\
\mathscr{R}\mathscr{\widetilde P}, & \mbox{in} \quad \widetilde U_\delta,
\end{array}
\right.
\end{equation}
where $\mathscr{R}$ is the solution of \rhr.  Moreover, $\det\mathscr{A}\equiv1$.
\end{lem}
\begin{proof}
It can be easily checked from the definition of $\mathscr{P}$, $\mathscr{\widetilde P}$, and $\mathscr{N}$, that $\mathscr{B}$, given by \eqref{eq:b}, is the solution of \rhb.  As $\det\mathscr{R}=\det(\mathscr{P})=\det(\widetilde{\mathscr{P}})\equiv1$, it holds that $\det(\mathscr{A})=\det(\mathscr{B})\equiv1$ in $\overline\C$, which finishes the proof of the lemma.
\end{proof}

\section{$\bar\partial$-Problem}
\label{sec:pbp}

In the previous section we completed the first step in solving \rhds. That is we solved \rha, the problem with the same conditions as in \rhds~ except for the deviation from analyticity, which was entirely ignored. In this section, we solve a complementary problem, namely, we show that a solution of a certain $\bar\partial$-problem for matrix functions exists. Set
\begin{equation}
\label{eq:dfaw}
\mathscr{W} = \mathscr{A}\mathscr{W}_0\mathscr{A}^{-1},
\end{equation}
where $\mathscr{W}_0$ was defined in \eqref{eq:dfaw0} and $\mathscr{A}$ is the solution of \rha. In what follows, we seek the solution of the following $\bar\partial$-problem (\dbd):
\begin{itemize}
\item[(a)] $\mathscr{D}$ is a continuous matrix function in $\overline\C$ and $\mathscr{D}(\infty)=\mathscr{I}$;
\item[(b)] $\mathscr{D}$ deviate from an analytic matrix function according to $\bar\partial\mathscr{D}=\mathscr{D}\mathscr{W}$, where the equality is understood in the sense of distributions. 
\end{itemize}
Then the following lemma holds.
\begin{lem}
\label{lem:deltabar}
If \eqref{eq:ab} is fulfilled, then \dbd~ is solvable for all $n$ large enough. The solution $\mathscr{D}$ is unique and satisfies
\begin{equation}
\label{eq:d}
\mathscr{D}=\mathscr{I}+o(1),
\end{equation}
where $o(1)$ satisfies \eqref{eq:deltane}.
\end{lem}
\begin{proof}
We start by examining the summability and smoothness of the entries of $\mathscr{W}$. As $\mathscr{A}$ is the solution of \rha, it is an analytic matrix function in $\Omega_\pm$ and its behavior near $\pm1$ is given by \rhds(c)--(d). Since $\det\mathscr{A}\equiv1$ in $\overline\C$, the behavior of $\mathscr{A}^{-1}$ near $\pm1$ is also governed by the matrices in \rhds(c)-(d) with the elements on the main diagonal interchanged. Observe also that $|c_nc|$ is uniformly bounded in $\Omega_\pm$. Then a simple computation combined with Lemma \ref{lem:extension} yields that
\begin{equation}
\label{eq:bnd1w}
|\mathscr{W}_{lk}| \leq \const |r_nf_{\alpha,\beta}f|, \quad l,k=1,2,
\end{equation}
where $f$ comes from the decomposition of $\bar\partial c$ in Lemma \ref{lem:extension} and $f_{\alpha,\beta}(z):=\log^2|1-z^2|$ if $\upsilon:=\max\{|\alpha|,|\beta|\}=0$ and $f_{\alpha,\beta}(z):=|1-z^2|^\upsilon$ otherwise. 

Let $s$ be as in \eqref{eq:ab} and denote by $\Omega$ the union $\Omega_+\cup\Omega_-\cup\Delta^\circ$. When $\theta\in\sof^{1-1/p}_p$, $p\in(2,\infty)$, it holds that $f\in\lf^p(\Omega)$ by Lemma~\ref{lem:extension}. Then we get from the H\"older inequality that 
\begin{equation}
\label{eq:bnd2w}
f_{\alpha,\beta}f\in\lf^q(\Omega), \quad q\in
\left\{
\begin{array}{ll}
\left(2,\frac{2}{1+\upsilon-s}\right), & s-\upsilon\leq 1, \smallskip \\
\left(2,\infty\right], & s-\upsilon>1,
\end{array}
\right. 
\end{equation}
since $f_{\alpha,\beta}\in\lf^q(\Omega)$, $q\in\left(\frac{2p}{p-2},\frac{2}{\upsilon}\right)$. When $\theta\in\cf^{0,\varsigma}$, $\varsigma\in\left(\frac12,1\right]$, \eqref{eq:bnd2w} can be obtained from the inclusion $\cf^{0,\varsigma}\subset\sof^{1-1/q}_q$ for $q\in\left(2,\frac{1}{1-\varsigma}\right)$. When $\theta\in\cf^{m,\varsigma}$, $m\in\N$, $\varsigma\in(0,1]$, we have that\footnote{Recall that by Lemma~\ref{lem:extension}, $f$ and all its partial derivatives up to and including the order $m-1$ have well-defined vanishing boundary values on $\Delta$ and therefore $f$ is indeed $C^{m-1,\varsigma}$ throughout~$\Omega$.} $f\in\cf^{m-1,\varsigma-\epsilon}_0(\overline\Omega)$, $\epsilon\in(0,\varsigma)$, by Lemma~\ref{lem:extension}. This, in particular, implies that
\[
|(f_{\alpha,\beta}f)(z)|\leq\const|z^2-1|^{s-1-\upsilon-\epsilon},
\]
which consequently yields \eqref{eq:bnd2w}. 

Suppose now that \dbd~ is solvable and $\mathscr{D}$ is a solution. Let $\Gamma$ be a smooth arc encompassing $\overline\Omega$. Convolve $\mathscr{D}$ with a family of mollifiers so that the result is smooth and converges in the Sobolev sense to $\mathscr{D}$. Then by applying the Cauchy-Green formula \eqref{eq:cthnaf} to this convolution and taking the limit, we get that
\[
\mathscr{D} = \oc_\Gamma(\mathscr{D})+\ok_\Gamma(\mathscr{DW}) = \mathscr{I} + \ok_\mathscr{W}\mathscr{D}
\]
since $\mathscr{W}$ has compact support $\overline\Omega$, i.e., $\mathscr{D}$ is analytic outside of $\overline\Omega$, and $\mathscr{D}(\infty)=\mathscr{I}$, where $\ok_\mathscr{W}(\cdot)=\ok(\cdot\mathscr{W})$. Hence, every solution of \dbd~ is a solution of the following integral equation
\begin{equation}
\label{eq:ieq}
\mathscr{I} = (\oi-\ok_\mathscr{W})\mathscr{D},
\end{equation}
where $\oi$ is the identity operator. As explained in Section \ref{subsec:io}, $\oi-\ok_\mathscr{W}$ is a bounded operator from $\lf^\infty(\C)$ into itself that maps continuous functions into continuous functions preserving their value at infinity. Conversely, if $\mathscr{D}$ is a solution of \eqref{eq:ieq} in $\lf^\infty(\C^{2\times2})$ then $\mathscr{D}$ is, in fact, continuous in $\overline\C^{2\times2}$, analytic outside of $\overline\Omega$, $\mathscr{D}(\infty)=\mathscr{I}$, and $\bar\partial\mathscr{D}=\mathscr{DW}$ by \eqref{eq:okreverse} in the distributional sense. Thus, \dbd~ is equivalent to uniquely solving \eqref{eq:ieq} in $(\lf^\infty(\C))^{2\times2}$ because $\mathscr{D}-\ok_\mathscr{W}\mathscr{D}$ is holomorphic in $\C$ and is identity at infinity.

We claim that
\begin{equation}
\label{eq:claim}
\|\ok_\mathscr{W}\| \leq \frac{C_0}{n^a}, \quad
a\in\left\{
\begin{array}{ll}
\left(0,\frac{s-\upsilon}{2}\right), & s-\upsilon\leq 1, \smallskip \\
\left(0,\frac12\right), & s-\upsilon>1,
\end{array}
\right. 
\end{equation}
where $\|\cdot\|$ is the norm of $\ok_\mathscr{W}$ as an operator from $\lf^\infty(\C)$ into itself and the constant $C_0$ depends on $a$. Assuming this claim to be true, we get that $(\oi-\ok_\mathscr{W})^{-1}$ exists as a Neumann series and
\[
\mathscr{D} = \mathscr{I} + O\left(\frac{\|\ok_\mathscr{W}\|}{1-\|\ok_\mathscr{W}\|}\right),
\]
which finishes the proof of the lemma granted the validity of \eqref{eq:claim}. Thus, it only remains to prove  estimate \eqref{eq:claim}. To this end, observe that \eqref{eq:bnd1w}, \eqref{eq:bnd2w}, and the H\"older inequality imply
\begin{equation}
\label{eq:ss1}
\|\ok_\mathscr{W}\| \leq C_1 \max_{z\in\overline\Omega}\left\|\frac{r_nf_{\alpha,\beta}f}{z-\cdot}\right\|_{1,\Omega} \leq C_2 \max_{z\in\overline\Omega}\left\|\frac{r_n}{z-\cdot}\right\|_{q,\Omega}, \quad
q\in\left\{
\begin{array}{ll}
\left(\frac{2}{s-\upsilon+1},2\right), & s-\upsilon\leq 1, \smallskip \\
\left[1,2\right), & s-\upsilon>1,
\end{array}
\right. 
\end{equation}
where $C_1,C_2$ are constants depending on $q$ and $\|\cdot\|_{q,\Omega}$ is the usual norm on $\lf^q(\Omega)$. Thus, it holds that
\begin{equation}
\label{eq:ss3}
\|\ok_\mathscr{W}\| \leq C_3 \|r_n\|_{q,\Omega}, \quad q\in\left\{
\begin{array}{ll}
\left(\frac{2}{s-\upsilon},\infty\right), & s-\upsilon\leq 1, \smallskip \\
\left(2,\infty\right), & s-\upsilon>1,
\end{array}
\right.
\end{equation}
by H\"older inequality, where $C_3$ is a constant depending on $q$.

In another connection, let $\psi$ be the conformal map of $D$ onto $\D$, $\psi(\infty)=0$, $\psi^\prime(\infty)>0$, and let $b_n$ be a Blaschke product with respect to $D$ that has the same zeros as $r_n$ counting multiplicities, i.e.,
\[
b_n(z)  = \prod_{r_n(e)=0}\frac{\psi(z)-\psi(e)}{1-\overline{\psi(e)}\psi(z)}, \quad z\in D.
\]
Then by the maximum modulus principle for analytic functions and Definition~\ref{df:S}-(1), we have that
\begin{equation}
\label{eq:kwbnd1}
|r_n(z)| \leq \max_{t\in\Delta}|r_n^\pm(t)||b_n(z)| \leq C_4 |b_n(z)|, \quad z\in D,
\end{equation}
where $C_4$ is independent of $n$ and $z$. Denote by $L_\rho$, $\rho\in(0,1)$, the level line of $\psi$, i.e. $L_\rho:=\{z\in D:~|\psi(z)|=\rho\}$. Due to Definition~\ref{df:S}-(2), there exist $1>\rho_0>\rho_1>0$ such that $\overline\Omega$ is contained within the bounded domain with boundary $L_{\rho_0}$, say $\Omega_{\rho_0}$, and all the zeros of $b_n$ are contained within the unbounded domain with boundary $L_{\rho_1}$. Then
\begin{equation}
\label{eq:kwbnd2}
\|b_n\|_{q,\Omega} \leq \|b_n\|_{q,\Omega_{\rho_0}} = \left\|(b_n\circ\psi^{-1})\left((\psi^{-1})^\prime\right)^2\right\|_{q,\A_{\rho_0,1}},
\end{equation}
where $\A_{\rho_0,1}:=\{z:~\rho_0<|z|<1\}$, $\psi^{-1}$ is the inverse of $\psi$, and the index $q$ is as in \eqref{eq:ss3}. As $\psi^{-1}$ is a conformal map of $\D$ onto $D$, it holds that $|(\psi^{-1})^\prime|\leq C_5$ in $\A_{\rho_0,1}$. Set
\[
b_n^*(z):=(b_n\circ\psi^{-1})(z) = \prod_{b_n(\psi^{-1}(e^*))=0}\frac{z-e^*}{1-z\overline{e^*}}, \quad z\in\D.
\]
Then by \eqref{eq:ss3}, \eqref{eq:kwbnd1}, and \eqref{eq:kwbnd2} and after, we get that
\begin{equation}
\label{eq:ssfinal}
\|\ok_\mathscr{W}\| \leq C_5 \|b^*_n\|_{q,\A_{\rho_0,1}},
\end{equation}
where the constant $C_5$ depends on $q$. Observe now that for $|z|=\rho$, $\rho\in(\rho_0,1)$, it holds that
\begin{equation}
\label{eq:bndbn}
|b_n^*(z)| \leq \prod\frac{\rho+|e^*|}{1+\rho|e^*|} \leq \exp\left\{-(1-\rho)\sum\frac{1-|e^*|}{1+\rho|e^*|}\right\} \leq \exp\left\{-2n\frac{(1-\rho)(1-\rho_1)}{1+\rho}\right\}
\end{equation}
since $|e^*|<\rho_1$ by the definition of $\rho_1$. Clearly, \eqref{eq:ssfinal} and \eqref{eq:bndbn} yield that
\[
\|\ok_\mathscr{W}\| \leq C_5 \left(\int_0^{2\pi}\int_{\rho_0}^1\exp\left\{-2nq\frac{(1-\rho)(1-\rho_1)}{1+\rho}\right\}\rho d\rho dt\right)^{1/q} \leq C_6 n^{-1/q},
\]
which is exactly \eqref{eq:claim} with $a=1/q$.
\end{proof}

\section{Solution of \rhy~ and Proof of Theorem \ref{thm:sa}}
\label{sec:solution}

In this last section, we gather the material from Sections \ref{sec:rhpbp}--\ref{sec:pbp} to prove Theorem \ref{thm:sa}. It is an immediate consequence of Lemmas \ref{lem:rhy}, \ref{lem:rht}, and \ref{lem:rhs}, combined with Lemmas \ref{lem:rhab}, \ref{lem:rhar}, and \ref{lem:deltabar} that the following result holds.

\begin{lem}
\label{lem:solutionY}
If \eqref{eq:ab} is fulfilled, then the solution of \rhy~ uniquely exists for all $n$ large enough and can be expressed by reversing the transformations $\mathscr{Y}\to\mathscr{T}\to\mathscr{S}$ using \eqref{eq:t} and \eqref{eq:s} with $\mathscr{S}=\mathscr{D}\mathscr{A}$, where $\mathscr{A}$ is the solution of \rha~ and $\mathscr{D}$ is the solution of \dbd.
\end{lem}

\subsection{Asymptotics away from $\Delta$, formula \eqref{eq:sa1}}

We claim that \eqref{eq:sa1} holds locally uniformly in $D$. Clearly, for any given closed set in $D$, it can be easily arranged that this set lies exterior to the lens $\Sigma_n^{rd}$, and therefore to the lenses $\Sigma_n$ and $\Sigma_{ext}$. Thus, the asymptotic behavior of $\mathscr{Y}$ on this closed is given by
\[
\mathscr{Y} = \left(2^n\epsilon_n\right)^{-\sigma_3}\mathscr{RND}E_n^{\sigma_3}
\]
due to Lemma \ref{lem:solutionY}, where $\epsilon_n$ and $E_n$ were defined in \eqref{eq:en}, $\mathscr{R}$ is the solution of \rhr~ given by Lemma \ref{lem:rhar}, and $\mathscr{N}$ is the solution of \rhn~ given by \eqref{eq:n}. Moreover, we have that
\begin{equation}
\label{eq:rnd}
\mathscr{RND} = \left(\left[1+o(1)\right]\mathscr{N}_{lk}\right)_{l,k=1,2},
\end{equation}
where $o(1)$ satisfies \eqref{eq:deltane} locally uniformly in $\overline\C\setminus\{\pm1\}$, including on $(\Delta^\pm)^\circ$, on account of \eqref{eq:r} and \eqref{eq:d}. Thus, it holds that
\[
\left\{
\begin{array}{lll}
\mathscr{Y}_{11} & = [1+o(1)]\mathscr{N}_{11}E_n/(2^n\epsilon_n) & = \map^n/(2^n\szf_{w\sr^+}\szf_{(hh_n/v_n)}) \smallskip \\
\mathscr{Y}_{12} & = [1+o(1)]\mathscr{N}_{12}/(E_n2^n\epsilon_n) & = 2i\gm_{w}\gm_{(hh_n/v_n)}\szf_w\szf_{(hh_n/v_n)}/(\map^{n+1}\szf_{\sr^+})
\end{array}
\right.
\]
by \eqref{eq:en} and \eqref{eq:n}, where $o(1)$ satisfies \eqref{eq:deltane} locally uniformly in $D$. Recall now that the entries of $\mathscr{N}$ are, in fact, continued Szeg\H{o} functions defined with respect to $\Delta_n$. However, we have already mentioned that they coincide with $\szf_{\sr^+}$ and $\szf_w$ outside of a set exterior to $\Delta_n\cup\Delta$. Thus, the equations above indeed hold true. Hence, asymptotic formulae \eqref{eq:sa1} follow from \eqref{eq:y}, \eqref{eq:sngamman}, and \eqref{eq:szfsr}.

\subsection{Asymptotics in the Bulk, Formula \eqref{eq:sa2}}

To derive asymptotic behavior of $q_n$ and $R_n$ on $\Delta\setminus\{\pm1\}$, we need to consider what happens within the lens $\Sigma_{ext}$ and outside the disks $U_\delta$ and $\widetilde U_\delta$. We shall consider the asymptotics of $\mathscr{Y}$ from within $\Omega_+$, the upper part of the lens $\Sigma_{ext}$, the behavior of $\mathscr{Y}$ in $\Omega_-$ can be deduced in a similar fashion.

Recall that $\Delta_n$ either coincides with $\Delta$ or intersects it at finite number of points, as both arcs are images of $[-1,1]$ under holomorphic maps. Set
\[
\Delta_n^* := \Delta\cap\Omega_{n+} \quad \Delta_n^{**}:=\Delta\cap\Omega_{n-},
\]
where $\Omega_{n+}$ and $\Omega_{n-}$ are the upper and lower parts of the lens $\Sigma_n^{md}$. Then, it holds that
\begin{equation}
\label{eq:aplus}
\mathscr{A}_+ = \left\{
\begin{array}{ll}
\mathscr{B}_+, & \mbox{on} \quad \Delta_n^*, \smallskip \\
\mathscr{B}_-\left(\begin{array}{cc} 0 & w \\ -1/w & 0 \end{array}\right), & \mbox{on} \quad \Delta_n^{**},
\end{array}
\right.
\end{equation}
by \eqref{eq:a}, where with a slight abuse of notation we denote by $\mathscr{B}_\pm$ the values of $\mathscr{B}$ in $\Omega_{n\pm}$ and on $\Delta_n^\pm$. Then it holds on $\Delta^\circ$ by \rhn(b) and on account of Lemma \ref{lem:solutionY}, \eqref{eq:aplus}, and \eqref{eq:b} that
\[
\mathscr{S}_+ = \left\{
\begin{array}{ll}
\mathscr{R}\mathscr{N}_+\mathscr{D}, & \mbox{on} \quad \Delta_n^* \smallskip \\
\mathscr{R}\mathscr{N}_-\left(\begin{array}{cc} 0 & w \\ -1/w & 0 \end{array}\right)\mathscr{D}, & \mbox{on} \quad \Delta_n^{**}
\end{array}
\right.
= \mathscr{R}\mathscr{\widetilde N}_+\mathscr{D},
\]
where, again, under $\mathscr{N}_\pm$ we understand the values of $\mathscr{N}$ in $\Omega_{n\pm}$ and on $\Delta_n^\pm$ and $\mathscr{\widetilde N}$ is the analytic continuation of $\mathscr{N}$ that satisfies \rhn, only with a jump across $\Delta$. Clearly, $\mathscr{\widetilde N}$ is defined by \eqref{eq:nstar} and \eqref{eq:n}, where $\szf_{\sr^+}$ and $\szf_w$ are the Szeg\H{o} functions of $\sr^+$ and $w$ with respect to $\Delta$ and not the continued functions that actually appear in \eqref{eq:nstar} and \eqref{eq:n}. Thus, we deduce from Lemma \ref{lem:solutionY} and \eqref{eq:rnd} that
\begin{eqnarray}
\mathscr{Y}_+ &=& \left(2^n\epsilon_n\right)^{-\sigma_3}\left(\left[1+o(1)\right]\mathscr{\widetilde N}_{lk}^+\right)_{l,k=1,2}\left(\begin{array}{cc} 1 & 0 \\ (r_nc_nc)^+/w & 1 \end{array}\right)(E_n^+)^{\sigma_3} \nonumber \\
{} &=& \left(2^n\epsilon_n\right)^{-\sigma_3}\left(\left[1+o(1)\right]\mathscr{\widetilde N}_{lk}^+\right)_{l,k=1,2}\left(\begin{array}{cc} E_n^+ & 0 \smallskip \\ E_n^-/w & 1/E_n^+ \end{array}\right) \nonumber
\end{eqnarray}
where $o(1)$ satisfies \eqref{eq:deltane} locally uniformly on $\Delta^\circ$ and we used \eqref{eq:encn} to obtain the second equality. Therefore, it holds that
\[
\left\{
\begin{array}{ll}
\mathscr{Y}_{11} & = [1+o(1)]\mathscr{\widetilde N}_{11}^+E_n^+/(2^n\epsilon_n) + [1+O(\delta_{n,\epsilon})]\mathscr{\widetilde N}_{12}^+E_n^-/(2^n\epsilon_nw)  \smallskip \\
\mathscr{Y}_{12}^+ & = [1+o(1)]\mathscr{\widetilde N}_{12}^+/(E_n^+2^n\epsilon_n) 
\end{array}
\right.
\]
with $o(1)$ satisfying \eqref{eq:deltane} locally uniformly on $\Delta^\circ$. As in the end of the previous section, we deduce \eqref{eq:sa2} from \eqref{eq:y}, the formulae
\[
\frac{\mathscr{\widetilde N}^\pm_{11} E_n^\pm}{2^n\epsilon_n} = \frac{1}{\szf_n^\pm} \quad \mbox{and} \quad \frac{\mathscr{\widetilde N}_{12}^+}{E_n^+2^n\epsilon_n} = \frac{\szf_n^+}{\sr^+},
\]
and by noticing that
\[
\frac{1}{w}\frac{\mathscr{\widetilde N}_{12}^+}{\mathscr{\widetilde N}_{11}^-}  = \frac{1}{w} \frac{\gm_w\szf_w^+\szf_w^-\szf_{\sr^+}^-}{\map^+\szf_{\sr^+}^+} = \frac{i\szf_{\sr^+}^-}{\map^+\szf_{\sr^+}^+} \equiv 1
\]
on $\Delta^\circ$ by \eqref{eq:szego} and \eqref{eq:szfsrpm}.

\bibliographystyle{plain}
\bibliography{jp}

\begin{thebibliography}{10}

\bibitem{AbramowitzStegun}
M.~Abramowitz and I.A. Stegun.
\newblock {\em Handbook of Mathematical Functions}.
\newblock Dover Publications, Inc., New York, 1968.

\bibitem{Adams}
R.A. Adams.
\newblock {\em Sobolev Spaces}, volume~65 of {\em Pure and Applied
  Mathematics}.
\newblock Academic Press, Inc., 1975.

\bibitem{Ap02}
A.I. Aptekarev.
\newblock Sharp constant for rational approximation of analytic functions.
\newblock {\em Mat. Sb.}, 193(1):1--72, 2002.
\newblock English transl. in {\it {M}ath. {S}b.} 193(1-2):1--72, 2002.

\bibitem{ArmEd70}
R.J. Arms and A.~Edrei.
\newblock {\em The Pad\'e tables of continued fractions generated by totally
  positive sequences}, pages 1--21.
\newblock Mathematical Essays dedicated to A.J. Macintyre. Ohio Univ. Press,
  Athens, Ohio, 1970.

\bibitem{AstalaIwaniecMartin}
K.~Astala, T.~Iwaniec, and G.~Martin.
\newblock {\em Elliptic Partial Differential Equations and Quasiconformal
  Mappings in the Plane}, volume~48 of {\em Princeton Mathematical Series}.
\newblock Princeton Univ. Press, 2009.

\bibitem{AvFaRe07}
P.~Avery, C.~Farhat, and G.~Reese.
\newblock Fast frequency sweep computations using a multipoint {P}ad\'e based
  reconstruction method and an efficient iterative solver.
\newblock {\em International J. for Num. Methods in Engin.}, 69(13):2848--2875,
  2007.

\bibitem{BakerGravesMorris}
G.A. Baker and P.~Graves-Morris.
\newblock {\em Pad\'e Approximants}, volume~59 of {\em Encyclopedia of
  Mathematics and its Applications}.
\newblock Cambridge University Press, 1996.

\bibitem{Baker}
G.A. Baker{, Jr.}
\newblock {\em Quantitative theory of critical phenomena}.
\newblock Academic Press, Boston, 1990.

\bibitem{uBY3}
L.~Baratchart and M.~Yattselev.
\newblock Convergent interpolation to {C}auchy integrals over analytic arcs.
\newblock {\it To appear in Found. Comput. Math.},
  http://arxiv.org/abs/0812.3919.

\bibitem{uBY1}
L.~Baratchart and M.~Yattselev.
\newblock Critical arcs for multipoint {P}ad\'e interpolation.
\newblock {\it In preparation}.

\bibitem{BlIt99}
P.~Bleher and A.~Its.
\newblock Semiclassical asymptotics of orthogonal polynomials,
  {R}iemann-{H}ilbert problem, and the universality in the matrix model.
\newblock {\em Ann. Math.}, 50:185--266, 1999.

\bibitem{Brezinski}
C.~Brezinski.
\newblock {\em Computational aspects of linear control}.
\newblock Kluwer, Dordrecht, 2002.

\bibitem{BrezinskiRedivoZaglia}
C.~Brezinski and M.~Redivo-Zaglia.
\newblock {\em Extrapolation methods: Theory and Practice}.
\newblock North-Holland, Amsterdam, 1991.

\bibitem{BrezR-Z06}
C.~Brezinski and M.~Redivo-Zaglia.
\newblock The pagerank vector: properties, computation, approximation, and
  acceleration.
\newblock {\em SIAM J. Matrix Anal. Appl.}, 28(2):551--575, 2006.

\bibitem{Bus02}
V.I. Buslaev.
\newblock On the {B}aker-{G}ammel-{W}ills conjecture in the theory of {P}ad\'e
  approximants.
\newblock {\em Mat Sb.}, 193(6):25--38, 2002.

\bibitem{But64}
J.C. Butcher.
\newblock Implicit {R}unge-{K}utta processes.
\newblock {\em Math. Comp.}, 18:50--64, 1964.

\bibitem{CelOcTanAt95}
M.~Celik, O.~Ocali, M.~A. Tan, and A.~Atalar.
\newblock Pole-zero computation in microwave circuits using multipoint {P}ad\'e
  approximation.
\newblock {\em IEEE Trans. Circuits and Syst.}, 42(1):6--13, 1995.

\bibitem{Deift}
P.~Deift.
\newblock {\em Orthogonal Polynomials and Random Matrices: a Riemann-Hilbert
  Approach}, volume~3 of {\em Courant Lectures in Mathematics}.
\newblock Amer. Math. Soc., Providence, RI, 2000.

\bibitem{DKMLVZ99a}
P.~Deift, T.~Kriecherbauer, K.T.-R. McLaughlin, S.~Venakides, and X.~Zhou.
\newblock Strong asymptotics for polynomials orthogonal with respect to varying
  exponential weights.
\newblock {\em Comm. Pure Appl. Math.}, 52(12):1491--1552, 1999.

\bibitem{DruMos01}
V.~Druskin and S.~Moskow.
\newblock Three point finite difference schemes, {P}ad\'e and the spectral
  {G}alerkin method. {I.} {O}ne sided impedance approximation.
\newblock {\em Mathematics of Computation}, 71(239):995--1019, 2001.

\bibitem{DudSp00}
R.V. Duduchava and F.-O. Speck.
\newblock Singilar integral equations in special weighted spaces.
\newblock {\em Georgian Math. J.}, 7(4):633--642, 2000.

\bibitem{FIK91}
A.S. Fokas, A.R. Its, and A.V. Kitaev.
\newblock Discrete {P}anlev\'e equations and their appearance in quantum
  gravity.
\newblock {\em Comm. Math. Phys.}, 142(2):313--344, 1991.

\bibitem{FIK92}
A.S. Fokas, A.R. Its, and A.V. Kitaev.
\newblock The isomonodromy approach to matrix models in {2D} quantum
  gravitation.
\newblock {\em Comm. Math. Phys.}, 147(2):395--430, 1992.

\bibitem{Gakhov}
F.D. Gakhov.
\newblock {\em Boundary Value Problems}.
\newblock Dover Publications, Inc., New York, 1990.

\bibitem{GL78}
A.A. Gonchar and G.~L\'opez Lagomasino.
\newblock On {M}arkov's theorem for multipoint {P}ad\'e approximants.
\newblock {\em Mat. Sb.}, 105(4):512--524, 1978.
\newblock English transl. in {\it{M}ath. {USSR} {S}b.} 34(4):449--459, 1978.

\bibitem{GonNovHen91}
A.A. Gonchar, N.N. Novikova, and G.M. Henkin.
\newblock Multipoint {P}ad\'e approximants in the inverse {S}turm-{L}iouville
  problem.
\newblock {\em Mat. Sb.}, 182(8):1118--1128, 1991.

\bibitem{GRakh87}
A.A. Gonchar and E.A. Rakhmanov.
\newblock Equilibrium distributions and the degree of rational approximation of
  analytic functions.
\newblock {\em Mat. Sb.}, 134(176)(3):306--352, 1987.
\newblock English transl. in {\it {M}ath. {USSR} Sbornik} 62(2):305--348, 1989.

\bibitem{Grafakos}
L.~Grafakos.
\newblock {\em Classical and Modern Fourier Analysis}.
\newblock Pearson Education, Inc., Upper Saddle River, New Jersey 07458, 2004.

\bibitem{Gr65}
W.B. Gragg.
\newblock On extrapolation algorithms for ordinary initial value problems.
\newblock {\em SIAM J. Num. Anal.}, 2:384--403, 1965.

\bibitem{Grisvard}
P.~Grisvard.
\newblock {\em Elliptic Problems in Nonsmooth Domains}.
\newblock Pitman Publishing Inc, 1985.

\bibitem{Horiguchi}
K.~Horiguchi.
\newblock {\em Linear circuits, systems and signal processing: advanced theory
  and applications}, chapter~4.
\newblock Marcel Dekker, 1990.

\bibitem{KratRiv08}
C.~Krattenthaler and T.~Rivoal.
\newblock Approximants de {P}ad\'e des $q$-polylogarithmes.
\newblock {\em Dev. Math.}, 16:221--230, 2008.

\bibitem{KMcLVAV04}
A.B. Kuijlaars, K.T.-R. McLaughlin, W.~Van Assche, and M.~Vanlessen.
\newblock The {R}iemann-{H}ilbert approach to strong asymptotics for orthogonal
  polynomials on $[-1,1]$.
\newblock {\em Adv. Math.}, 188(2):337--398, 2004.

\bibitem{Lub85}
D.S. Lubinsky.
\newblock Pad\'e tables of entire functions of very slow and smooth growth.
\newblock {\em Constr. Approx.}, 1:349--358, 1985.

\bibitem{Lub03}
D.S. Lubinsky.
\newblock {R}ogers-{R}amanujan and the {B}aker-{G}ammel-{W}ills ({P}ad\'e)
  conjecture.
\newblock {\em Ann. of Math.}, 157(3):847--889, 2003.

\bibitem{Mar95}
A.A. Markov.
\newblock Deux d\'emonstrations de la convergence de certaines fractions
  continues.
\newblock {\em Acta Math.}, 19:93--104, 1895.

\bibitem{McLM08}
K.T.-R. McLaughlin and P.D. Miller.
\newblock The $\bar\partial$ steepest descent method for orthogonal polynomials
  on the real line with varying weights.
\newblock {\em Int. Math. Res. Not. IMRN}, 2008:66 pages, 2008.

\bibitem{Pade92}
H.~Pad\'e.
\newblock Sur la repr\'esentation approch\'ee d'une fonction par des fractions
  rationnelles.
\newblock {\em Ann. Sci Ecole Norm. Sup.}, 9(3):3--93, 1892.

\bibitem{Pommerenke}
Ch. Pommerenke.
\newblock {\em Boundary Behavior of Conformal Maps}, volume 299 of {\em
  Grundlehren der Math. Wissenschaften}.
\newblock Springer-Verlag, Berlin, 1992.

\bibitem{Pozzi}
A.~Pozzi.
\newblock {\em Applications of Pad\'e approximation in fluid dynamics},
  volume~14 of {\em Advances in Maths. for Applied Sci.}
\newblock World Scientific, 1994.

\bibitem{RivZud03}
T.~Rivoal and W.~Zudilin.
\newblock Diophantine properties of numbers related to {C}atalan's constant.
\newblock {\em Math. Ann.}, 326(4):241--251,705--721, 2003.

\bibitem{SarkarBhatt92}
B.~Sarkar and K.~Bhattacharyya.
\newblock Accurate evaluation of lattice constants using the multipoint
  {P}ad\'e approximant technique.
\newblock {\em Physical review B}, 45(9):4594--4599, 1992.

\bibitem{Siegel}
C.L. Siegel.
\newblock {\em Transcendental Numbers}.
\newblock Princeton Univ. Press, 1949.

\bibitem{St85}
H.~Stahl.
\newblock Extremal domains associated with an analytic function. {I, II}.
\newblock {\em Complex Variables Theory Appl.}, 4:311--324, 325--338, 1985.

\bibitem{St85b}
H.~Stahl.
\newblock Structure of extremal domains associated with an analytic function.
\newblock {\em Complex Variables Theory Appl.}, 4:339--356, 1985.

\bibitem{St86}
H.~Stahl.
\newblock Orthogonal polynomials with complex valued weight function. {I, II}.
\newblock {\em Constr. Approx.}, 2(3):225--240,241--251, 1986.

\bibitem{St89}
H.~Stahl.
\newblock On the convergence of generalized {P}ad\'e approximants.
\newblock {\em Constr. Approx.}, 5(2):221--240, 1989.

\bibitem{St97}
H.~Stahl.
\newblock The convergence of {P}ad\'e approximants to functions with branch
  points.
\newblock {\em J. Approx. Theory}, 91:139--204, 1997.

\bibitem{StahlTotik}
H.~Stahl and V.~Totik.
\newblock {\em General Orthogonal Polynomials}, volume~43 of {\em Encycl.
  Math.}
\newblock Cambridge University Press, Cambridge, 1992.

\bibitem{Suet00}
S.P. Suetin.
\newblock Uniform convergence of {P}ad\'e diagonal approximants for
  hyperelliptic functions.
\newblock {\em Mat. Sb.}, 191(9):81--114, 2000.
\newblock English transl. in {\it {M}ath. {S}b.} 191(9):1339--1373, 2000.

\bibitem{TeToga01}
J.J. Telega, S.~Tokarzewski, and A.~Galka.
\newblock {\em Numerical analysis and its applications}, chapter Modelling
  torsional properties of human bones by multipoint {P}ad\'e approximants,
  pages 33--38.
\newblock Lect. Notes in Comp. Sci. Springer, 2001.

\bibitem{Tj_PRA77}
J.A. Tjon.
\newblock Operator {P}ad\'e approximants and three body scattering.
\newblock In E.B. Saff and R.S. Varga, editors, {\em Pad\'e and Rational
  Approximation}, pages 389--396, 1977.

\bibitem{Y10}
M.~Yattselev.
\newblock On uniform approximation of rational perturbations of {C}auchy
  integrals.
\newblock {\em Comput. Methods Funct. Theory}, 10(1):1--33, 2010.

\end{thebibliography}

\end{document}